\documentclass[a4paper]{article}
\usepackage{geometry}
\usepackage[utf8]{inputenc}
\usepackage{amssymb}

\usepackage{bm}
\usepackage[normalem]{ulem}
\usepackage[justification=centering]{caption}
\usepackage{subcaption}
\usepackage{optidef}

\usepackage{stanli}
\usepackage{adjustbox}
\usepackage{graphicx}
\usepackage{multirow}
\usepackage{multicol}
\usepackage{hhline}
\usepackage{adjustbox}
\usepackage{amssymb}
\usepackage{amsthm}
\usepackage{thmtools}
\usepackage{nicematrix}

\usepackage[mathscr]{euscript}

\newcommand{\stkout}[1]{\ifmmode\text{\sout{\ensuremath{#1}}}\else\sout{#1}\fi}

\newtheorem{proposition}{Proposition}
\newtheorem{lemma}{Lemma}
\newtheorem{remark}{Remark}
\newtheorem{example}{Example}
\newtheorem{definition}{Definition}
\newtheorem{theorem}{Theorem}

\DeclareMathOperator*{\argmin}{arg~min}

\newtheorem{assumption}{Assumption}

 \usepackage{xcolor}
\usepackage{hyperref}
  \hypersetup{
  colorlinks,
  citecolor=blue,
  linkcolor=blue,
  urlcolor=blue}

\newcommand*\circled[1]{\tikz[baseline=(char.base)]{
		\node[shape=circle,draw,inner sep=0pt,fill=white, minimum size=4mm] (char) {#1};}}
\newcommand*\squared[1]{\tikz[baseline=(char.base)]{
		\node[shape=rectangle,draw,inner sep=0pt, minimum size=4mm] (char) {#1};}}

\usepackage{soul}

\usepackage[
    backend=biber,
    style=numeric,
  ]{biblatex}
\bibliography{TSP.bib}
\usepackage{authblk}

\title{Term-sparse polynomial optimization for the design of frame structures  }

\author[1]{Marouan Handa}
\author[1,2]{Marek Tyburec}
\author[1,3]{Michal Ko\v{c}vara}

\affil[1]{Institute of Information Theory and Automation, 18208 Prague, Czech Republic}
\affil[2]{Faculty of Civil Engineering, Czech Technical University in Prague, Thakurova 7, 16629 Prague 6}
\affil[3]{School of Mathematics, University of Birmingham, Birmingham B15 2TT, UK}

\date{}

\bibliography{TSP.bib}

\begin{document}

\maketitle

\begin{abstract}
This work investigates an efficient solution to two fundamental problems in topology optimization of frame structures. The first one involves minimizing structural compliance under linear-elastic equilibrium and weight constraint, while the second one minimizes the weight under compliance constraints. These problems are non-convex and generally challenging to solve globally, with the non-convexity concentrated in a polynomial matrix inequality. In \cite{tyburec2021global,tyburec2022global}, the authors tackled the problems using the moment-sum-of-squares hierarchy (mSOS), but were only able to solve smaller instances globally. Here, we aim to improve the scalability of solution to these problems by using the mSOS hierarchy supplemented with the Term Sparsity Pattern technique (TSP), which was introduced by Magron and Wang \cite{magron2023sparse}. Due to the unique polynomial structure of our problems in which the objective and constraint functions are separable polynomials, we further improve scalability by adopting a reduced monomial basis containing non-mixed terms only. From extensive numerical testing, we conclude that these techniques allow for a global solution to two times larger instances when compared to \cite{tyburec2021global,tyburec2022global}, and accelerate the solution of the problems from \cite{tyburec2021global,tyburec2022global} significantly.

pattern $\cdot$ Nonmixed basis $\cdot$ Global optimality

\end{abstract}

\section{Introduction}\label{sec:intro}
Topology optimization of frame structures is a classical problem in structural design. The nature of the problem lies between the discrete setting of truss topology optimization and the polynomial nature of stiffness interpolations in continuum topology optimization \cite{bendsoe2003topology}. It consists of finding the continuous cross-section areas that minimize compliance or weight for a given finite element discretization and boundary conditions, under the linear-elastic equilibrium and other physical constraints. In this context, the optimization problem can be formulated as a minimization of a linear function over a basic semi-algebraic feasible set, which is described using a polynomial matrix inequality. This makes the resulting optimization problems non-convex and, therefore, challenging to solve globally \cite{yamada2016relaxation,kanno2016mixed,toragay2022exact}. In the literature, mainly local solvers \cite{fredricson2003topology,chan1995automatic} and meta-heuristic methods \cite{an2017topology} were used, providing local solutions or solutions that cannot be certified as approximately global. 

Few works have dealt with the problem globally. Toragay \textit{et al.}~\cite{toragay2022exact} formulated the weight minimization problem as a mixed integer quadratically-constrained program and solved it by branch-and-bound-type method. In their work, they considered degree-two polynomials only, but included constraints to prevent virtually intersecting frame elements, semi-continuous constraints for the cross-section areas and additional cut constraints to reduce the feasible set. In \cite{tyburec2021global,tyburec2022global}, the compliance and weight minimization problems were formulated as polynomial optimization problems (POP) and solved by Lasserre's moment-sum-of-square (mSOS) hierarchy \cite{lasserre2001global,henrion2006convergent}. 

mSOS is a powerful tool for solving POPs by converting them to a hierarchy of increasing-sized semidefinite programming (SDP) relaxations of a given POP. The hierarchy is indexed by an integer called the relaxation order. With increasing relaxation order, the hierarchy provides a monotonic sequence of lower bounds and converges to the global minimum in the limit. Convergence occurs generically in a finite number of steps \cite{Nie2013}, and global optimality can be recognized from the satisfaction of the rank flatness condition \cite{curto1996solution}. A linear procedure is then used to extract all the minimizers (see \cite{henrion2005detecting}). In contrast to \cite{toragay2022exact}, the formulations in \cite{tyburec2021global,tyburec2022global} naturally accommodate higher-degree polynomials and do not rely on enumerative principles. In addition, they provide simple sufficient conditions for approximate global optimality, which can be used as a complement to the rank condition in the standard mSOS method.

 Nevertheless, using the mSOS hierarchy can rapidly become challenging when the number of variables in the POP is large and/or if a high relaxation order is required to certify the global optimality. To overcome this issue, many techniques were developed to exploit the polynomial data of problems to reduce the size of the SDP matrices while maintaining the convergence properties. These works include exploiting the symmetry in the problem \cite{riener2013exploiting,lofberg2009pre}, using other positivity certificates \cite{ahmadi2014dsos,lasserre2017bounded} and by exploiting the constant trace property of semidefinite relaxations \cite{yurtsever2021scalable,mai2023hierarchy}. 

 Another way of improving performance is to exploit the link between the variables involved in the polynomials defining the objective function and the constraints, which is called the correlative sparsity pattern (CSP) \cite{lasserre2006convergent,waki2006sums}. It consists of defining a graph in which the nodes represent the variables of the POP, and an edge is added between two nodes if the corresponding two variables are involved in the expression of the objective function or in one of the constraints. According to the maximal cliques of the CSP graph, partitions of the variables are defined, leading to a decomposition using quasi block-diagonal SDP matrices. In addition, the CSP technique conserves the theoretical convergence and also the extraction procedure of the global minimizers \cite[Chapter $3$]{magron2023sparse}. Cheaper alternatives based on CSP, with better complexity and without theoretical convergence, have been used \cite{franc2023minimal,henrion2023occupation}. However, the CSP technique is not suitable for topology optimization problems that we consider in this work because one constraint involves all (resp. almost all) the variables of the weight (resp. compliance) problem.

The term sparsity pattern (TSP) \cite{magron2023sparse} approach aims to exploit sparsity from the perspective of monomials. As for CSP, the sparsity pattern is also represented by a graph. Its nodes are elements of the monomial basis, and an edge is added between two nodes if the product of the corresponding monomials appears in the objective function or any of the constraints. A sequence of embedding graphs is then constructed from the TSP graph by two successive operations: the support extension and the chordal extension. This sequence is indexed by an integer called sparsity order. The maximal cliques of these graphs enable decomposing the SDP matrices into smaller ones and lead to savings in computations. The choice of the chordal extension operation determines the complexity and the convergence properties: minimal chordal extension leads to smaller maximal cliques but no guarantee on theoretical convergence; the use of maximal chordal extension (block completion) guarantees the convergence but leads to larger maximal cliques. However, to the best of our knowledge, there is no general procedure for extracting minimizers while using TSP.

\subsection{Contribution and the organization of the paper}
The aim of this contribution is to certify global optimality in topology optimization of frame structure problems by exploiting term sparsity. In particular, we
\begin{itemize}
\item adapt the TSP method to handle polynomial matrix inequalities by looking at the PMI as matrix-linear combination of the monomial basis;
\item use a reduced (non-mixed) monomial basis that matches the separable nature of polynomials in the frame structure problems. This basis is used in two ways: for the construction of smaller-sized dense mSOS or combined with TSP. 
\end{itemize}

 This paper is presented in a way that ensures accessibility also for readers in the field of structural engineering. It is organized as follows. We start by presenting the polynomial formulation of frame structure problems in Section \ref{sec:frame_structures}, and Section \ref{sec:lassere_mSOS} provides an outline of the Lasserre mSOS hierarchy. Moreover, we also recall stopping criteria for global $\varepsilon-$optimality in the context of frame structure problems. In Section \ref{sec:mainTSP}, we explain the technical steps to apply the term sparsity pattern technique and highlight, through examples, the term sparsity occurring in frame structure problems. We end this section by discussing the non-mixed monomial basis. Section \ref{sec:numerical-experience} provides numerical experiments by comparing different term sparsity strategies. Finally, we summarize our contributions and suggest related future work in Section \ref{sec:conclusion-perspective}.  
\section{Background: topology optimization of frame structures}
\label{sec:frame_structures}
This section summarizes the main results on the design of frame structures using polynomial optimization \cite{tyburec2021global,tyburec2022global}. In particular, here we consider (i) minimization of structural compliance under a weight constraint and (ii) minimization of weight under a compliance constraint.
 
\subsection{Formulation}

Consider a fixed finite element discretization consisting of $n_\mathrm{e}$ prismatic Euler-Bernoulli elements. In a topology optimization problem, we search for continuous cross-sectional areas $\bm{a} \in \mathbb{R}^{n_\mathrm{e}}_{\ge 0}$ of individual finite elements that provide the most efficient design with respect to the compliance, $c(\bm{a})$, and weight, $w(\bm{a})$, functions.

For simplicity, here we review a single-load scenario and design-independent loads. We refer the reader to \cite{tyburec2021global,tyburec2022global} for an extension to other cases. Therefore, we define the compliance function as 

\begin{equation}
    c(\bm{a}) := \bm{f}^\mathrm{T} \bm{u}, \text{ where } \bm{K} (\bm{a}) \bm{u} = \bm{f},
\end{equation}

where $\bm{f} \in \mathbb{R}^{n_{\text{dof}}}$ and $\bm{u} \in \mathbb{R}^{n_{\text{dof}}}$ denote the force and displacement column vectors, $ n_{\text{dof}} \in \mathbb{N}$ stands for the associated number of degrees of freedom, and $\bm{K}(\bm{a}) \in \mathbb{R}^{n_{\text{dof}}\times n_{\text{dof}}}$ is the corresponding symmetric positive semidefinite stiffness matrix assembled as

\begin{equation}
    \bm{K} (\bm{a})=\bm{K}_{0}+\sum_{e=1}^{n_e} \left[\bm{K}_{e}^{(1)}a_e+\bm{K}_{e}^{(2)}a_e^2+\bm{K}_{e}^{(3)}a_e^3 \right].
     \label{stifness_matrix}
 \end{equation}

In \eqref{stifness_matrix}, the matrix $ \bm{K}_{0} \succeq 0$, with ``$\succeq 0$'' denoting semidefiniteness of $\bm{K}_0$, collects the design-independent stiffnesses belonging, e.g., to the non-optimized structural elements or springs, $\bm{K}_e^{(1)}$ stands for the membrane stiffness of element $e$, and the matrices $\bm{K}_e^{(2)}$ and $\bm{K}_e^{(3)}$ are associated with bending. Furthermore, by the construction of the stiffness matrices, we have $\forall i \in \lbrace 1, 2, 3 \rbrace : \bm{K}_{e}^{(i)} \succeq 0$. Thus, while \eqref{stifness_matrix} is a polynomial function of $\bm{a}$, it is a linear function of monomials $a_e^i$. For the assembled stiffness matrix $\bm{K}(\bm{a})$, we also suppose that $\forall \bm{a} >\bm{0}: \bm{K}(\bm{a}) \succ 0$, with ``$\succ 0$'' denoting positive definiteness, to exclude the undesired cases of rigid body motions.

Finding the minimum-compliant structure for a maximum weight $\overline{w} \in \mathbb{R}_{> 0}$ is naturally formulated as
\begin{minie}|s|%
{\strut \bm{a}, \bm{u}}%
 { \bm{f}^T\bm{u}\label{compliance_obj}}%
{\label{compliance_problem}}%
{}
\addConstraint{\bm{K}(\bm{a})\bm{u}-\bm{f}}{= \bm{0}\label{compliance_const1}}
\addConstraint{ \overline{w}-\sum_{e=1}^{n_\mathrm{e}}\ell_e \rho_e a_e}{\geq  0 \label{compliance_const2}}
\addConstraint{\bm{a}}{\geq \bm{0}, \label{compliance_const3}}
\end{minie}
with the weight $w(\bm{a}) = \sum_{e=1}^{n_\mathrm{e}} \ell_e \rho_e a_e$ computed from the element lengths $\bm{\ell} \in \mathbb{R}^{n_\mathrm{e}}_{>0}$ and their densities $\bm{\rho} \in \mathbb{R}^{n_\mathrm{e}}_{>0}$.

Similarly, the minimum-weight design for a given maximum compliance $\overline{c} \in \mathbb{R}_{>0}$ is stated as
\begin{minie}|s|%
{\strut \bm{a}, \bm{u}}%
 { \sum_{e=1}^{n_\mathrm{e}}\ell_e \rho_e a_e\label{weight_obj}}%
{\label{weight_problem}}%
{}
\addConstraint{\bm{K}(\bm{a})\bm{u}-\bm{f}}{= \bm{0}\label{weight_const1}}
\addConstraint{ \overline{c}-\bm{f}^T\bm{u}}{\geq  0 \label{weight_const2}}
\addConstraint{\bm{a}}{\geq \bm{0}. \label{weight_const3}}
\end{minie}

In \cite{Achtziger2008}, it has been shown that state variables $\bm{u}$ can be eliminated from the formulations, leading to equivalent semidefinite programs
\begin{mini!}|s|%
{\strut \bm{a}, c}%
 {  c \label{compliance_obj_sdp}}%
{\label{compliance_problem_sdp}}%
{}
\addConstraint{\begin{pmatrix}
    c & -\bm{f}^T \\ -\bm{f} & \bm{K}(\bm{a})
\end{pmatrix}}{\succeq 0 \label{compliance_const1_sdp}}
\addConstraint{\sum_{e=1}^{n_\mathrm{e}} \rho_e \ell_e a_e}{\le  \overline{w} \label{compliance_const2_sdp}}
\addConstraint{a_e}{\geq 0, ~\forall e \in \lbrace 1, \ldots, n_{e}\rbrace\label{compliance_const3_sdp}}
\end{mini!}
and 
\begin{mini!}|s|%
{\strut \bm{a}}%
 { \sum_{e=1}^{n_\mathrm{e}} \rho_e \ell_e a_e \label{weight_obj_sdp}}%
{\label{weight_problem_sdp}}%
{}
\addConstraint{\begin{pmatrix}
    \overline{c}& -\bm{f}^T \\ -\bm{f} & \bm{K}(\bm{a})
\end{pmatrix}}{\succeq 0\label{weight_const1_sdp}}
\addConstraint{a_e}{\geq 0, ~\forall e \in \lbrace 1, \ldots, n_{e}\rbrace,\label{weight_const2_sdp}}
\end{mini!}
in which the matrix inequalities \eqref{compliance_const1_sdp} and \eqref{weight_const1_sdp} enforce the equilibrium.

Without loss of generality, we further consider that optimal compliance $c^*$ can be bounded as $0 \le c^* \le \overline{c}$, where $\overline{c}$ is a constant. While the lower bound follows from the problem physics, the upper bound is problem-specific. In the case of weight minimization under a compliance constraint, $\overline{c}$ is already predefined in \eqref{weight_const2_sdp}. However, for compliance minimization, $\overline{c}$ can still be obtained from any statically-admissible design satisfying the resource constraint \eqref{compliance_const2_sdp}, that is, $\overline{c} = \bm{f}^T \bm{K}(\hat{\bm{a}})^\dagger \bm{f}$ for any $\hat{\bm{a}} \in \left\{{\bm{a}} \;\vert\; \bm{f} \in \mathrm{Im}(\bm{K}({\bm{a}})), w({\bm{a}}) \le \overline{w} \right\}$ in which $\bullet^\dagger$ denotes the Moore-Penrose pseudoinverse of $\bullet$; see \cite{tyburec2021global} for more details. As an example, we may take $\hat{\bm{a}} = \bm{1} \overline{w}/\bm{\ell}^T \bm{\rho}$.

The second function controlling the structural design is its weight $w(\bm{a}) := \sum_{e=1}^{n_\mathrm{e}} \rho_e \ell_e a_e$, calculated from the contributions of the element densities $\rho_e$ and their lengths $\ell_e$. Similarly to the compliance function, we also require the weight function to be bounded. In particular, for the optimum cross-section areas $\bm{a}^*$ it must hold that $0 \le w(\bm{a}^*) \le \overline{w}$, where $\overline{w}$ is a constant value. Indeed, for the case of compliance minimization, $\overline{w}$ follows directly from a resource constraint \eqref{compliance_const2_sdp}, while in the case of weight minimization, $\overline{w}$ can be computed based on a statically admissible design that satisfies the compliance constraint $\hat{\bm{a}} \in \left\{{\bm{a}} \;\vert\; \bm{f} \in \mathrm{Im}(\bm{K}({\bm{a}})),  c({\bm{a}})\le \overline{c} \right\}$. A suitable value of $\hat{\bm{a}}$ is, for example, $\hat{\bm{a}} = \delta\bm{1}$, with minimal $\delta$ satisfying the compliance constraint \eqref{weight_const2_sdp} found by bisection. Again, we refer the interested reader to \cite{tyburec2022global} for more details.

Then, the space of the cross-section areas $\bm{a}$ and of compliance $c$ can be bounded by
\begin{subequations}
\begin{align}
     \forall e \in \{1,\dots, n_\mathrm{e}\}: 0 \le a_e \le \frac{\overline{w}}{\ell_e \rho_e} &\longleftrightarrow a_e \left(\frac{\overline{w}}{\rho_e \ell_e} - a_e\right) \ge 0,\\
     0 \le c \le \overline{c} &\longleftrightarrow c \left(\overline{c} - c\right) \ge 0.
\end{align}
\end{subequations}
Using the substitutions
\begin{subequations}
    \begin{align}
        a_e &= \frac{\overline{w}}{2 \rho_e \ell_e} (a_{\mathrm{s},e} + 1),\\
        c &= 0.5 \overline{c} (c_{\mathrm{s}} + 1),
    \end{align}
\end{subequations}
to secure $[-1,1]$ domains of the design variables, the compliance optimization problem is formalized as \cite{tyburec2021global} 
\begin{mini!}|s|%
{\strut \bm{a}_s, c_s}%
 {  0.5\overline{c}(c_{s} +1) \label{compliance_obj_sdp_scalled}}%
{\label{compliance_problem_sdp_scalled}}%
{}
\addConstraint{\begin{pmatrix}
    0.5( c_{s}+1 )\overline{c}& -\bm{f}^T \\ -\bm{f} & \bm{K}(\bm{a}_s)
\end{pmatrix}}{\succeq 0 \label{compliance_const1_sdp_scalled}}
\addConstraint{ 2-n_e-\bm{1}^T\bm{a}_s}{\geq  0 \label{compliance_const2_sdp_scalled}}
\addConstraint{1-a_{s,e}^2}{\geq 0, ~\forall e \in \lbrace 1, \ldots, n_{e}\rbrace\label{compliance_const3_sdp_scalled}}
\addConstraint{1-c_{s}^2}{\geq 0, \label{compliance_const4_sdp_scalled}}
\end{mini!}
in which we slightly abuse the notation and write $\bm{K}(\bm{a})$ in terms of $\bm{a}_\mathrm{s}$ as $\bm{K}(\bm{a}_\mathrm{s})$. Similarly, the weight minimization problem reads as \cite{tyburec2022global}
\begin{mini!}|s|%
{\strut \bm{a}_s}%
 { 0.5 \overline{w} \left(n_\mathrm{e} + \sum_{e=1}^{n_\mathrm{e}} a_{\mathrm{s},e}\right) \label{weight_obj_sdp_scalled}}%
{\label{weight_problem_sdp_scalled}}%
{}
\addConstraint{\begin{pmatrix}
    \overline{c}& -\bm{f}^T \\ -\bm{f} & \bm{K}(\bm{a}_s)
\end{pmatrix}}{\succeq 0\label{weight_const1_sdp_scalled}}
\addConstraint{1-a_{s,e}^2}{\geq 0, ~\forall e \in \lbrace 1, \ldots, n_{e}\rbrace.\label{weight_const2_sdp_scalled}}
\end{mini!}

\subsection{Illustration problem}
We conclude this section by introducing a small topology optimization problem. Its aim is to demonstrate the reformulation procedure and to provide a basis for illuminating the techniques presented in subsequent sections of this article.

The illustration problem contains three Euler-Bernoulli frame elements and four nodes; see Fig.~\ref{fig:illustration}. Three of the nodes are considered clamped, while the remaining node is hinged and under a moment load of unitary magnitude. Thus, this problem has a single degree of freedom.

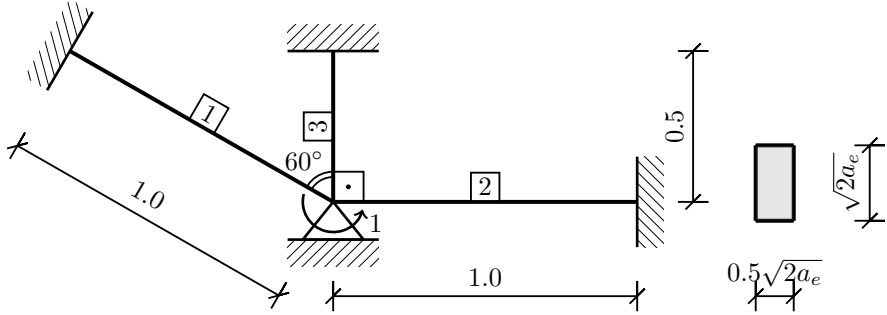
\begin{figure}[!htbp]
\centering
\begin{tikzpicture}
    \scaling{4}
    \point{a}{0}{0.5};
    \point{a0}{-0.28}{0};
    \point{b}{0.866025404}{0};
    \point{c}{1.866025404}{0};
    \point{d}{0.866025404}{0.5};
    \beam{2}{a}{b};
    \notation{4}{a}{b}[$1$];
    \beam{2}{b}{c};
    \notation{4}{b}{c}[$2$];
    \beam{2}{b}{d};
    \notation{4}{b}{d}[$3$];
    \support{3}{a}[240];
    \support{3}{c}[90];
    \support{3}{d}[180];
    \support{1}{b}[0];
    \load{2}{b}[-15][-180];
    \notation{1}{b}{$1$}[below right=0.5mm and 3.5mm];
    \addon{2}{b}{d}{c}[-1];
    \addon{3}{b}{a}{d}[-1];
    \notation{1}{b}{$60^\circ$}[above left=3mm and 0mm];
    \dimensioning{1}{b}{c}{-1.25}[$1.0$];
    \dimensioning{2}{b}{d}{8.2}[$0.5$];
    \begin{scope}[rotate=-30]
    \dimensioning{1}{a0}{b}{0.3}[\rotatebox{-30}{$1.0$}]
    \end{scope}
\end{tikzpicture}
\begin{tikzpicture}
    \point{a}{0}{0};
    \point{b}{0.5}{0};
    \point{c}{0.5}{1};
    \point{d}{0}{1};
    \beam{2}{a}{b};
    \beam{2}{b}{c};
    \beam{2}{c}{d};
    \beam{2}{d}{a};
    \draw[fill=gray, fill opacity=0.2] (a) -- (b) -- (c) -- (d) -- cycle;
    \dimensioning{1}{a}{b}{-1.0}[$0.5\sqrt{2a_e}$];
    \dimensioning{2}{b}{c}{1.5}[$\sqrt{2a_e}$];
\end{tikzpicture}
\caption{Illustrative problem: boundary conditions and cross-section parametrization.}
\label{fig:illustration}
\end{figure}

The structure is made of a linear elastic material of a dimensionless Young modulus $E=1$ and a density $\rho_e=1$ for all elements $e$. We assume rectangular cross-section areas $a_e$ for all elements $e$, with the height-to-width ratio fixed to $2$. The goal of optimization is to find the lightest structure for which the compliance (rotation at the moment load location) is at most $37.5$.

Let $I_e (a_e):= \frac{1}{6} a_e^2$ denote the second moment of area of the elements. Then, the principal submatrix associated with the degrees of freedom of the stiffness matrix is the $1\times1$ matrix%
\begin{equation}
    \bm{K}(\bm{a}) = 
       \frac{2}{3}a_1^2 + \frac{2}{3}a_2^2 + \frac{4}{3}a_3^2.
\end{equation}
Similarly, the force vector $\bm{f}$ associated with the degrees of freedom is $\bm{f} = 1$.

An upper bound feasible design based on the uniform distribution of cross sections, $a_1=a_2=a_3 = \frac{1}{10}$, has the weight of $\overline{w} = \frac{1}{4}$. Thus, it follows that
\begin{subequations}
\begin{align}
    0 \le a_1 \le \frac{1}{4} & \longleftrightarrow a_1 = \frac{1}{8} + \frac{1}{8}a_{\mathrm{s},1} \text{ with } a_{s,1}^2 \le 1,\\
    0 \le a_2 \le \frac{1}{4} & \longleftrightarrow a_2 = \frac{1}{8} + \frac{1}{8}a_{\mathrm{s},2} \text{ with } a_{s,2}^2 \le 1,\\
    0 \le a_3 \le \frac{1}{2} & \longleftrightarrow a_3 = \frac{1}{4} + \frac{1}{4}a_{\mathrm{s},3} \text{ with } a_{s,3}^2 \le 1,
\end{align}
\end{subequations}
with the scaled design variables $\bm{a}_s \in [-1,1]^3$ introduced for accurate numerical solution.

Consequently, we can write $\bm{K}(\bm{a})$ in terms of $\bm{a}_\mathrm{s}$ as
\begin{equation}\label{eq:stiffness_scaled_example}
    \bm{K}(\bm{a}_\mathrm{s}) = \frac{5}{48} + \frac{1}{96}a_{\mathrm{s},1}^2 + \frac{1}{48}a_{\mathrm{s},1} + \frac{1}{96}a_{\mathrm{s},2}^2 + \frac{1}{48}a_{\mathrm{s},2} + \frac{1}{12}a_{\mathrm{s},3}^2 + \frac{1}{6}a_{\mathrm{s},3}
\end{equation}
and the optimization problem is formulated as
\begin{mini!}|s|%
{\strut \bm{a}_s}%
 { \frac{3}{8} + \frac{1}{8}(a_{\mathrm{s},1} + a_{\mathrm{s},2}+ a_{\mathrm{s},3}) \label{weight_obj_example}}%
{\label{weight_problem_example}}%
{}
\addConstraint{\begin{pmatrix}
        \frac{75}{2} & -1\\
        -1 & \bm{K}(\bm{a}_\mathrm{s})
    \end{pmatrix}}{\succeq 0\label{weight_const1_example}}
\addConstraint{1-a_{s,e}^2}{\geq 0, ~\forall e \in \lbrace 1, \ldots, n_{e}\rbrace.\label{weight_const2_example}}
\end{mini!}

The optimization problem \eqref{weight_problem_example} has three locally optimal solutions written in terms of cross-sectional variables $\bm{a}$ as $\bm{a}_1^* = (\frac{1}{5}, 0, 0)$ with the corresponding weight $w(\bm{a}_1^*) = \frac{1}{5}$, $\bm{a}_2^* = (0, \frac{1}{5}, 0)$ associated with the weight $w(\bm{a}_2^*) = \frac{1}{5}$, and the globally optimal design $\bm{a}_3^* = (0, 0, \frac{\sqrt{2}}{10})$ of the weight $w(\bm{a}_3^*) = \frac{\sqrt{2}}{20}$. These local solutions are indicated as green points in Fig.~\ref{fig:feasible_set}.
\begin{figure}[!htbp]
    \centering
    \includegraphics[width=8cm]{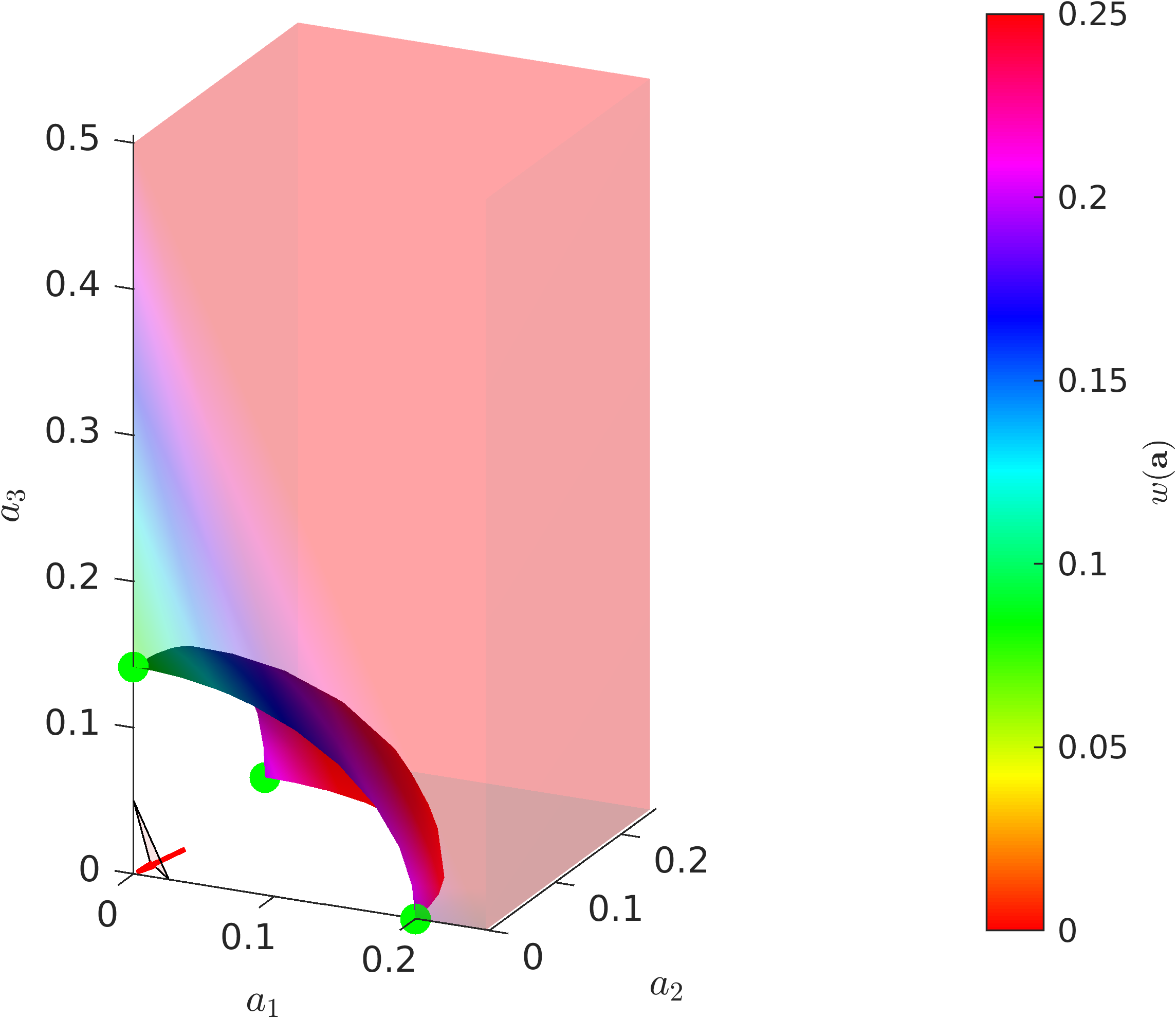}
    \caption{Feasible set of the illustration problem in terms of the cross-section areas $\bm{a}$. The red arrow, which is a normal vector of the triangular patch in the lower left corner, indicates the direction of minimization. Scattered green points denote local minimizers. The boundaries of the feasible set are drawn as semitransparent surfaces, and their color denotes the weight $w(\bm{a})$.}
    \label{fig:feasible_set}
\end{figure}

\section{Lasserre's moment sum-of-square hierarchy} 
\label{sec:lassere_mSOS}
This section is devoted to explaining practical steps for building the Lasserre moment-sum-of-squares (mSOS) hierarchy. This comprehensive part is crucial for understanding the exploitation of term sparsity in the following sections. The theory behind the mSOS hierarchy is beyond the scope of this document, but the reader can consult \cite{lasserre2015introduction,henrion2006convergent} for more details and theoretical aspects, in particular.

We start by introducing tools on multivariate polynomials that will be necessary in the sequel. 

\subsection{Definitions and notations}\label{sec:preliminaries}

Let $\bm{x} \in \mathbb{R}^n$ be a vector of variables. A polynomial $p $ in $\bm{x}$ and of degree $\text{deg}(p)=d$ is written as $p(\bm{x})=  \sum_{\bm{\alpha} \in \mathbb{N}^n_d}p_{\bm{\alpha}}\bm{x}^{\bm{\alpha}}$ where $p_{\bm{\alpha}} \in \mathbb{R}$ is a real number indexed by $\bm{\alpha}$, $\bm{x}^{\bm{\alpha}}=x_1^{\alpha_{1}} \ldots x_n^{\alpha_{n}}$ and $\mathbb{N}^n_d=\left\lbrace \bm{\alpha} \in \mathbb{N}^n : \sum_{i=1}^n \alpha_i \leq d \right\rbrace$. It can also be written as $p(\bm{x})= \bm{q}^T\bm{b}_d(\bm{x})$, where
\begin{equation}
\begin{aligned}
\bm{b}_d(\bm{x})=&\Big( 1 \quad x_1 \quad x_2 \quad \ldots \quad x_1^2 \quad x_1x_2 \quad  \ldots  \quad x_1x_n \\ &  x_2^2 \quad x_2x_3 \quad   \ldots \quad x_n^2 \quad \ldots \quad x_1^d \quad \ldots \quad x_n^d \Big)^T
\end{aligned}
\label{strandard_basis}
\end{equation}
is the standard (or canonical) monomial basis of the polynomial ring $\mathbb{R}[\bm{x}]_d$ and $\bm{q} \in \mathbb{R}^{\vert\bm{b}_d(\bm{x}) \vert}$ is a vector of coefficients associated with the basis $\bm{b}_d(\bm{x})$. Let us note that $\vert \bm{b}_d(\bm{x}) \vert = \binom{n+d}{n}$ where the notation $\vert \bullet \vert $ stands for the set $\bullet$ cardinality. The representation basis \eqref{strandard_basis} is often not unique. In fact, if the expression of a polynomial $p$ does not involve all the monomial terms of the standard basis $\bm{b}_d(\bm{x})$, one can represent $p$ in another basis with fewer monomial terms than in $\bm{b}_d(\bm{x})$ with an appropriate vector $\bm{q}$. This is illustrated in the following example. 

\begin{example}
Let $p(\bm{x})=1+4x_1x_2-3x_1^2+4x_1x_2^2+6x_2^2$ be a polynomial in $n=2$ variables and of degree $d=3$. It can be written as $p(\bm{x})=\bm{q}^T \bm{b} (\bm{x})$ where
 $\bm{b} (\bm{x})=\bm{b}_3(\bm{x})= \begin{pmatrix}
     1 & x_1 & x_2 & x_1^2 & x_1x_2 & x_2^2 & x_1^3 & x_1^2x_2 & x_1x_2^2 & x_2^3
 \end{pmatrix}^T $ and $\bm{q}= \begin{pmatrix}
     1 & 0 & 0 & -3 & 4 & 6 & 0 & 0 & 4 &  0
 \end{pmatrix} $. 
However, we can also choose $\bm{b} (\bm{x})=\begin{pmatrix}
    1 & x_1^2 & x_1x_2 & x_2^2 & x_1x_2^2
\end{pmatrix}^T $ and
$\bm{q}= \begin{pmatrix}
    1 & -3 & 4 & 6 & 5
\end{pmatrix}.$ 
\label{example_basis}
\end{example}

Next, we denote the set of real symmetric matrices of order $s$ with $\mathbb{S}^s$. Let $\bm{G}$ be a real symmetric polynomial matrix, i.e.,  $\bm{G} : \mathbb{R}^n \rightarrow \mathbb{S}^s$ such that $\forall i, j : G_{i,j}(\bm{x})=G_{j,i}(\bm{x})$ are scalar polynomials of $\bm{x}$. The degree of $\bm{G}(\bm{x})$ is the highest degree of polynomials in $\bm{G}(\bm{x})$ and will be denoted as $\text{deg}(\bm{G})$ in the following. A polynomial matrix $\bm{G}$ of degree $d$ can be written as
$\bm{G}(\bm{x})=\sum_{\bm{\alpha} \in \mathbb{N}^n_d}\bm{G}_{\bm{\alpha}}\bm{x}^{\bm{\alpha}}$ where $\bm{G}_{\bm{\alpha}} \in \mathbb{S}^s$. In particular, if $s=1$ then $\bm{G}$ is a scalar polynomial. A polynomial matrix inequality (PMI) is an inequality of the form $\bm{G}(\bm{x}) \succeq 0$.

Let us now consider the standard basis $\bm{b}_d(\bm{x})$ defined in \eqref{strandard_basis} for an $\bm{x} \in \mathbb{R}^n$. Let $\bm{y} \in \mathbb{R}^{\vert \bm{b}_d(\bm{x}) \vert }$ and $\bm{\alpha} \in \mathbb{N}^{n}_d$. By the notation $y_{\bm{\alpha}}$ we mean the component of the vector $\bm{y}$ that corresponds to the exponent $\bm{\alpha}$ of a monomial $\bm{x^\alpha}=x_1^{\alpha_1}x_2^{\alpha_2}\ldots x_n^{\alpha_n}$ in $\bm{b}_d(\bm{x})$. 
For example, if $\bm{x}=\begin{pmatrix} x_1 & x_2 & x_3 & x_4 \end{pmatrix}$, then $y_{3102}$  is the component of $\bm{y}$ that corresponds to the monomial $x_1^3x_2x_4^2 \in \bm{b}_6(\bm{x}) $ where $\bm{\alpha}= \begin{pmatrix}    3 & 1 & 0 & 2 \end{pmatrix}$. In Section \ref{certificate-global}, we will also use the notation $y_{\bm{x^\alpha}}$ to refer to $y_{\bm{\alpha}}$, allowing to write $y_{3102}$ as $y_{x_1^3 x_2 x_4^2}$.

 Now, given a polynomial expression $\sum_{\bm{\alpha}}p_{\bm{\alpha}}\bm{x}^{\bm{\alpha}}$, we can linearize it using the following  functional $L_{\bm{y}} : \mathbb{R}[\bm{x}]_d  \rightarrow \mathbb{R}$ such that $$L_{\bm{y}} \left(\sum_{\bm{\alpha}}p_{\bm{\alpha}}\bm{x}^{\bm{\alpha}}\right)=\sum_{\bm{\alpha}}p_{\bm{\alpha}}y_{\bm{\alpha}},$$ where $\bm{y}=(y_{\bm{\alpha}})_{\bm{\alpha} \in \mathbb{N}^n_d}$ is a sequence indexed by the standard monomial basis. In the following, the notation $L_{\bm{y}} (\bm{W}(\bm{x}))$, where $\bm{W}(\bm{x})$ is a polynomial matrix, means that the functional $L_{\bm{y}} $ is applied entry-wise to the matrix $\bm{W}$.

\begin{definition}[Moment and localizing matrix]
\label{def-moment-matrix}
Let $\bm{x} \in \mathbb{R}^n$, $\bm{y} \in \mathbb{R}^{\vert \bm{b}_{2d}(\bm{x}) \vert}$ and $\bm{G}(\bm{x})$ a polynomial matrix of degree $s$.
\begin{itemize}
    \item The moment matrix $\bm{M}_d(\bm{y})$ associated with $\bm{y}$ is the matrix whose rows and columns are indexed by the standard monomial basis $\bm{b}_d(\bm{x})$ such that $$ \bm{M}_d(\bm{y})=L_{\bm{y}} \left( \bm{b}_d(\bm{x})\bm{b}_d(\bm{x})^T  \right)=\left(y_{\bm{\alpha} +\bm{\beta}} \right) _{\bm{\alpha}, \bm{\beta} \in \mathbb{N}^n_d}.$$
    \item The localizing matrix associated to $\bm{G}$ and $\bm{y}$ is the matrix with rows and columns indexed by the standard monomial basis $ \bm{b}_{d-\lceil s/2 \rceil}(\bm{x})$, such that 

\begin{align*}
    \bm{M}_d(\bm{G}\bm{y})&=L_{\bm{y}} \left( \bm{b}_{d-\lceil s/2 \rceil}(\bm{x})^T\bm{b}_{d-\lceil s/2 \rceil}(\bm{x}) \otimes \bm{G}(\bm{x}) \right)\\
    & =\left(\sum_{\bm{\gamma} \in \mathbb{N}^n_s} y_{\bm{\alpha} +\bm{\beta}+\bm{\gamma}} \otimes \bm{G}_{\bm{\gamma}} \right) _{\bm{\alpha}, \bm{\beta} \in \mathbb{N}^n_d},
\end{align*}
   where the symbol $\otimes$ stands for the Kronecker or tensor product and $\lceil\bullet\rceil$ rounds $\bullet$ up to the nearest integer.
\end{itemize}
\end{definition}

\subsection{The moment sum-of-squares hierarchy}
\label{mSOS}
Consider now an optimization problem of the form
\begin{equation}
\begin{aligned}
&  f^*=&\underset{\bm{x} \in \mathbb{R}^n}{\inf} f(\bm{x}),&  \\
& \text{s.t. }& \bm{G}_j(\bm{x}) \succeq 0,&~ \forall j \in \lbrace 1,\ldots, m\rbrace ,
\end{aligned}
\label{POP1}
\end{equation}
where $f$ is a scalar polynomial, and $\bm{G}_j$ for $ j \in \lbrace 1, \ldots, m \rbrace$ are symmetric matrix polynomials in the variable $\bm{x}$.

Let $d_j=\left\lceil \frac{\text{deg}(\bm{G}_j)}{2}\right\rceil$ for $j \in \lbrace 1, \ldots, m \rbrace$ and $r \geq \max \left\lbrace \left\lceil \frac{\text{deg}(f)}{2}\right\rceil, d_1,\ldots d_m  \right\rbrace$. \emph{Lasserre's moment sum-of-square hierarchy} (mSOS) of SDP relaxations of \eqref{POP1} indexed by the relaxation order $r$ reads as 
\begin{equation}
\begin{aligned}
&  f_r=&\underset{\bm{y} \in \mathbb{R}^{\vert \bm{b_{2r}}(\bm{x}) \vert}}{\inf} \bm{L_y}\left(f \right),&  \\
& \text{s.t. }& \bm{M}_{r-d_j}(\bm{G}_j \bm{y}) \succeq 0, &~ \forall j \in \lbrace 1,\ldots, m \rbrace, \\
& & \bm{M}_r(\bm{y}) \succeq 0,& \\
& & y_{\bm{0}}=1.
\end{aligned}
\label{lassereSOS}
\end{equation}
A solution $\bm{y}^*$ to \eqref{lassereSOS} is called a vector of optimal moments. Next, we will state a theorem that proves that, under an assumption on algebraic compactness of the feasible set, the sequence of lower bounds $f_r$ indexed by $r$ and obtained by solving \eqref{lassereSOS} converges monotonically to the optimal value of problem \eqref{POP1}.

First, let us note that the constraints of the problem \eqref{POP1} can be concatenated into a single PMI constraint. This follows from replacing the constraints by $\widetilde{\bm{G}}(\bm{x}) \succeq 0 $ where $\widetilde{\bm{G}}(\bm{x})=\text{diag}\left(\bm{G}_{1}(\bm{x}), \ldots,\bm{G}_{m}(\bm{x}) \right) \in \mathbb{S}^N$.

Let us first recall that a polynomial matrix $\bm{G}: \mathbb{R}^n \rightarrow \mathbb{S}^s $ of even degree is a \emph{matrix SOS} if there exists a polynomial matrix $\bm{H}: \mathbb{R}^n \rightarrow \mathbb{R}^{s \times \ell}  $ \text{for some }$\ell \in \mathbb{N}$, such that $$ \bm{G}(\bm{x})= \bm{H}(\bm{x}) \bm{H}(\bm{x})^T.$$

\begin{assumption}
Assume that there exist matrix SOS polynomials $p_0 : \mathbb{R}^n  \rightarrow \mathbb{R}$ and $\bm{R}: \mathbb{R}^n  \rightarrow \mathbb{S}^N$ such that the set $\left\lbrace \bm{x} \in \mathbb{R}^n : p_0(\bm{x})+  \left\langle \bm{R}(\bm{x}), \widetilde{\bm{G}}(\bm{x}) \right\rangle \geq 0 \right\rbrace$ is compact. 
\label{archimedian}
\end{assumption}

\begin{theorem}[\cite{henrion2006convergent}, Theorem 2.2]
If Assumption \ref{archimedian} is satisfied, then $f_r \nearrow f^*$ as $r \rightarrow \infty$ in \eqref{lassereSOS}, i.e., there exists an $\tilde{r} \geq r$ such that
$$ f_r \leq f_{r+1} \leq \ldots \leq f_{\tilde{r}}=f^*.$$ 
\label{Henrion2006}
\end{theorem}
As stated in \cite{henrion2006convergent}, Assumption \ref{archimedian} is not very restrictive, it suffices that one of the constraints of problem \eqref{POP1} is of the form $\rho^2-\rVert \bm{x} \rVert^2 \geq 0$. If this is not the case, it can be added as a redundant constraint by choosing $\rho$ so that it does not affect the solution of problem \eqref{POP1}.
For frame structure problems \eqref{weight_problem} and \eqref{compliance_problem}, the formulations \eqref{compliance_problem_sdp_scalled} and \eqref{weight_problem_sdp_scalled} have a compact feasible set and, in addition, also satisfy Assumption \ref{archimedian} for the convergence of the Lasserre hierarchy \cite[Proposition 4]{tyburec2021global}.  

The finite convergence stated in Theorem \ref{Henrion2006} occurs when the following flatness condition holds \cite{curto1996solution}
\begin{equation}
 \text{Rank}(\bm{M}_k(\bm{y}^*))= \text{Rank}(\bm{M}_{k-d}(\bm{y}^*)),
\label{flatness}
\end{equation}
where $\bm{y}^*$ is a solution to \eqref{lassereSOS} and $d=\max \left\lbrace d_1,\ldots,d_m \right\rbrace$.
Based on this flatness condition, one can extract from \eqref{lassereSOS} $\tau=\text{Rank}(\bm{M}_k(\bm{y}^*))$ distinct global minimizers of problem \eqref{POP1}, using a linear algebra routine described in detail in \cite{henrion2005detecting}. Depending on the nature of the problem, we can also have other types of conditions for finite convergence; see, for instance, the recent work \cite{henrion2023occupation}. Other conditions to certify finite convergence have been developed in \cite{tyburec2021global,tyburec2022global} for frame structure problems. These conditions are briefly presented in the following section.

\subsection{Certificate of global $\varepsilon$-optimality in frame structures optimization}
\label{certificate-global}
In the context of topology optimization of frame structures, computationally simpler condition based on the first-order moments can be used as a certificate of global optimality for \eqref{compliance_problem_sdp_scalled} and \eqref{weight_problem_sdp_scalled}.

\subsubsection{Compliance optimization}

Let $y^*_{c^1}$ and $\bm{y}^*_{\bm{a}^1}$ be the first-order moments associated with variables $c_s$ and $\bm{a}_s$ obtained from a solution to any relaxation of \eqref{compliance_problem_sdp_scalled} using the moment sum of squares hierarchy. Then, 
$$\tilde{{a}}_e=0.5(y_{a_{e}^1}^*+1)\frac{\overline{w}}{\ell_e \rho_e}, ~\forall e \in \lbrace 1, \ldots, n_e\rbrace,$$
$$\tilde{{c}}=\bm{f}^T \bm{K}(\tilde{\bm{a}})^{\dag}\bm{f}$$ gives a feasible upper bound to \eqref{compliance_problem_sdp_scalled} \cite[Theorem 2]{tyburec2021global}.

In addition, $y^*_{c^1}$ and $\bm{y}^*_{\bm{a}^1}$ can be used to construct a sufficient condition for global optimality: if for every $\varepsilon > 0$ there is a relaxation degree $r$ such that \begin{equation}
\tilde{c}^{(r)}- 0.5(y^{(r)*}_{c^1}+1) \leq \varepsilon\,,
    \label{sufficient_condition_convergence}
\end{equation}
then $\tilde{\bm{a}}$, $\tilde{c}$ tend to a global solution to \eqref{compliance_problem_sdp_scalled}.

Condition \eqref{sufficient_condition_convergence} is not a necessary condition. Indeed, in the case of a disconnected set of global minimizers, the optimality gap $\varepsilon$ may remain strictly positive while the flatness condition \eqref{flatness} holds; see \cite[Section 4.2]{tyburec2021global} for an illustration. However, when the optimization problem possesses a convex set of global minimizers, global optimality can always be certified for the relaxation degree $r\rightarrow \infty$ by \cite[Theorem 3]{tyburec2021global}

\begin{equation}
\lim_{r \rightarrow \infty} \tilde{c}^{(r)}- 0.5(y_{c^1}^{(r)*}+1)\overline{c} = 0.
    \label{necessary_condition_convergence}
\end{equation}

We note here that \eqref{necessary_condition_convergence} also holds for the case of infinitely many minimizers within a convex set, for which the rank in the condition \eqref{flatness} does not stabilize.

\subsubsection{Weight optimization}\label{sec:weightUB}

In addition, we consider the case of weight optimization.
Let $\bm{y}^*_{\bm{a}^1}$ be the first-order moments associated with the variables $\bm{a}_s$ obtained from a solution to any relaxation of \eqref{weight_problem_sdp_scalled} using the moment sum of squares hierarchy. If $\mathbf{K}_{0} = \mathbf{0}$ (no prescribed stiffness) or when $\bm{y}^*_{\bm{a}^1} > -\mathbf{1}$ (all cross-section areas are strictly positive), then there exists $\delta\ge 1$ such that
\begin{equation*}
    \tilde{a}_e = \frac{y^*_{{a}_e^1}+1}{2} \frac{\overline{w}}{\rho_e \ell_e} \delta,\ \  \forall e \in \{1,\dots, n_\mathrm{e}\}
\end{equation*}
is a feasible upper-bound solution to \eqref{weight_problem_sdp_scalled}, where the minimal value $\delta^*$ can be found by solving a quasi-convex optimization problem \\$\delta^* = \argmin_\delta \left\{\delta \; \vert \; \delta>0, \bm{f}^T\bm{K}(\delta\tilde{\bm{a}})^\dagger \bm{f} \le \overline{c} \right\}$ globally by a bisection-type algorithm \cite[Proposition 5]{tyburec2022global}.

Thus, we can again state a simple sufficient condition of global $\varepsilon$-optimality \cite[Proposition 6]{tyburec2022global} as
\begin{equation}
    (\delta^*-1) 0.5 \overline{w} (n_\mathrm{e} + \sum_{e=1}^{n_\mathrm{e}} y_{a_e^1}^{*}) \le \varepsilon,
\end{equation}
which approaches zero when the first-order moments are feasible to the original problem, i.e., $\delta^* = 1$.

Similarly to the compliance optimization, also here we have convergence in the limit for optimization problems with a convex set of global minimizers \cite[Theorem 3]{tyburec2022global}:
\begin{equation}
    \lim_{r\rightarrow \infty} (\delta^*-1) 0.5 \overline{w} (n_\mathrm{e} + \sum_{e=1}^{n_\mathrm{e}} y_{a_e^1}^{(r)*}) = 0.
\end{equation}

\begin{example}
Consider the problem \eqref{weight_problem_example} and set $\bm{x}=\begin{pmatrix}
    a_{s,1} & a_{s,2} & a_{s,3}
\end{pmatrix}$. The problem then writes as
\begin{equation}
\begin{aligned}
&  \underset{\bm{x} \in \mathbb{R}^3}{\min} & f(\bm{x})= \frac{3}{8}+\frac{1}{8}(x_1+x_2+x_3),&  \\
& \text{s.t. } &  \bm{G}_1(\bm{x})=\begin{pmatrix}
\frac{75}{2} & -1 \\ -1 & K(\bm{x})
\end{pmatrix} & \succeq 0, \\
& & \bm{G}_2(\bm{x})=1-x_1^2 & \geq 0,\\
&  & \bm{G}_3(\bm{x})=1-x_2^2 & \geq 0,\\
&  & \bm{G}_4(\bm{x})=1-x_3^2 & \geq 0,\\
\end{aligned}
\label{POP-example}
\end{equation}
where $K(\bm{x})=\frac{5}{48}+\frac{1}{48}(x_1+x_2+8x_3)+ \frac{1}{96}(x_1^2+x_2^2+8 x_3^2)$. To ease notations in the example, we will write $K:=K(\bm{x})$
\begin{itemize}
\item For $r=1$, the Lasserre's SOS hierarchy writes as
\begin{mini!}|s|%
{\strut \bm{y} \in \mathbb{R}^{4}}%
 { \frac{3}{8}y_{000}+\frac{1}{8}(y_{100}+y_{010}+y_{001})}%
{\label{SOS-example_1}}{}
\addConstraint{\bm{M}_0(\bm{G}_j\bm{y})}{\succeq 0,~j=1, 2, 3, 4}
\addConstraint{\bm{M}_1(\bm{y})}{\succeq 0}
\addConstraint{y_{000}}{= 1,}
\end{mini!}
where
 $\bm{M}_1(\bm{y})=
\begin{pmatrix}
    y_{00} & y_{10} &  y_{01} \\
    y_{10} & y_{20} &  y_{11} \\
    y_{01} & y_{11} &  y_{02}
\end{pmatrix}, $ 
$  \bm{M}_0(\bm{G}_1\bm{y}) = \begin{pmatrix}
 \frac{75}{2}y_{000} & -y_{000} \\ -y_{000} & \bm{L_y}(K)
 \end{pmatrix},$ and \\$ \bm{M}_0(\bm{G}_2\bm{y})=y_{000}-y_{200}$, $\bm{M}_0(\bm{G}_3\bm{y})=y_{000}-y_{020}$ and $ \bm{M}_0(\bm{G}_4\bm{y})=y_{000}-y_{002}.$ The optimal value of SDP \eqref{SOS-example_1} is $f_1^*=0.02$ (see Figure \ref{fig:rel1dense}) and the flatness condition \eqref{flatness} does not hold. The first-order moments from the relaxation read as $\bm{y}_{\bm{a}^1} = \begin{pmatrix}
    -1 & -1 & -0.84
\end{pmatrix}$. Using the optimality condition from Section \ref{sec:weightUB}, we find the smallest $\delta$ that makes the first-order moments feasible as $2.5\sqrt{2}$. The corresponding feasible cross-section areas evaluate as $\tilde{\bm{a}} = \begin{pmatrix}
    0 & 0 & \sqrt{2}/10
\end{pmatrix}$ and the upper-bound weight thus reads $\sqrt{2}/20$.
\item For $r=2$, Lasserre's mSOS hierarchy reads as
\begin{mini!}|s|%
{\strut \bm{y} \in \mathbb{R}^{10}}%
 {\frac{3}{8}y_{000}+\frac{1}{8}(y_{100}+y_{010}+y_{001})}%
{\label{SOS-example_2}}{}
\addConstraint{\bm{M}_1(\bm{G}_j\bm{y})}{\succeq 0,~j=1, 2, 3, 4}
\addConstraint{\bm{M}_2(\bm{y})}{\succeq 0}
\addConstraint{y_{000}}{= 1,}
\end{mini!}

where 
 $\bm{M}_1(\bm{G}_1\bm{y})= \begin{bmatrix}
\bm{M}_{000}(\bm{G}_1\bm{y}) & \bm{M}_{100}(\bm{G}_1\bm{y}) & \bm{M}_{010}(\bm{G}_1\bm{y}) & \bm{M}_{001}(\bm{G}_1\bm{y}) \\
\bm{M}_{100}(\bm{G}_1\bm{y}) & \bm{M}_{200}(\bm{G}_1\bm{y}) & \bm{M}_{110}(\bm{G}_1\bm{y}) & \bm{M}_{101}(\bm{G}_1\bm{y}) \\
\bm{M}_{010}(\bm{G}_1\bm{y}) & \bm{M}_{110}(\bm{G}_1\bm{y}) & \bm{M}_{020}(\bm{G}_1\bm{y})  & \bm{M}_{011}(\bm{G}_1\bm{y}) \\
\bm{M}_{001}(\bm{G}_1\bm{y}) & \bm{M}_{101}(\bm{G}_1\bm{y}) & \bm{M}_{011}(\bm{G}_1\bm{y})  & \bm{M}_{002}(\bm{G}_1\bm{y}) 
\end{bmatrix},$ \\
with

\begin{align*}
\bm{M}_{000}(\bm{G}_1\bm{y})&=\begin{pmatrix}
\frac{75}{2}y_{000} & -y_{000} \\
-y_{000} & L_{\bm{y}}(K)
 \end{pmatrix}, \\
\bm{M}_{100}(\bm{G}_1\bm{y})&=\begin{pmatrix}
\frac{75}{2}y_{100} & -y_{100} \\
-y_{100} & L_{\bm{y}}(x_1K)
 \end{pmatrix},
\\
 \bm{M}_{010}(\bm{G}_1\bm{y})&=\begin{pmatrix}
\frac{75}{2}y_{010} & -y_{010} \\
-y_{010} & L_{\bm{y}}(x_2K)
 \end{pmatrix}, \cdots
\end{align*}
and $\bm{M}_1(\bm{G}_2\bm{y})=\begin{pmatrix}
   y_{000}-y_{200} &  y_{100}-y_{300} &   y_{010}-y_{210} &   y_{001}-y_{201} \\
  y_{100}-y_{300} &  y_{200}-y_{400} &   y_{110}-y_{310} &   y_{101}-y_{301}\\
   y_{010}-y_{210} &  y_{110}-y_{310} &   y_{020}-y_{220} &   y_{011}-y_{211}\\
  y_{001}-y_{201} &  y_{101}-y_{301} &   y_{011}-y_{211} &   y_{002}-y_{202}
\end{pmatrix}.$ 
The construction of the matrices $ \bm{M}_1(\bm{G}_3\bm{y})$ and $\bm{M}_1(\bm{G}_4\bm{y})$ is similar to $\bm{M}_1(\bm{G}_2\bm{y})$. Here, the flatness condition \eqref{flatness} holds and the optimal value of the SDP \eqref{SOS-example_2} is $f_2^*=0.0707 \approx \frac{\sqrt{2}}{20}=f^*$, see Figure \ref{fig:rel2dense}. For the optimality certificate in Section \ref{sec:weightUB}, we take the first-order moments $\bm{y}_{\bm{a}^1} = \begin{pmatrix}
    -1 & -1 & 0.4\sqrt{2}-1
\end{pmatrix}$
and find that they are feasible to the optimization problem, i.e., $\delta^* = 1$. Consequently, we have the optimal solution $\bm{a}^* = \begin{pmatrix}
    0 & 0 & \sqrt{2}/10
\end{pmatrix}$.
\end{itemize}
\begin{figure}[!htbp]
    \subfloat[\label{fig:rel1dense}]{%
      \includegraphics[width=0.475\linewidth]{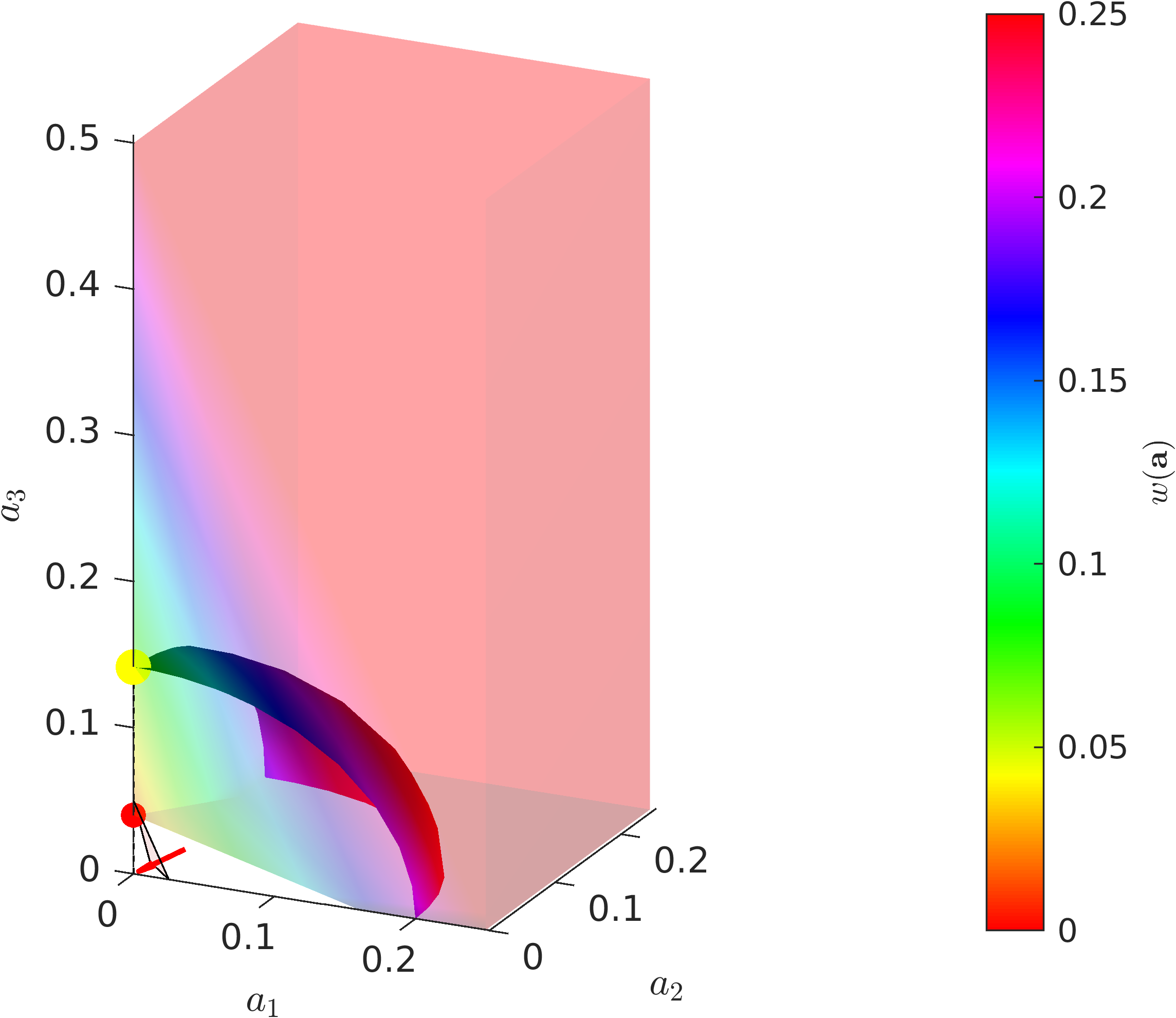}
    }%
    \hfill%
    \subfloat[\label{fig:rel2dense}]{%
      \includegraphics[width=0.475\linewidth]{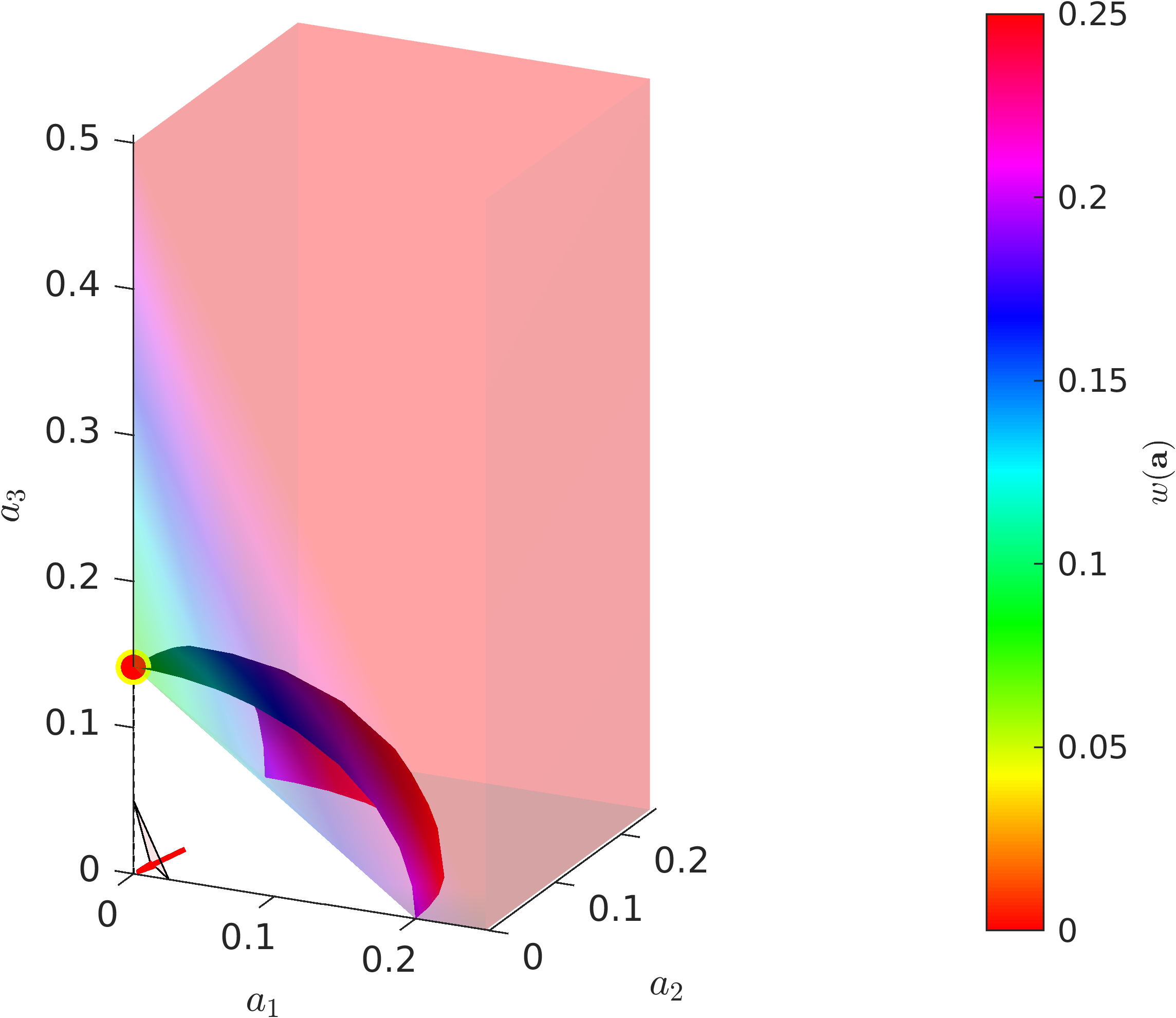}
    }
    \caption{Feasible set of the first-order moments for (a) $r=1$ and (b) $r=2$. The red arrow, which is a normal vector of the triangular patch in the lower left corner, indicates the direction of minimization. The scattered red points denote the relaxation lower bounds and the yellow points the reconstructed upper bounds based on Section \ref{sec:weightUB}. The boundaries of the feasible set are drawn as solid, whereas the relaxation feasible set is drawn as a transparent surface, its color denoting the weight $w(\bm{a})$.} \label{fig:dense}
\end{figure}
\end{example}

\section{Exploiting term sparsity in the mSOS }

The polynomials defining a POP can be sparse in terms of monomials, i.e., involve fewer monomial terms in the standard basis $\bm{b}_d(\bm{x})$; recall Example~\ref{example_basis}. The \emph{term sparsity pattern} (TSP) technique consists of exploiting this fact to construct sparse moment and localizing matrices introduced in Definition~\ref{def-moment-matrix}. Each of these sparse matrices can be decomposed into smaller matrices using the results presented in the following section, and thus provide computational and time savings. 

In this section, we describe the TSP method for polynomial matrix inequalities. The developments are inspired by and extend the results of \cite{wang2020chordal,wang2021tssos,magron2023sparse,magron2021tssos,wang2022cs}.

\subsection{Graphs and sparse positive semidefinite matrices }
\label{sec:graphs-psd}
In this part, we recall the basic notions of graph theory and its relation to positive semidefinite (PSD) matrices. These notions are crucial for the decomposition of sparse PSD matrices, but also for exploiting the term sparsity. 

An \emph{undirected graph}\footnote{An undirected graph is also called a graph.} $ \mathcal{G}(V,E)$ is a set of nodes $V$ and a set of edges $E$ between these nodes, that is, $E \subseteq \left\lbrace \lbrace v_i,v_j \rbrace: (v_i,v_j) \in V \times V \right\rbrace$. An edge in $E$ that connects a node to itself is called a \emph{loop}. For a given graph $\mathcal{G}$, we use the notation $V(\mathcal{G})$ (reps. $E(\mathcal{G})$) to indicate the set of nodes (resp. edges) of $\mathcal{G}$. A \emph{cycle of length $p$} is a set of consecutive distinct nodes $\lbrace v_1, \ldots, v_p\rbrace \subseteq V$ such that $\forall i \in \{1,\ldots, p-1\}: \lbrace v_i, v_{i+1} \rbrace \in E$ and $\lbrace v_1, v_p \rbrace \in E$. Given a cycle, a \emph{chord} is an edge between two nonconsecutive nodes in that cycle. 

If $\mathcal{G}$ and $\mathcal{H}$ are two graphs, we say that $\mathcal{G}$ is a \emph{subgraph} of $\mathcal{H}$ which we denote as $\mathcal{G} \subseteq \mathcal{H}$, if $V(\mathcal{G}) \subseteq V(\mathcal{H}) $ and $E(\mathcal{G}) \subseteq E(\mathcal{H}) $. A \emph{connected graph} is a graph in which for every two vertices $v_i$ and $v_j$ there exists a set of distinct vertices $v_i, v_{i+1},\dots, v_j$ such that $(v_i,v_{i+1}) \in E(\mathcal{G})$. A \emph{connected component} in a graph is a connected subgraph that is not part of any larger connected subgraph. A graph is called \emph{complete} if every two distinct nodes are connected by an edge. 

A graph is \emph{chordal} if all its cycles of the length at least $4$ have a chord. For example, the graph in Figure~\ref{example_chordal_graph}(a) is not chordal because the cycle $1-2-3-4-5$ does not contain any chord. If a graph $\mathcal{G}(V,E)$ is not chordal, it can be extended to a chordal graph by adding appropriate edges to $E$, e.g., Figure \ref{example_chordal_graph}(b) and (c). This operation is called \emph{chordal extension} (or triangulation) and will be denoted by $\overline{\mathcal{G}}$ in this paper. A chordal extension is not unique in general. We call the chordal extension maximal (resp. minimal) if it is constructed by adding the maximum (resp. minimum) number of edges in order to make the graph chordal. The minimum chordal extension of a graph is generally an NP-hard problem, commonly tackled using heuristic algorithms (see, for example, \cite{heggernes2006minimal,amestoy2004algorithm,silva2023biased}).

\begin{figure}[h!]

\subfloat[]{\begin{tikzpicture}[node distance={15mm}, thick, main/.style = {draw, circle}, auto]
\node[main]  (1) {$1$}; 
\node[main] [below right of=1] (2) {$2$}; 
\node[main] [ below left of=1] (5) {$5$}; 
\node[main] [ below of=2] (3) {$3$}; 
\node[main] [ below of=5](4) {$4$}; 
\node[main] [ right of=2](6) {$6$}; 
\node[main] [ below of=6](7) {$7$}; 
\node[main] [ above right of=7](8) {$8$}; 
\draw (1) -- (2);
\draw (2) -- (3);
\draw (3) -- (4);
\draw (4) -- (5);
\draw (5) -- (1);
\draw (6) -- (7);
\draw (7) -- (8);
\end{tikzpicture}  }
\hspace{0.5cm}
\subfloat[]{\begin{tikzpicture}[node distance={15mm}, thick, main/.style = {draw, circle}, auto]
\node[main]  (1) {$1$}; 
\node[main] [below right of=1] (2) {$2$}; 
\node[main] [ below left of=1] (5) {$5$}; 
\node[main] [ below of=2] (3) {$3$}; 
\node[main] [ below of=5](4) {$4$}; 
\node[main] [ right of=2](6) {$6$}; 
\node[main] [ below of=6](7) {$7$}; 
\node[main] [ above right of=7](8) {$8$}; 
\draw (1) -- (2);
\draw (2) -- (3);
\draw (3) -- (4);
\draw (4) -- (5);
\draw (5) -- (1);
\draw (6) -- (7);
\draw (7) -- (8);
\draw[blue,dashed] (1) -- (3);
\draw[blue,dashed] (1) -- (4);
\draw[blue,dashed] (2) -- (4);
\draw[blue,dashed] (2) -- (5);
\draw[blue,dashed] (3) -- (5);
\draw[blue,dashed] (6) -- (8);
\end{tikzpicture}}
\\ \centering \subfloat[]{
  \begin{tikzpicture}[node distance={15mm}, thick, main/.style = {draw, circle}, auto]
\node[main]  (1) {$1$}; 
\node[main] [below right of=1] (2) {$2$}; 
\node[main] [ below left of=1] (5) {$5$}; 
\node[main] [ below of=2] (3) {$3$}; 
\node[main] [ below of=5](4) {$4$}; 
\node[main] [ right of=2](6) {$6$}; 
\node[main] [ below of=6](7) {$7$}; 
\node[main] [ above right of=7](8) {$8$}; 
\draw (1) -- (2);
\draw (2) -- (3);
\draw (3) -- (4);
\draw (4) -- (5);
\draw (5) -- (1);
\draw (6) -- (7);
\draw (7) -- (8);
\draw[blue,dashed] (4) -- (1);
\draw[blue,dashed] (4) -- (2);
\end{tikzpicture}}
\caption{(a) Example of a non-chordal graph with two connected components, its (b) maximal 
 and (c) minimal chordal extensions.}
\label{example_chordal_graph}
\end{figure}
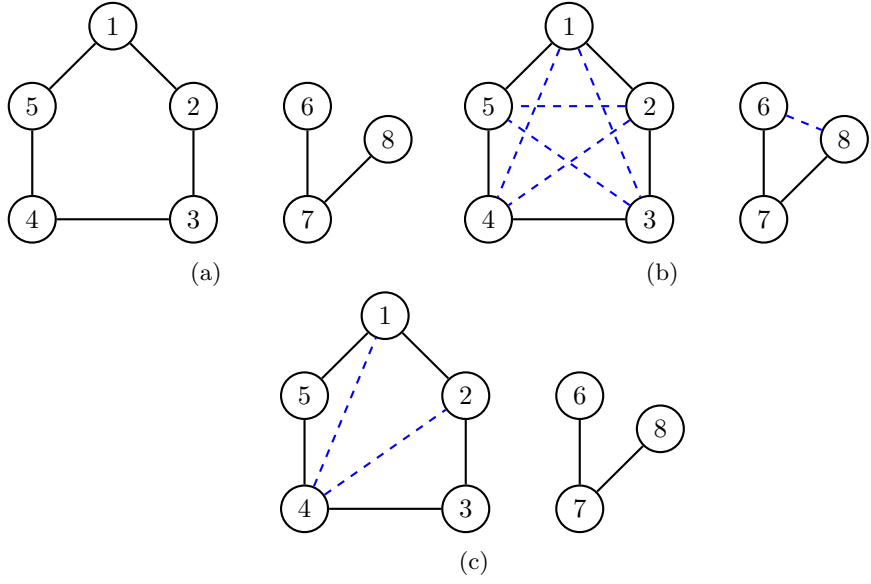

A \emph{clique} $C$ of $\mathcal{G}$ is a subset of nodes $C \subseteq V$ such that any two nodes in $C$ are related by an edge, i.e., $\forall v_i,v_j \in E$, $\lbrace v_i,v_j \rbrace \in E$, and thus $C$ is a complete subgraph. We say that a clique $C$ is \emph{maximal} if it is not a subset of any other clique. For example, the maximal cliques of the graph in Figure \ref{example_chordal_graph}(b) are $C_1=\begin{pmatrix} 1 & 2 & 3 & 4 & 5 \end{pmatrix}$ and $C_2=\begin{pmatrix} 6 & 7 & 8\end{pmatrix}$, but the clique $\begin{pmatrix}1 & 2 & 3\end{pmatrix}$ is not maximal because it is contained in $C_1$. Similarly, the maximal cliques of the graph in Figure \ref{example_chordal_graph}(c) are  $C_1=\begin{pmatrix} 1 & 4 & 5\end{pmatrix}$, $C_2=\begin{pmatrix} 1 & 2 & 4\end{pmatrix}$, $C_3=\begin{pmatrix} 2 & 3 & 4\end{pmatrix}$, $C_4=\begin{pmatrix} 6 & 7\end{pmatrix}$ and $C_5=\begin{pmatrix} 7 & 8 \end{pmatrix}$.

Let $\mathcal{G}(V,E)$ be a graph and $\bm{M}$ be a symmetric matrix whose rows and columns are indexed by $V$. We say that $\bm{M}$ has a \emph{sparsity pattern} $\mathcal{G}$ if $M_{v_{i}v_{j}}=M_{v_{j}v_{i}}=0$ whenever $\lbrace v_i,v_j \rbrace \notin E$. In the sequel, we denote a submatrix of $\bm{M}$ whose rows and columns are indexed by a subset $C \subseteq V$ as $\bm{M}_{C}$. 
The following theorem allows decomposing a sparse PSD matrix that has a chordal sparsity graph into smaller PSD matrices using its maximal cliques. 

\begin{theorem}\cite[Theorem 9.2, Theorem 10.1]{andersen2015chordal}
\label{psd-matrix-decomposition}
Let $\mathcal{G}(V,E)$ be a chordal graph and let $C_1, \ldots, C_s$ be the maximal cliques of $\mathcal{G}$. Let $\bm{M}$ be a symmetric matrix whose sparsity pattern is $\mathcal{G}$. Then $\bm{M} \succeq 0$ if and only if $\bm{M}_{C_{k}} \succeq 0$,  $\forall k \in \lbrace 1, \ldots, s \rbrace$.
\label{matrix_decomposition}
\end{theorem}
For a more detailed study on the chordal graphs and their relation to PSD matrices, the reader can consult \cite{andersen2015chordal}.

\begin{example}
    Let $\bm{M}$ be a matrix associated with sparsity pattern of the graph in Figure \ref{example_chordal_graph}(c). Then, we have 

     \[
\begin{array}{c c}
   \quad \quad ~ \color{red} 1   ~~ \color{red} 2 ~~~    \color{red} 3   ~~ \color{red} 4    ~~ \color{red} 5 ~~      \color{red} 6 ~~ \color{red} 7 ~~   \color{red} 8 \\
\bm{M}=
\begin{pmatrix}
   \times & \times & 0 & \times & \times & 0 & 0 & 0 \\ 
      \times & \times & \times & \times & 0 & 0 & 0 & 0 \\ 
     0 & \times & \times & \times & 0 & 0 & 0 & 0 \\ 
      \times & \times & \times & \times & \times & 0 & 0 & 0 \\ 
      \times & 0 & 0 & \times & \times & 0 & 0 & 0 \\ 
     0 & 0 & 0 & 0 & 0 & \times & \times  & 0 \\ 
      0 & 0 & 0 & 0 & 0 & \times & \times  & \times \\ 
    0 & 0 & 0 & 0 & 0 & 0 & \times  & \times 
\end{pmatrix} &
\begin{matrix}
   \color{red} 1 \\
   \color{red} 2 \\
   \color{red} 3 \\
    \color{red} 4 \\
   \color{red} 5 \\
   \color{red} 6 \\
   \color{red} 7 \\
   \color{red} 8
\end{matrix}
\end{array}.
\]
Requiring $\bm{M} \succeq 0 $ is equivalent to $\bm{M}_{C_i} \succeq 0$ $\forall i \in \lbrace 1, 2, 3, 4, 5 \rbrace $, where 
   \[
\begin{array}{c c}
   \quad \quad \quad ~ \color{red} 1   ~~ \color{red} 4 ~~    \color{red} 5   \\
\bm{M}_{C_{1}}=
\begin{pmatrix}
   \times & \times & \times  \\ 
 \times & \times & \times  \\ 
   \times & \times & \times   
\end{pmatrix} &
\begin{matrix}
   \color{red} 1 \\
   \color{red} 4 \\
   \color{red} 5 
\end{matrix}
\end{array}
\begin{array}{c c}
   \quad \quad \quad ~ \color{red} 1   ~~ \color{red} 2 ~~    \color{red} 4   \\
\bm{M}_{C_{2}}=
\begin{pmatrix}
   \times & \times & \times  \\ 
 \times & \times & \times  \\ 
   \times & \times & \times   
\end{pmatrix} &
\begin{matrix}
   \color{red} 1 \\
   \color{red} 2 \\
   \color{red} 4 
\end{matrix}
\end{array} \begin{array}{c c}
   \quad \quad \quad ~ \color{red} 2   ~~ \color{red} 3 ~~    \color{red} 4   \\
\bm{M}_{C_{3}}=
\begin{pmatrix}
   \times & \times & \times  \\ 
 \times & \times & \times  \\ 
   \times & \times & \times   
\end{pmatrix} &
\begin{matrix}
   \color{red} 2 \\
   \color{red} 3 \\
   \color{red} 4 
\end{matrix}
\end{array}
\]

$$ \begin{array}{c c}
   \quad \quad \quad ~ \color{red} 6   ~~ \color{red} 7  \\
\bm{M}_{C_{4}}=
\begin{pmatrix}
   \times & \times \\
 \times & \times  
\end{pmatrix} &
\begin{matrix}
   \color{red} 6 \\
   \color{red} 7
\end{matrix}
\end{array} \begin{array}{c c}
   \quad \quad \quad ~ \color{red} 7   ~~ \color{red} 8 \\
\bm{M}_{C_{5}}=
\begin{pmatrix}
   \times & \times \\
 \times & \times 
\end{pmatrix} &
\begin{matrix}
   \color{red} 7 \\
   \color{red} 8
\end{matrix}
\end{array} $$

\end{example}

In order to efficiently exploit term sparsity of a POP using graphs, we represent the polynomials defining the POP by their support. In the next subsection, we recall the definition of the support of a polynomial and generalize it to the matrix polynomial case.  

\subsection{Support of a polynomial}

The \emph{support of a polynomial} is the set of monomial exponents defining the polynomial. More concretely, for $\bm{x} \in \mathbb{R}^n$ and $f(\bm{x})=\sum_{\bm{\alpha} \in \mathbb{N}^n_r}f_{\bm{\alpha}}\bm{x^\alpha}$, the support of $f$ is defined as $$\text{supp}\left(f\right):= \lbrace \bm{\alpha} \in \mathbb{N}^n_r: f_{\bm{\alpha}} \neq 0\rbrace. $$
For a better understanding, we will abuse notation in the sequel and refer to the support of a polynomial as 
\begin{equation}
\text{supp}\left(f\right)= \lbrace \bm{x^\alpha} \in \bm{b}_r(\bm{x}): f_{\bm{\alpha}} \neq 0\rbrace. 
\label{support_function}
\end{equation}

Now, given a polynomial matrix $\bm{G}_j(\bm{x})=\sum_{\alpha \in \mathbb{N}^n_r}\bm{G}_{j,\bm{\alpha}}\bm{x^\alpha}$, $j \in \lbrace 1,\ldots, m \rbrace$ and where $\bm{G}_{j,\bm{\alpha}} \in \mathbb{S}^{{s_j}\times {s_j}} $, the previous definition generalizes naturally to this case as 
\begin{equation}
\text{supp}\left(\bm{G}_j\right)= \lbrace \bm{x^\alpha} \in \bm{b}_r(\bm{x}): \bm{G}_{j,\bm{\alpha}} \neq \bm{0}_{\mathbb{R}^{{s_j}\times {s_j}}}\rbrace. 
\label{support_matrix}
\end{equation}

The definitions \eqref{support_function} and \eqref{support_matrix} are unchanged and lead to the same result if we replace the standard monomial basis $\bm{b}_r(\bm{x})$ by any other smaller monomial basis, which we will do in Section \ref{sec:reducedbasis}.

\subsection{Term sparsity procedure}
\label{sec:mainTSP}
Let $r$ be a fixed-degree relaxation of Problem \eqref{POP1}. We start by defining the set of initial support as 
\begin{equation}
\mathcal{S}:=\text{supp}(f) \cup \bigcup_{j=1}^{m} \text{supp}(\bm{G}_j) \cup \bm{b}_r(\bm{x}^2),
\label{initial_support}
\end{equation}
where\footnote{The set $\bm{b}_r(\bm{x}^2)$ is added to the initial support defined in \eqref{initial_support} to secure that the diagonal entries of the \textit{a priori} sparse moment matrix do not vanish.} $\bm{b}_r(\bm{x}^2)$ is the vector of squares of each element in $\bm{b}_r(\bm{x})$.
Let $\mathcal{G}^{\text{tsp}}$ be a graph with $V\left( \mathcal{G}^{\text{tsp}}\right)=\bm{b}_r(\bm{x})$ and 
\begin{equation}
E\left(\mathcal{G}^{\text{tsp}}\right):= \Big\lbrace \left\lbrace \bm{x^\alpha}, \bm{x^\beta} \right\rbrace: \bm{x}^{\bm{\alpha} + \bm{\beta}} \in \mathcal{S} \Big\rbrace.
    \label{TSP_graph}
\end{equation}
In other words, we add an edge between two monomials in the basis $\bm{b}_r(\bm{x})$ if their product is equal to a monomial involved in the expression of polynomials defining Problem \eqref{POP1}. This graph is called the TSP graph of the problem and schematizes the link between the monomial basis and the monomial terms used to build \eqref{POP1}: the fewer edges it contains, the sparser is the problem. In general, as long as this graph is not complete, one can still exploit term sparsity. 

Now, let $ \bm{G}_0=1$ and $d_0=0$. We define a recursive sequence of graph embeddings $\mathcal{G}_{r,j}^{(k)}$ indexed by the sparsity order $k \in \mathbb{N}_{0}$, relaxation order $r$, and $j \in \lbrace 0, \ldots, m \rbrace$. We start with $\mathcal{G}_{r,0}^{(0)}=\mathcal{G}^{\text{tsp}}$, $ \mathcal{S}^{(0)}=\mathcal{S}$ and $ \mathcal{G}_{r,j}^{(0)} = \varnothing$, $ \forall j \in \lbrace 1, \ldots, m \rbrace$. For $k \geq 1$ and for each $j \in  \lbrace 0, \ldots, m \rbrace$, we perform two successive operations:

\begin{enumerate}
\item \textbf{Support extension.} Define the graph $\mathcal{F}_{r,j}^{(k)}$ with $V \left( \mathcal{F}_{r,j}^{(k)}\right)=\bm{b}_{r-d_{j}}(\bm{x})$ and

 \begin{equation}
E \left( \mathcal{F}_{r,j}^{(k)}\right):= \Big\lbrace  \lbrace \bm{x^\alpha}, \bm{x^\beta}\rbrace: \exists \bm{x^\gamma} \in  \text{supp}(\bm{G}_j) \text{ s.t } \bm{x}^{\bm{\alpha}+\bm{\beta}+\bm{\gamma}}  \in \mathcal{S}^{(k-1)}   \Big\rbrace
\label{support_extention}
\end{equation}
and the set $\mathcal{S}^{(k-1)}$ is defined by \eqref{extended_support} for $k \geq 2$.

\item \textbf{Chordal extension.} Set the next iteration by computing a chordal extension \begin{equation}
\mathcal{G}_{r,j}^{(k)}=\overline{\mathcal{F}_{r,j}^{(k)}}
\label{chordal_extension}
\end{equation} 
and updating the support
\begin{equation}
\mathcal{S}^{(k)}:= \bigcup_{j=0}^{m} \Big\lbrace  \bm{x}^{\bm{\alpha}+\bm{\beta}+\bm{\gamma}}: \bm{x^\gamma} \in  \text{supp}(G_j) \text{ and }  \lbrace \bm{x^\alpha}, \bm{x^\beta}\rbrace \in E \left( \mathcal{G}_{r,j}^{(k)}\right) \Big\rbrace.
\label{extended_support}
\end{equation}
\end{enumerate}

The support extension operation defined in \eqref{support_extention} is related to the way we build the moment matrix and the localizing matrices shown in Section \ref{sec:preliminaries}. This definition generalizes the operation \eqref{TSP_graph} to handle also the graphs $\mathcal{G}_{r,j}^{(k)}$ for $j \in  \lbrace 1, \ldots, m \rbrace$. When $j=0$, the support extension step means that, if a product of two monomials is equal to a monomial in $\mathcal{S}^{(k)}$, we add an edge between them, as in \eqref{TSP_graph}. When $j \in  \lbrace 1, \ldots, m \rbrace$, it means that if a product of two monomials from the basis by a monomial in $\text{supp}(\bm{G}_j)$ is equal to a monomial in $\mathcal{S}^{(k)}$, we add an edge.  

At each step $k$, the set defined in \eqref{extended_support} represents the accumulated monomial terms involved in each entry of the \textit{a priori} sparse moment and localizing matrix, before applying the Riesz operator as in Definition \ref{def-moment-matrix}.

For each $j \in \lbrace 0, \ldots, m \rbrace$, we obtain a sequence of embedded graphs $\mathcal{G}^{(k)}_{r,j} \subseteq \mathcal{G}_{r,j}^{(k+1)}$, $k=0,1,\ldots$ that always stabilizes (settles into a fixed value) after a few iterations. Indeed, in the worst case, this sequence leads to a complete graph, which corresponds to the dense relaxation hierarchy \eqref{lassereSOS}.

Now, for a fixed sparsity order $k$ and a relaxation order $r$, let $\left\lbrace C_{r,j,c}^{(k)} \right\rbrace_{c=1}^{\rho_{d,j}}$ be maximal cliques of $\mathcal{G}_{r,j}^{(k)}$ for each $j \in \lbrace 0, \ldots, m \rbrace$. Then, the term-sparse mSOS for \eqref{lassereSOS} reads as
\begin{equation}
\begin{aligned}
  f_{r}^{(k)}=&\underset{\bm{y} \in \mathbb{R}^{h}}{\inf}  L_{\bm{y}}\left(f \right),  \\
 \text{s.t. }  & \left[\bm{M}_{r-d_j}(\bm{G}_j \bm{y})\right]_{C_{d,j,c}^{(k)}} \succeq 0,  ~\forall c \in \lbrace 1 , \ldots, \rho_{d,j} \rbrace,~ \forall j \in \lbrace 1,\ldots, m \rbrace,\\
&   \left[\bm{M}_r(\bm{y})\right]_{C_{d,0,c}^{(k)}} \succeq 0, ~ \forall c \in \lbrace 1, \ldots, \rho_{d,0} \rbrace, \\
&   y_{\bm{0}}=1,   
\end{aligned}
\label{TS_lassereSOS}
\end{equation}
where $h \leq \vert \bm{b}_{2r}(\bm{x}) \vert$.
By solving \eqref{TS_lassereSOS} at each relaxation order $r$ and at each sparsity order $k$, we obtain a bidirectional hierarchy of lower bounds, i.e., for $\tilde{r} \geq r$, we have 
 \begin{small}
 \begin{equation}
  \begin{aligned}
 f_r^{(1)} & \leq  &  f_r^{(2)} & \leq & \ldots  & \leq &  f_r \\
  \rotatebox[origin=c]{90}{$\geq$} & &  \rotatebox[origin=c]{90}{$\geq$} & & & &  \rotatebox[origin=c]{90}{$\geq$} \\
   f_{r+1}^{(1)} & \leq  &  f_{r+1}^{(2)} & \leq & \ldots  & \leq &  f_{r+1} \\
    \rotatebox[origin=c]{90}{$\geq$} & &  \rotatebox[origin=c]{90}{$\geq$} & & & &  \rotatebox[origin=c]{90}{$\geq$} \\
      \vdots & &  \vdots  & &\vdots  & & \vdots   \\ 
     f_{\tilde{r}}^{(1)} & \leq  &  f_{\tilde{r}}^{(2)} & \leq & \ldots  & \leq &  f_{\tilde{r}} \\    
         \rotatebox[origin=c]{90}{$\geq$} & &  \rotatebox[origin=c]{90}{$\geq$} & & & &  \rotatebox[origin=c]{90}{$\geq$} \\
      \vdots & &  \vdots  & &\vdots  & & \vdots       
 \end{aligned}
 \label{hierarchy_TSSOS}
 \end{equation}
 \end{small}

Indeed, at each fixed relaxation $r$ and a fixed sparsity order $k$, we solve a ``relaxation of a relaxation'', because removing entries (components of the moment vector) in the problem (\ref{TS_lassereSOS}) leads to a smaller feasible set than in (\ref{lassereSOS}). By letting $r$ fixed and varying $k$, the feasible sets of (\ref{TS_lassereSOS}) become tighter as the sparsity graphs become denser.

Because of \eqref{hierarchy_TSSOS}, it may be beneficial to perform successive TSP operations until the graph $\mathcal{G}^{(k)}_{r,j}$ stabilizes and then solve \eqref{TS_lassereSOS}. However, the graph $\mathcal{G}^{(k)}_{r,j}$ can be complete after few iterations and, therefore, \eqref{TS_lassereSOS} is no longer sparse (see Figures \ref{figure-TSP-maximalchordalextension-graph-example} and \ref{figure-TSP-minimalchordalextension-graph-example}).

The quality of lower bounds and the sparsity structure necessarily depend on the choice of the chordal extension algorithm. The chordal extension is used mainly to allow the decomposition through cliques stated in Section \ref{sec:graphs-psd}. The extension can be approximately minimal \cite{wang2020chordal} or maximal (block completion) \cite{wang2021tssos}; but we can also use any other kind of chordal extension. However, the chosen type of chordal extension will influence:
\begin{itemize}
\item The computational complexity of the sparse problem \eqref{TS_lassereSOS}. The approximately minimal chordal extension produces cliques smaller than the maximal chordal extension; see \cite{wang2020chordal} for numerical examples comparing both strategies in the scalar polynomial case. 

\item The gap between $f_{r}^{(k)}$ and $f_r$ in each fixed relaxation order $r$. If the maximal chordal extension is used, then the convergence of the sequence $(f_{r}^{(k)})_k$ to $f_r$ at each relaxation order $r$ is guaranteed \cite{wang2021tssos}. This result does not hold for approximately minimal chordal extension even if numerical convergence occurs in many cases, as shown in \cite{wang2020chordal}.
\end{itemize}

\subsection{Term sparsity in the optimization of frame structures }
Finally, we describe the term sparsity technique when applied to the frame structures problems as introduced in Section \ref{sec:frame_structures}. To make the exposition short, we focus on the compliance minimization problem since it has more variables than the weight minimization problem but the same polynomial structure. We denote by $\mathcal{S}_{\eqref{compliance_problem_sdp_scalled}}$ the support of \eqref{compliance_problem_sdp_scalled}, i.e., the union of the support of its objective function and of all its constraints. To simplify the notation in the sequel, we set $\bm{x}=(\bm{a}_s, c_\mathrm{s}) \in \mathbb{R}^n$ where $n=n_{\bm{a}}+1$. Let $\bm{K}$ be the stiffness matrix defined in \eqref{stifness_matrix}. We consider two cases:
\begin{enumerate}
    \item if $\forall e$: $\bm{K}_e^{(3)}=\bm{0}$, the support reads as
    \begin{equation}
            \mathcal{S}_{\eqref{compliance_problem_sdp_scalled}}= \left\lbrace 1, x_1,\ldots,x_n,x_1^2,\ldots, x_n^2 \right\rbrace;
            \label{support_compliance_K3=0}
    \end{equation}

      \item if $\forall e$: $\bm{K}_e^{(3)}\neq \bm{0}$, the support reads as \begin{equation}
          \mathcal{S}_{\eqref{compliance_problem_sdp_scalled}}= \left\lbrace 1, x_1,\ldots,x_n,x_1^2,\ldots, x_n^2,x_1^3,\ldots, x_n^3 \right\rbrace.
          \label{support_compliance_K3!=0}
      \end{equation}
\end{enumerate} 
As a first observation, all polynomials that define the problem \eqref{compliance_problem_sdp_scalled} are separable, i.e., they can be written as a sum of univariate polynomials. This makes the polynomials dependent on non-mixed monomial terms\footnote{In other words, monomials of the form $x_1^{\alpha_{1}}x_2^{\alpha_{2}} \ldots x_n^{\alpha_{n}}$ where only one exponent is not equal to zero.} only. Therefore, \eqref{compliance_problem_sdp_scalled} is very sparse in the perspective of monomial terms, see Figure \ref{figure-TSP-supportextension-graph-example}. Consequently, we have $$\vert \mathcal{S}_{\eqref{compliance_problem_sdp_scalled}} \vert = nd+1 < \vert \bm{b}_d(\bm{x}) \vert = \frac{(n+d)!}{n!d!},  $$
where $d=2$ in the case of \eqref{support_compliance_K3=0} and $d=3$ in the case of \eqref{support_compliance_K3!=0}. This special polynomial structure inspired us to use a non-mixed monomial basis in Section \ref{sec:reducedbasis}.  

\begin{example}
    For $r=2$, the initial TSP support for the weight minimization problem \eqref{POP-example} reads as
    \begin{align*}
        \mathcal{S}^{(0)}&=\mathcal{S}_{\eqref{POP-example}} \cup \bm{b}_2(\bm{x}^2) \\
        &= \left\lbrace 1,x_1,x_2,x_3,x_1^2,x_2^2,x_3^2,x_1^4,x_2^4,x_3^4,x_1^2x_2^2,x_1^2x_3^2,x_2^2x_3^2 \right\rbrace.
    \end{align*}
    The associated TSP graph appears visualized in Figure \ref{figure-TSP-supportextension-graph-example}.
   
\end{example}

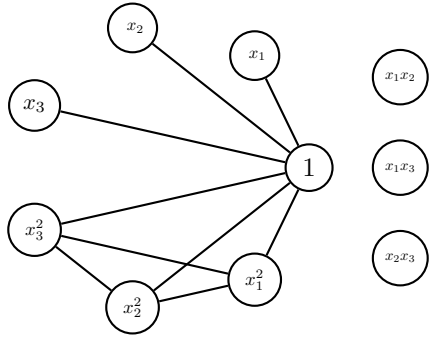
\begin{figure}[!h]

\centering \begin{tikzpicture}[node distance={12mm}, thick, main/.style = {draw, circle, radius=0.1mm}, auto]
\node[main] at (360/7* 1:1.9cm ) (1) {\adjustbox{max width=3mm}{ $x_1$}}; 
\node[main] at (360/7* 2:1.9cm ) (2) {\adjustbox{max width=3mm}{ $x_2$}}; 
\node[main] at (360/7* 3:1.9cm ) (3) {\adjustbox{max width=3mm}{$x_3$}}; 
\node[main] at (360/7* 4:1.9cm ) (4) {\adjustbox{max width=2.5mm}{$x_3^2$}}; 
\node[main] at (360/7* 5:1.9cm )(5) {\adjustbox{max width=2.5mm}{$x_2^2$}}; 
\node[main] at (360/7* 6:1.9cm )(6) {\adjustbox{max width=2.5mm}{$x_1^2$}}; 
\node[main]  at (360/7* 7:1.9cm ) (7) {\adjustbox{max width=3mm}{$1$}}; 

\node[main]  [  right of=7] (8) {\adjustbox{max width=4mm}{$x_1x_3$}}; 
\node[main] [  above of=8] (9) {\adjustbox{max width=4mm}{$x_1x_2$}}; 
\node[main] [  below of=8]  (10) {\adjustbox{max width=4mm}{$x_2x_3$}}; 
\draw (7) -- (1);
\draw (7) -- (2);
\draw (7) -- (3);
\draw (7) -- (4);
\draw (7) -- (5);
\draw (7) -- (6);
\draw (4) -- (5);
\draw (4) -- (6);
\draw (5) -- (6);
\end{tikzpicture}  

\caption{The TSP graph of the illustrative problem \eqref{POP-example} with a relaxation degree $r=2$.}
\label{figure-TSP-supportextension-graph-example}
\end{figure}

\subsubsection{TSP with maximal chordal extension}
As stated in Section \ref{sec:mainTSP}, using the TSP with the maximal chordal extension leads to theoretical convergence when the graphs stabilize after a finite number of sparsity order steps. However, in our application, using TSP with maximal chordal extension leads, after few iterations, to complete graphs and, therefore, to dense SDP relaxations. This claim is shown in Figure \ref{figure-TSP-maximalchordalextension-graph-example}: in Figure \ref{figure-TSP-maximalchordalextension-graph-example}(a) the support extension operation is empty, therefore the graph $\mathcal{G}_{2,0}^{(0)}$ is identical to its support extension $\mathcal{F}_{2,0}^{(1)}$; the support extension of $\mathcal{G}_{2,0}^{(1)}$ in Figure \ref{figure-TSP-maximalchordalextension-graph-example}(c) makes the graph $\mathcal{F}_{2,0}^{(2)}$ connected and, therefore, after the maximal chordal extension, the graph $\mathcal{G}_{2,0}^{(2)}$ becomes complete. More generally, we have the following proposition:
\begin{proposition}
\label{prop_maximal_extension}
Let us consider Problem \eqref{compliance_problem_sdp_scalled}. For any relaxation order $r \geq 2$, the sequence of graphs $\mathcal{G}^{(k)}_{r,j}, \forall j \in  \lbrace 0,\ldots, m \rbrace$, obtained by TSP with maximal chordal extension are complete for $k=2$.
\end{proposition}

  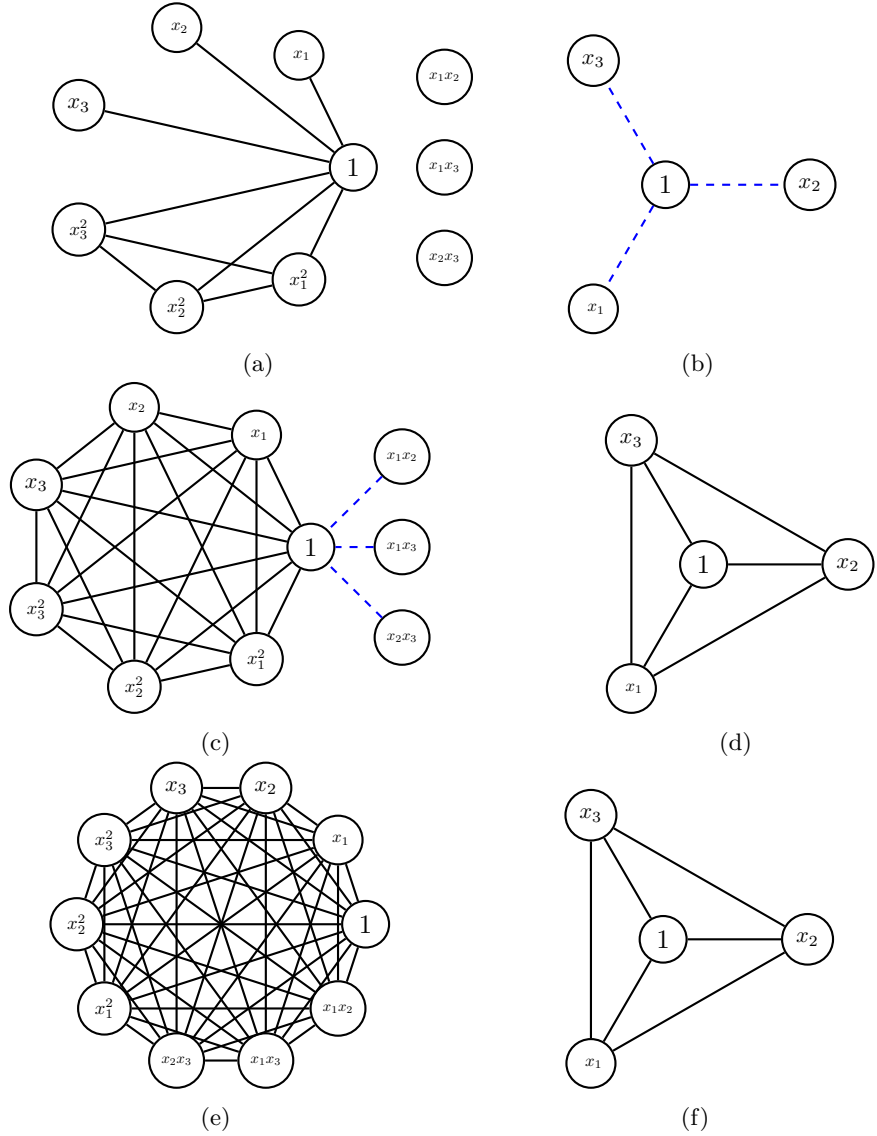
\begin{figure}[htbp!]

\centering
\subfloat[]{\begin{tikzpicture}[node distance={12mm}, thick, main/.style = {draw, circle, radius=0.1mm}, auto]
\node[main] at (360/7* 1:1.9cm ) (1) {\adjustbox{max width=3mm}{ $x_1$}}; 
\node[main] at (360/7* 2:1.9cm ) (2) {\adjustbox{max width=3mm}{ $x_2$}}; 
\node[main] at (360/7* 3:1.9cm ) (3) {\adjustbox{max width=3mm}{$x_3$}}; 
\node[main] at (360/7* 4:1.9cm ) (4) {\adjustbox{max width=2.5mm}{$x_3^2$}}; 
\node[main] at (360/7* 5:1.9cm )(5) {\adjustbox{max width=2.5mm}{$x_2^2$}}; 
\node[main] at (360/7* 6:1.9cm )(6) {\adjustbox{max width=2.5mm}{$x_1^2$}}; 
\node[main]  at (360/7* 7:1.9cm ) (7) {\adjustbox{max width=3mm}{$1$}}; 

\node[main]  [  right of=7] (8) {\adjustbox{max width=4mm}{$x_1x_3$}}; 
\node[main] [  above of=8] (9) {\adjustbox{max width=4mm}{$x_1x_2$}}; 
\node[main] [  below of=8]  (10) {\adjustbox{max width=4mm}{$x_2x_3$}}; 
\draw (7) -- (1);
\draw (7) -- (2);
\draw (7) -- (3);
\draw (7) -- (4);
\draw (7) -- (5);
\draw (7) -- (6);
\draw (4) -- (5);
\draw (4) -- (6);
\draw (5) -- (6);
\end{tikzpicture}  } \hspace*{1cm} 
\subfloat[]{\begin{tikzpicture}[node distance={12mm}, thick, main/.style = {draw, circle, radius=0.1mm}, auto]
\node[main] at (360/3* 2:1.9cm ) (2) {\adjustbox{max width=3mm}{ $ x_1$}}; 
\node[main] at (360/3* 3:1.9cm ) (3) {\adjustbox{max width=3mm}{$x_2$}}; 
\node[main] at (360/3* 4:1.9cm ) (4) {\adjustbox{max width=3mm}{$x_3$}}; 
\node[main]  at (0,0) (1) {\adjustbox{max width=3mm}{$1$}}; 
\draw[blue,dashed] (1) -- (2);
\draw[blue,dashed] (1) -- (3);
\draw[blue,dashed] (1) -- (4);

\end{tikzpicture}  } 

\centering \subfloat[]{\begin{tikzpicture}[node distance={12mm}, thick, main/.style = {draw, circle}, auto]
\node[main] at (360/7* 1:1.9cm ) (1) {\adjustbox{max width=3mm}{ $x_1$}}; 
\node[main] at (360/7* 2:1.9cm ) (2) {\adjustbox{max width=3mm}{ $x_2$}}; 
\node[main] at (360/7* 3:1.9cm ) (3) {\adjustbox{max width=3mm}{$x_3$}}; 
\node[main] at (360/7* 4:1.9cm ) (4) {\adjustbox{max width=2.5mm}{$x_3^2$}}; 
\node[main] at (360/7* 5:1.9cm )(5) {\adjustbox{max width=2.5mm}{$x_2^2$}}; 
\node[main] at (360/7* 6:1.9cm )(6) {\adjustbox{max width=2.5mm}{$x_1^2$}}; 
\node[main]  at (360/7* 7:1.9cm ) (7) {\adjustbox{max width=3mm}{$1$}}; 

\node[main]  [  right of=7] (8) {\adjustbox{max width=4mm}{$x_1x_3$}}; 
\node[main] [  above of=8] (9) {\adjustbox{max width=4mm}{$x_1x_2$}}; 
\node[main] [  below of=8]  (10) {\adjustbox{max width=4mm}{$x_2x_3$}}; 

\draw (1) -- (2);
\draw (1) -- (3);
\draw (1) -- (4);
\draw (1) -- (5);
\draw (1) -- (6);
\draw (1) -- (7);
\draw (2) -- (3);
\draw (2) -- (4);
\draw (2) -- (5);
\draw (2) -- (6);
\draw (2) -- (7);
\draw (3) -- (4);
\draw (3) -- (5);
\draw (3) -- (6);
\draw (3) -- (7);
\draw (4) -- (5);
\draw (4) -- (6);
\draw (4) -- (7);
\draw (5) -- (6);
\draw (5) -- (7);
\draw (6) -- (7);
\draw[blue,dashed] (8) -- (7);
\draw[blue,dashed] (9) -- (7);
\draw[blue,dashed] (10) -- (7);
\end{tikzpicture}} \hspace{1cm}
\hspace*{1cm} \subfloat[]{\begin{tikzpicture}[node distance={12mm}, thick, main/.style = {draw, circle, radius=0.1mm}, auto]
\node[main] at (360/3* 2:1.9cm ) (2) {\adjustbox{max width=3mm}{ $ x_1$}}; 
\node[main] at (360/3* 3:1.9cm ) (3) {\adjustbox{max width=3mm}{$x_2$}}; 
\node[main] at (360/3* 4:1.9cm ) (4) {\adjustbox{max width=3mm}{$x_3$}}; 
\node[main]  at (0,0) (1) {\adjustbox{max width=3mm}{$1$}}; 
\draw (1) -- (2);
\draw (1) -- (3);
\draw (1) -- (4);
\draw (2) -- (3);
\draw (2) -- (4);
\draw (3) -- (4);
\end{tikzpicture}  } \\ \subfloat[]{\begin{tikzpicture}[node distance={12mm}, thick, main/.style = {draw, circle}, auto]
\node[main] at (360/10* 1:1.9cm ) (2) {\adjustbox{max width=3mm}{ $ x_1$}}; 
\node[main] at (360/10* 2:1.9cm ) (3) {\adjustbox{max width=3mm}{$x_2$}}; 
\node[main] at (360/10* 3:1.9cm ) (4) {\adjustbox{max width=3mm}{$x_3$}}; 
\node[main] at (360/10* 4:1.9cm ) (10) {\adjustbox{max width=2.5mm}{$x_3^2$}} ; 
\node[main] at (360/10* 5:1.9cm )(8) {\adjustbox{max width=2.5mm}{$x_2^2$}}; 
\node[main] at (360/10* 6:1.9cm )(5) {\adjustbox{max width=2.5mm}{$x_1^2$}}; 
\node[main]  at (360/10* 10:1.9cm ) (1) {\adjustbox{max width=3mm}{$1$}}; 

\node[main]  at (360/10* 8:1.9cm ) (7) {\adjustbox{max width=4mm}{$x_1x_3$}}; 
\node[main] at (360/10* 9:1.9cm ) (6) {\adjustbox{max width=4mm}{$x_1x_2$}}; 
\node[main] at (360/10* 7:1.9cm )  (9) {\adjustbox{max width=4mm}{$x_2x_3$}};  

\draw (1) -- (2);
\draw (1) -- (3);
\draw (1) -- (4);
\draw (1) -- (5);
\draw (1) -- (6);
\draw (1) -- (7);
\draw (1) -- (8);
\draw (1) -- (9);
\draw (1) -- (10);
\draw (2) -- (3);
\draw (2) -- (4);
\draw (2) -- (5);
\draw (2) -- (6);
\draw (2) -- (7);
\draw (2) -- (8);
\draw (2) -- (9);
\draw (2) -- (10);
\draw (3) -- (4);
\draw (3) -- (5);
\draw (3) -- (6);
\draw (3) -- (7);
\draw (3) -- (8);
\draw (3) -- (9);
\draw (3) -- (10);
\draw (4) -- (5);
\draw (4) -- (6);
\draw (4) -- (7);
\draw (4) -- (8);
\draw (4) -- (9);
\draw (4) -- (10);
\draw (5) -- (6);
\draw (5) -- (7);
\draw (5) -- (8);
\draw (5) -- (9);
\draw (5) -- (10);
\draw (6) -- (7);
\draw (6) -- (8);
\draw (6) -- (9);
\draw (6) -- (10);
\draw (7) -- (8);
\draw (7) -- (9);
\draw (7) -- (10);
\draw  (8) -- (9);
\draw  (8) -- (10);
\draw  (9) -- (10);
\end{tikzpicture}} \hspace{1cm}
\hspace*{1cm} \subfloat[]{\begin{tikzpicture}[node distance={12mm}, thick, main/.style = {draw, circle, radius=0.1mm}, auto]
\node[main] at (360/3* 2:1.9cm ) (2) {\adjustbox{max width=3mm}{ $ x_1$}}; 
\node[main] at (360/3* 3:1.9cm ) (3) {\adjustbox{max width=3mm}{$x_2$}}; 
\node[main] at (360/3* 4:1.9cm ) (4) {\adjustbox{max width=3mm}{$x_3$}}; 
\node[main]  at (0,0) (1) {\adjustbox{max width=3mm}{$1$}}; 
\draw (1) -- (2);
\draw (1) -- (3);
\draw (1) -- (4);
\draw (2) -- (3);
\draw (2) -- (4);
\draw (3) -- (4);
\end{tikzpicture}  }

\caption{TSP with maximal chordal extension for the illustrative problem \eqref{POP-example} for $j=0, 1$: (a) the graph $\mathcal{F}_{2,0}^{(1)}$, (b) the graph $\mathcal{F}_{2,1}^{(1)}$, (c) the graph $\mathcal{F}_{2,0}^{(2)}$, (d) the graph $\mathcal{F}_{2,1}^{(2)}$, (e) the graph $\mathcal{G}_{2,0}^{(2)}$, and (f) the graph $\mathcal{G}_{2,1}^{(2)}$. The support extensions are drawn in dashed blue lines.}
\label{figure-TSP-maximalchordalextension-graph-example}
\end{figure}

A detailed proof of Proposition \ref{prop_maximal_extension} can be found in Appendix \ref{proof-prop}. In numerical implementation, we thus stop the TSP with maximal chordal extension in sparsity order $k=1$. The theoretical convergence is no longer guaranteed in this case; nevertheless, numerical convergence holds in all numerical examples in Section \ref{sec:numerical-experience}.

\subsubsection{TSP with minimal chordal extension}
In many cases, the TSP with minimal chordal extension will also lead to complete graphs after a stabilized sparsity order $k$. This is shown in Figure \ref{figure-TSP-minimalchordalextension-graph-example}: the support extension operation in Figure \ref{figure-TSP-minimalchordalextension-graph-example}(a) is empty and the obtained graph $\mathcal{F}_{2,0}^{(1)}$ is minimal chordal so that $\mathcal{G}_{2,0}^{(1)}=\mathcal{F}_{2,0}^{(1)}$; the graph $\mathcal{F}_{2,0}^{(2)}$ is also minimal chordal, so that only edges corresponding to the support extensions are added in Figure \ref{figure-TSP-minimalchordalextension-graph-example}(c) leading to $\mathcal{G}_{2,0}^{(2)}=\mathcal{F}_{2,0}^{(2)}$. The same holds for the graph $\mathcal{F}_{2,0}^{(3)}$, we only add the support extension edges as shown in Figure \ref{figure-TSP-minimalchordalextension-graph-example}(e). The obtained graph $\mathcal{G}_{2,0}^{(3)}$ is complete. However, this property does not always hold.

When the problem has many variables and/or requires a high relaxation order to be solved, using the approximate chordal extension at a sparsity order $k \geq 3$ can lead to large number of cliques and/or cliques of large sizes.  In some cases, this can diminish the computational benefit of using TSP with minimal chordal extension.

In our numerical implementation, we stop the TSP procedure at a given sparsity order $\overline{k}$ if one or more of the following criteria hold: 
\begin{itemize}
    \item if the certificate of $\varepsilon-$global optimality, described in Section \ref{certificate-global}, is satisfied for $\overline{k}$;
    \item if at the sparsity order $\overline{k}+1$, the graphs are complete; 
    \item if at the sparsity order $\overline{k}+1$, there is a large number of obtained cliques with large sizes (see Table \ref{tab:14-elements-compliance-TSP-min-ch-k=2}).
\end{itemize}

\begin{figure}[!h]

\centering
\subfloat[]{\begin{tikzpicture}[node distance={12mm}, thick, main/.style = {draw, circle, radius=0.1mm}, auto]
\node[main] at (360/7* 1:1.9cm ) (1) {\adjustbox{max width=3mm}{ $x_1$}}; 
\node[main] at (360/7* 2:1.9cm ) (2) {\adjustbox{max width=3mm}{ $x_2$}}; 
\node[main] at (360/7* 3:1.9cm ) (3) {\adjustbox{max width=3mm}{$x_3$}}; 
\node[main] at (360/7* 4:1.9cm ) (4) {\adjustbox{max width=2.5mm}{$x_3^2$}}; 
\node[main] at (360/7* 5:1.9cm )(5) {\adjustbox{max width=2.5mm}{$x_2^2$}}; 
\node[main] at (360/7* 6:1.9cm )(6) {\adjustbox{max width=2.5mm}{$x_1^2$}}; 
\node[main]  at (360/7* 7:1.9cm ) (7) {\adjustbox{max width=3mm}{$1$}}; 

\node[main]  [  right of=7] (8) {\adjustbox{max width=4mm}{$x_1x_3$}}; 
\node[main] [  above of=8] (9) {\adjustbox{max width=4mm}{$x_1x_2$}}; 
\node[main] [  below of=8]  (10) {\adjustbox{max width=4mm}{$x_2x_3$}}; 
\draw (7) -- (1);
\draw (7) -- (2);
\draw (7) -- (3);
\draw (7) -- (4);
\draw (7) -- (5);
\draw (7) -- (6);
\draw (4) -- (5);
\draw (4) -- (6);
\draw (5) -- (6);

\end{tikzpicture}  } \hspace*{1cm}
\subfloat[]{\begin{tikzpicture}[node distance={12mm}, thick, main/.style = {draw, circle, radius=0.1mm}, auto]
\node[main] at (360/3* 2:1.9cm ) (2) {\adjustbox{max width=3mm}{ $ x_1$}}; 
\node[main] at (360/3* 3:1.9cm ) (3) {\adjustbox{max width=3mm}{$x_2$}}; 
\node[main] at (360/3* 4:1.9cm ) (4) {\adjustbox{max width=3mm}{$x_3$}}; 
\node[main]  at (0,0) (1) {\adjustbox{max width=3mm}{$1$}}; 
\draw[blue,dashed] (1) -- (2);
\draw[blue,dashed] (1) -- (3);
\draw[blue,dashed] (1) -- (4);

\end{tikzpicture}  }\\
\subfloat[]{\begin{tikzpicture}[node distance={12mm}, thick, main/.style = {draw, circle, radius=0.1mm}, auto]
\node[main] at (360/10* 1:1.9cm ) (2) {\adjustbox{max width=3mm}{ $ x_1$}}; 
\node[main] at (360/10* 2:1.9cm ) (3) {\adjustbox{max width=3mm}{$x_2$}}; 
\node[main] at (360/10* 3:1.9cm ) (4) {\adjustbox{max width=3mm}{$x_3$}}; 
\node[main] at (360/10* 4:1.9cm ) (10) {\adjustbox{max width=2.5mm}{$x_3^2$}} ; 
\node[main] at (360/10* 5:1.9cm )(8) {\adjustbox{max width=2.5mm}{$x_2^2$}}; 
\node[main] at (360/10* 6:1.9cm )(5) {\adjustbox{max width=2.5mm}{$x_1^2$}}; 
\node[main]  at (360/10* 10:1.9cm ) (1) {\adjustbox{max width=3mm}{$1$}}; 

\node[main]  at (360/10* 8:1.9cm ) (7) {\adjustbox{max width=4mm}{$x_1x_3$}}; 
\node[main] at (360/10* 9:1.9cm ) (6) {\adjustbox{max width=4mm}{$x_1x_2$}}; 
\node[main] at (360/10* 7:1.9cm )  (9) {\adjustbox{max width=4mm}{$x_2x_3$}}; 

\draw (1) -- (2);
\draw (1) -- (3);
\draw (1) -- (4);
\draw (1) -- (5);
\draw (1) -- (8);
\draw (1) -- (10);
\draw (5) -- (8);
\draw  (5) -- (10);
\draw  (8) -- (10);

\draw[blue,dashed] (6) -- (7);
\draw[blue,dashed] (6) -- (9);
\draw[blue,dashed] (7) -- (9);
\draw[blue,dashed] (1) -- (7);
\draw[blue,dashed] (1) -- (6);
\draw[blue,dashed] (1) -- (9);
\draw[blue,dashed] (2) -- (3);
\draw[blue,dashed] (2) -- (4);
\draw[blue,dashed] (2) -- (5);
\draw[blue,dashed] (2) -- (6);
\draw[blue,dashed] (2) -- (7);
\draw[blue,dashed] (2) -- (8);
\draw[blue,dashed] (2) -- (9);
\draw[blue,dashed] (2) -- (10);
\draw[blue,dashed] (3) -- (4);
\draw[blue,dashed] (3) -- (5);
\draw[blue,dashed] (3) -- (6);
\draw[blue,dashed] (3) -- (8);
\draw[blue,dashed] (3) -- (9);
\draw[blue,dashed] (3) -- (10);
\draw[blue,dashed] (4) -- (5);
\draw[blue,dashed] (4) -- (7);
\draw[blue,dashed] (4) -- (8);
\draw[blue,dashed] (4) -- (9);
\draw[blue,dashed] (4) -- (10);

\end{tikzpicture}  } \hspace*{1cm}
\subfloat[]{\begin{tikzpicture}[node distance={12mm}, thick, main/.style = {draw, circle, radius=0.1mm}, auto]
\node[main] at (360/3* 2:1.9cm ) (2) {\adjustbox{max width=3mm}{ $ x_1$}}; 
\node[main] at (360/3* 3:1.9cm ) (3) {\adjustbox{max width=3mm}{$x_2$}}; 
\node[main] at (360/3* 4:1.9cm ) (4) {\adjustbox{max width=3mm}{$x_3$}}; 
\node[main]  at (0,0) (1) {\adjustbox{max width=3mm}{$1$}}; 
\draw (1) -- (2);
\draw (1) -- (3);
\draw (1) -- (4);
\draw[blue,dashed] (2) -- (3);
\draw[blue,dashed] (2) -- (4);
\draw[blue,dashed] (3) -- (4);

\end{tikzpicture}  }\\
\subfloat[]{\begin{tikzpicture}[node distance={12mm}, thick, main/.style = {draw, circle, radius=0.1mm}, auto]
\node[main] at (360/10* 1:1.9cm ) (2) {\adjustbox{max width=3mm}{ $ x_1$}}; 
\node[main] at (360/10* 2:1.9cm ) (3) {\adjustbox{max width=3mm}{$x_2$}}; 
\node[main] at (360/10* 3:1.9cm ) (4) {\adjustbox{max width=3mm}{$x_3$}}; 
\node[main] at (360/10* 4:1.9cm ) (10) {\adjustbox{max width=2.5mm}{$x_3^2$}} ; 
\node[main] at (360/10* 5:1.9cm )(8) {\adjustbox{max width=2.5mm}{$x_2^2$}}; 
\node[main] at (360/10* 6:1.9cm )(5) {\adjustbox{max width=2.5mm}{$x_1^2$}}; 
\node[main]  at (360/10* 10:1.9cm ) (1) {\adjustbox{max width=3mm}{$1$}}; 

\node[main]  at (360/10* 8:1.9cm ) (7) {\adjustbox{max width=4mm}{$x_1x_3$}}; 
\node[main] at (360/10* 9:1.9cm ) (6) {\adjustbox{max width=4mm}{$x_1x_2$}}; 
\node[main] at (360/10* 7:1.9cm )  (9) {\adjustbox{max width=4mm}{$x_2x_3$}}; 

\draw (1) -- (2);
\draw (1) -- (3);
\draw (1) -- (4);
\draw (1) -- (5);
\draw (1) -- (8);
\draw (1) -- (10);
\draw (6) -- (7);
\draw (6) -- (9);
\draw (7) -- (9);
\draw (1) -- (7);
\draw (1) -- (6);
\draw (1) -- (9);
\draw(2) -- (3);
\draw (2) -- (4);
\draw (2) -- (5);
\draw (2) -- (6);
\draw (2) -- (7);
\draw (2) -- (8);
\draw (2) -- (9);
\draw (2) -- (10);
\draw (3) -- (4);
\draw (3) -- (5);
\draw (3) -- (6);
\draw (3) -- (8);
\draw (3) -- (9);
\draw (3) -- (10);
\draw (4) -- (5);
\draw (4) -- (7);
\draw (4) -- (8);
\draw (4) -- (9);
\draw (4) -- (10);
\draw (5) -- (8);
\draw (5) -- (10);
\draw (8) -- (10);
\draw[blue,dashed] (3) -- (7);
\draw[blue,dashed] (4) -- (6);
\draw[blue,dashed] (5) -- (6);
\draw[blue,dashed] (5) -- (7);
\draw[blue,dashed] (5) -- (9);
\draw[blue,dashed] (6) -- (7);
\draw[blue,dashed] (6) -- (8);
\draw[blue,dashed] (6) -- (9);
\draw[blue,dashed] (6) -- (10);
\draw[blue,dashed] (7) -- (8);
\draw[blue,dashed] (7) -- (9);
\draw[blue,dashed] (7) -- (10);
\draw[blue,dashed] (8) -- (9);
\draw[blue,dashed] (9) -- (10);

\end{tikzpicture}  }\hspace*{1cm}
\subfloat[]{\begin{tikzpicture}[node distance={12mm}, thick, main/.style = {draw, circle, radius=0.1mm}, auto]
\node[main] at (360/3* 2:1.9cm ) (2) {\adjustbox{max width=3mm}{ $ x_1$}}; 
\node[main] at (360/3* 3:1.9cm ) (3) {\adjustbox{max width=3mm}{$x_2$}}; 
\node[main] at (360/3* 4:1.9cm ) (4) {\adjustbox{max width=3mm}{$x_3$}}; 
\node[main]  at (0,0) (1) {\adjustbox{max width=3mm}{$1$}}; 
\draw (1) -- (2);
\draw (1) -- (3);
\draw (1) -- (4);
\draw (2) -- (3);
\draw (2) -- (4);
\draw (3) -- (4);
\end{tikzpicture}  }
\caption{TSP with minimal chordal extension for the illustrative problem \eqref{POP-example} for $j=0, 1$: (a) the graph $\mathcal{G}_{2,0}^{(1)}=\mathcal{F}_{2,0}^{(1)}$, (b) the graph $\mathcal{F}_{2,1}^{(1)}$, (c) the graph $\mathcal{G}_{2,1}^{(2)}=\mathcal{F}_{2,0}^{(2)}$, (d) the graph $\mathcal{F}_{2,1}^{(2)}$, (e) the graph $\mathcal{G}_{2,0}^{(3)}=\mathcal{F}_{2,0}^{(3)}$, and (f) the graph $\mathcal{G}_{2,1}^{(3)}=\mathcal{F}_{2,1}^{(3)}$. The support extensions are drawn in dashed blue lines.}

\label{figure-TSP-minimalchordalextension-graph-example}
\end{figure}
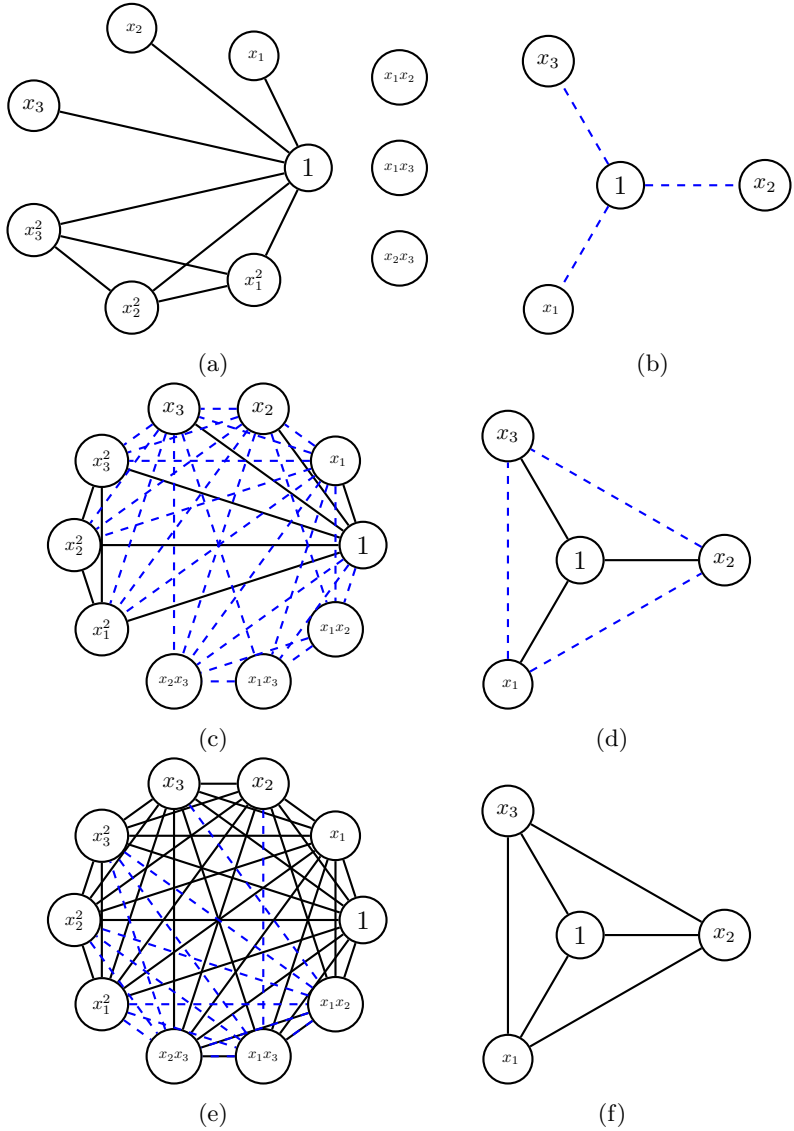

\subsubsection{The non-mixed term monomial basis}
\label{sec:reducedbasis}
Another way to exploit term sparsity in a POP is to use a reduced monomial basis to build dense SDP relaxations \eqref{lassereSOS} or sparse SDP relaxations \eqref{TS_lassereSOS}, i.e., using a smaller monomial basis $\bm{b}(\bm{x}) \subset \bm{b}_r(\bm{x})$. Such method has already been developed for unconstrained POPs and is known as the Newton polytope method \cite{reznick1978extremal,bajbar2015coercive}. However, the technique does not apply to the constrained case. For the constrained case, the authors of \cite{magron2023sparse,wang2020chordal} suggested an iterative procedure, while performing TSP for a fixed relaxation and sparsity order, to obtain a possibly smaller monomial basis.

Here, we suggest using a reduced monomial basis $\bm{b}_r^{\mathrm{NMT}}(\bm{x})$ which we call a non-mixed term (NMT) monomial basis. This basis exploits the separable nature of polynomials in the frame structure problems.

\begin{definition}[Non-mixed term monomial basis]
We define the non-mixed term monomial basis by removing all mixed terms in the basis defined in \eqref{strandard_basis}, i.e.,  \begin{equation}
\bm{b}_{r}^{\mathrm{NMT}}(\bm{x}):= \left\lbrace 1,x_1,\ldots, x_n,x_1^2, \ldots, x_n^2, \ldots, x_1^r, \ldots, x_n^r \right\rbrace.
\label{nmt_basis}
\end{equation}
\end{definition}
 The monomial basis \eqref{nmt_basis} is built easily by adding $n$ new terms at each relaxation order, and do not depend on the cliques of the TSP graphs as in \cite{magron2023sparse,wang2020chordal}. 
 
 To exploit term sparsity using the NMT basis, we can build the dense SDP relaxations \eqref{lassereSOS} using the NMT basis in the same way as shown in Section \ref{sec:lassere_mSOS}. Since the NMT basis contains fewer elements than the standard monomial basis \eqref{strandard_basis}, $\vert \bm{b_{r}}^{NMT}(\bm{x}) \vert=nr+1 < \vert \bm{b_{r}}(\bm{x})\vert $, we have smaller moment and localizing matrices. On the other hand, it also reduces the CPU time to construct the hierarchy \eqref{lassereSOS}. Furthermore, the NMT basis can be combined with the TSP technique and provide even sparser versions of \eqref{TS_lassereSOS}. However, the maximal chordal extension does not provide any benefit in this case because the TSP graph is complete at sparsity order $k=1$ (see Lemma \ref{lem_maximal_extension} in Appendix \ref{proof-prop}).  
 
As can be seen in Section \ref{sec:numerical-experience}, numerical convergence occurs in all cases when using this basis for problems \eqref{compliance_problem_sdp_scalled} and \eqref{weight_problem_sdp_scalled}. However, we do not have a theoretical convergence result. To the best of our knowledge, the use of the NMT basis for constrained POPs that have a separable structure has not been explored in the literature, neither from a numerical nor theoretical point of view.

 We conclude this section by noting that the sparsity pattern of mSOS using the NMT basis cannot be obtained using the TSP with the standard monomial basis. Specifically, in this case, we would have to obtain cliques by TSP as a subset of the NMT basis. However, when we use the TSP (minimal or maximal) with the standard basis, we already have cross-terms in the cliques. To show this, let us consider the example of TSP graph in Figure \ref{figure-TSP-supportextension-graph-example}. There, we have a connected component $\lbrace 1, x_1, x_2, x_3, x_1^2, x_2^2,x_3^2 \rbrace$ but also the singleton cliques $x_1x_2$, $x_1x_3,$ $x_2x_3$, which we would not obtain using the NMT basis.

\section{Numerical experiments }
\label{sec:numerical-experience}

In this section, we provide numerical examples of the methods presented above: the dense mSOS hierarchy, both minimal and maximal chordal TSP mSOS hierarchy with the standard monomial basis, the minimal TSP with the non-mixed monomial basis, and finally the dense mSOS hierarchy with the non-mixed monomial basis. Numerical tests cover various frame structure optimization problems; each of them is described individually before the series of tests.   

All numerical tests were implemented in MATLAB and solved using the Mosek optimizer \cite{mosek} on the Apple M1 MacBook Pro laptop with 8 GB RAM. The implementation can be found at \url{https://gitlab.com/tyburec/pof-dyna}. We use the abbreviations:
\begin{itemize}
\item mSOS: the dense moment-SOS hierarchy using the standard monomial basis
\item NMT mSOS: the dense moment-SOS hierarchy using the NMT monomial basis
\item TSP max.ch: the TS SOS hierarchy with maximal chordal extension, sparsity order $1$ and using the standard monomial basis   
\item TSP min.ch: the TS SOS hierarchy with minimal chordal extension, and using the standard monomial basis   
\item NMT TSP: the TS SOS hierarchy with minimal chordal extension and using the NMT monomial basis
\item $k$: the sparsity order 
\item  $(n_c,s)$: the number $n_c$ of the linear matrix inequality of size $s$ in the relaxed problem
\item nvar: number of variables in the SDP relaxation
\item l.b: the lower bound
\item time: the CPU time in seconds for solving the SDP relaxation
\item $\varepsilon_{\text{rel}}$: the relative $\varepsilon-$optimally gap defined in \ref{certificate-global}; i.e., $\varepsilon_{\text{rel}}= \frac{\text{upper bound}}{\text{lower bound}}-1.$
\item $-$: out of memory 
\end{itemize}

\subsection{$21$-elements frame structure}
\label{sec:numerical-example-static21}

As the first illustration, we consider the frame structure containing $21$ elements and $8$ nodes as introduced in \cite[Section 4.5]{tyburec2021global}, see Fig.~\ref{fig:frame21bc}. Two downward loads are applied to the nodes $\circled{c}$ and $\circled{f}$ with magnitudes of $2$ and $3.5$, respectively, and the left nodes $\circled{a}$, $\circled{d}$ and $\circled{g}$ are clamped. The elements are made of a linear elastic isotropic material of the dimensionless Young modulus $E=1$, density $\rho = 1$ and share a circular hollow cross section, the thickness of the wall being $0.376$-multiple of the radius $r_i$ of individual cross sections, Fig.~\ref{fig:frame21cs}. Consequently, the stiffness matrix is a polynomial of degree two. We note here that in \cite{tyburec2021global}, the value of wall thickness was incorrectly reported as $0.2$ and also the number of elements was falsely stated as $22$. However, the optimization results presented there are correct.

\begin{figure}[!htbp]
\subfloat[\label{fig:frame21bc}]{%
    \begin{tikzpicture}
		\scaling{1.5};
		
		\point{a}{0}{0};
		\point{b}{1}{0};
		\point{c}{2}{0};
		\point{d}{0}{1};
		\point{e}{1}{1};
		\point{f}{2}{1};
		\point{g}{0}{2};
		\point{h}{1}{2};
		\point{i}{2}{2};
				
		\beam{2}{a}{b};
		\beam{2}{a}{e};
		\beam{2}{a}{f};
		\beam{2}{a}{h};
		\beam{2}{b}{c};
		\beam{2}{b}{d};
		\beam{2}{b}{e};
		\beam{2}{b}{f};
		\beam{2}{b}{g};
		\beam{2}{c}{d};
		\beam{2}{c}{e};
		\beam{2}{c}{f};
		\beam{2}{c}{h};
		\beam{2}{d}{e};
		\beam{2}{e}{f};
		\beam{2}{d}{h};
		\beam{2}{e}{g};
		\beam{2}{e}{h};
		\beam{2}{g}{h};
		\beam{2}{f}{g};
		\beam{2}{f}{h};
		
		\support{3}{a}[270];
		\support{3}{d}[270];
		\support{3}{g}[270];
		
		\load{1}{c}[90][0.7][-0.7];
		\load{1}{f}[90][1][0.12];
		\notation{1}{c}{$2$}[below right=2mm];
		\notation{1}{f}{$3.5$}[above right=2mm];
		
		\dimensioning{1}{a}{b}{-1.2}[$1$]
		\dimensioning{1}{b}{c}{-1.2}[$1$]
		\dimensioning{2}{a}{d}{-0.8}[$1$]
		\dimensioning{2}{d}{g}{-0.8}[$1$]
		
		\notation{1}{a}{\circled{$a$}}[align=center];
		\notation{1}{b}{\circled{$b$}}[align=center];
		\notation{1}{c}{\circled{$c$}}[align=center];
		\notation{1}{d}{\circled{$d$}}[align=center];
		\notation{1}{e}{\circled{$e$}}[align=center];
		\notation{1}{f}{\circled{$f$}}[align=center];
		\notation{1}{g}{\circled{$g$}}[align=center];
		\notation{1}{h}{\circled{$h$}}[align=center];
	\end{tikzpicture}
}
\subfloat[\label{fig:frame21cs}]{%
    \raisebox{1.5cm}{\begin{tikzpicture}[scale=1]
		\scaling{0.2}
		\point{a}{0}{0};
		\point{a1}{3.12182042}{0};
		\point{a2}{5}{0};
		\draw[black, fill=gray, fill opacity=0.2] (0,0) circle (1);
		\draw[black, fill=white, fill opacity=1.0] (0,0) circle (0.624364084);
		\dimensioning{1}{a}{a1}{1.2}[\raisebox{2mm}{$0.624r_i$}];
            \dimensioning{1}{a}{a2}{0}[$r_i\quad$];
	\end{tikzpicture}}
}
\subfloat[\label{fig:frame21opt}]{%
      \includegraphics[width=0.35\linewidth]{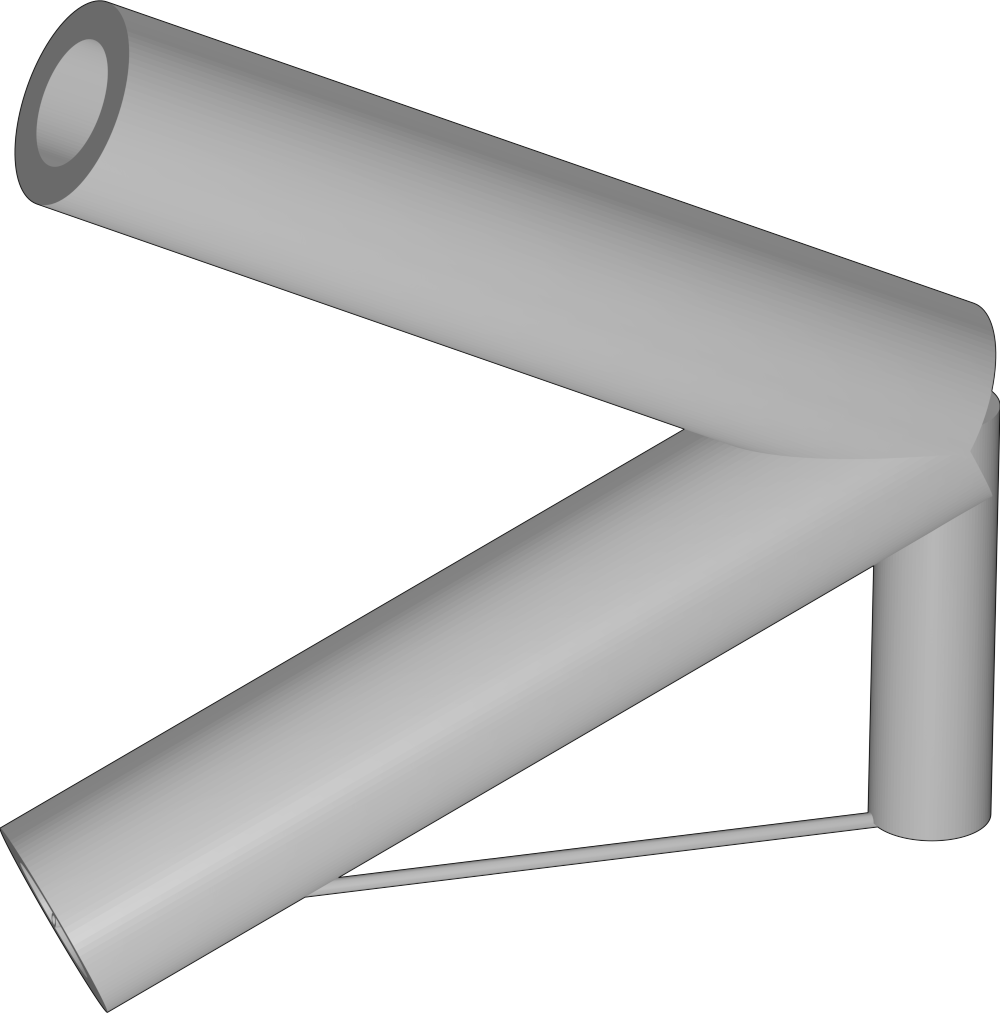}
    }
    \caption{$21$-elements frame structure: (a) discretization and boundary conditions, (b) cross-section parametrization, and (c) optimal design.}
\end{figure}
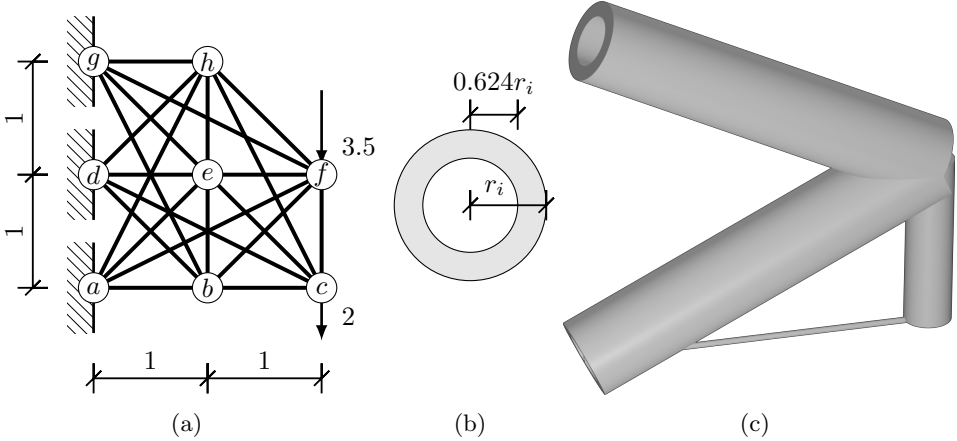

\subsection*{Compliance minimization}

As the first setting for the $21$-elements frame problem, we assume the compliance minimization setting. To make the problem bounded, we constrain the total weight of material from above by $\overline{w}=0.5$. In this case, Tyburec~\textit{et al.}~\cite{tyburec2021global} showed that the problem exhibits at least two local minimizers that are not global, and that would be obtained by standard local optimization algorithms such as SQP. However, the global solution can be reached using the mSOS hierarchy; see Fig.~\ref{fig:frame21opt} for a visualization of the optimal design.

In particular, using the mSOS hierarchy, convergence based on the rank-flatness condition occurs in the second relaxation, and the associated relative optimality gap $\varepsilon_{\mathrm{rel}}$ is within the expected range of the accuracy of the numerical optimizer. Slightly improved relative optimality gaps follow from a solution to the second relaxation using term-sparsity techniques such as TSP with maximal chordal extension and sparsity order one, TSP with minimal chordal extension and sparsity order two, or the NMT basis version. In all these cases, global designs up to optimizer tolerance have been found with more than $15$-times speed-up; see Table~\ref{numerical-example-dense-static21-compliance}. On the other hand, NMT TSP with sparsity order one was not sufficient for convergence and a higher sparsity order leads to the dense NMT mSOS.

     \begin{table}[h!]
         \caption{Computational results for compliance minimization of $21$-elements frame structure and the relaxation order $r=2$.}
  \begin{tabular}{ l  l  c c c c }
    \hline 
   \textbf{Method}    &$ \bm{(n_c ,s)} $ & $\textbf{nvar}$&\textbf{l.b.} & \textbf{time} & $\bm{\varepsilon_}{rel}$\\
       \hline 
        mSOS  & $ (22, 23), (1, 276), (1 , 368) $ & $14949$ &  $1668.58$   & $1800$ &$8\mathrm{e}{-05}$ \\

          \multirow{2}{*}{TSP max.ch. $k=1$  }   & $(22,23), (231,1),(1,45)$  & \multirow{2}{*}{ $7403$ }& \multirow{2}{*}{$1668.58$}  & \multirow{2}{*}{$  100 $ }& \multirow{2}{*}{$2\mathrm{e}{-06}$}\\
           &$(1,368)$ & & &  \\

           \multirow{2}{*}{TSP min.ch. $k=2$}   & $(231,4), (22,23), (1,45)$  & \multirow{2}{*}{$7403$ }& \multirow{2}{*}{$1668.59$}  & \multirow{2}{*}{$  120 $ }& \multirow{2}{*}{$6\mathrm{e}{-06}$}\\
           &$(1,368)$ & & &  \\
        
        NMT mSOS & $(22,23), (1,45), (1,368) $  & $7403$& $1668.58$   & $90$ &$ 2\mathrm{e}{-06}$  \\
        
         NMT TSP $k=1$  & $(528,2), (22,32)$ & $1012$ & $1530.60$ & $2$ & $0.5$\\ 
       \hline 
    \end{tabular}
    \label{numerical-example-dense-static21-compliance}
    \end{table}

\subsection*{Weight minimization}

Similarly to the compliance optimization version, we have also considered a solution of weight minimization. In particular, we fix $\overline{c} = 1668.58$, which is the optimal objective function of the compliance optimization variant, and we expect the optimal weight to be less than or equal to $0.5$, which is the weight bound used in compliance minimization.

The solution to this problem setting using the mSOS hierarchy is computationally intractable using testing hardware: In the second relaxation, the hierarchy does not converge, and the degree-three relaxation requires almost $291,745$ variables and, thus, enormous memory requirements. A similar situation emerges when using TSP with maximal chordal extension and sparsity order $1$.

On the contrary, the degree-three relaxation can be solved to approximate global optimality using the remaining TSP techniques presented in this paper. Specifically, TSP with minimal chordal extension and sparsity order one reached a relative optimality gap of $10^{-4}$ in $600$~s, NMT basis variant required $510$~s to achieve a guarantee of $3\times10^{-5}$ and NMT with TSP and sparsity order two achieved a relative optimality gap of $3\times 10^{-5}$ in $400$~s, see Table~\ref{numerical-example-dense-static21}.

 \begin{table}
     \caption{Computational results for weight minimization of $21$-elements frame structure and the relaxation order $r=3$.}
  \begin{tabular}{ l l c c c c }
    \hline 
   \textbf{Method}    &$ \bm{(n_c ,s)} $&  \textbf{nvar} &\textbf{l.b.} & \textbf{time} & $\bm{\varepsilon_}{rel}$\\
       \hline 
        mSOS &$(1,253), (1,2024), (1,4048) $  & $291745$& $-$& $-$& $- $ \\
 
 \multirow{2}{*}{TSP max.ch. $k=1$ }   & $(5950,1), (21,43), (1,484)$  & \multirow{2}{*}{ $241745$ }& \multirow{2}{*}{$- $}  & \multirow{2}{*}{$  - $}& \multirow{2}{*}{$ -$}\\
           &$(1,4048)$ & & &  \\
    
           \multirow{2}{*}{TSP min.ch. $k=1$ }   & $(462,2), (43,22), (209,48)$  & \multirow{2}{*}{ $13326$ }& \multirow{2}{*}{$0.50 $}  & \multirow{2}{*}{$  600 $}& \multirow{2}{*}{$ 1\mathrm{e}{-04}$ }\\
           &$(1,64) ,(1,688)$ & & &  \\
        
       NMT mSOS & $ (21,43), (1,64), (1,688) $  & $13496$ &$ 0.50$   & $  510 $ &  $3\mathrm{e}{-05}$ \\
         NMT TSP $k=1$  & $ (945,2), (21,48)$ &   $2646 $ &   $0.40 $&  $   3  $ &  $ 2\mathrm{e}{-01}$   \\ 
         NMT TSP $k=2$  & $ (21,43), (21,44), (1,688)$ & $13286$ & $ 0.50 $& $   400 $ & $3\mathrm{e}{-05}$   \\ 
       \hline 
    \end{tabular}
    \label{numerical-example-dense-static21}
    \end{table}

\subsection{$14$-elements simply-supported beam problem}
\label{sec:numerical-example-mbb14eb}

As the second example, we consider a symmetric, simply-supported beam shown in Fig.~\ref{fig:mbb14bc}. Here, the symmetric part of the design domain contains $8$ nodes, with nodes $\circled{a}$ and $\circled{f}$ supported in the horizontal direction and in rotation, and node $\circled{c}$ fixed in the vertical direction. A unit load is applied in the top-down direction at the node $\circled{f}$. The nodes are interconnected using $14$ elements made of a linear elastic material with the Young modulus $E=1$ and of the density $\rho = 1$. All elements share rectangular cross sections of widths $0.2$, see Fig.~\ref{fig:mbb14cs}. Consequently, we optimize the heights of the cross sections, and this results in a stiffness matrix that is degree-$3$ polynomial of the cross sections. 

\begin{figure}[!htbp]
\subfloat[\label{fig:mbb14bc}]{%
    \begin{tikzpicture}
		\scaling{1.5};
		
		\point{a}{0}{0};
		\point{b}{1}{0};
		\point{c}{2}{0};
		\point{d}{0.5}{0.5};
		\point{e}{1.5}{0.5};
		\point{f}{0}{1};
		\point{g}{1}{1};
		\point{h}{2}{1};
				
		\beam{2}{a}{b};
		\beam{2}{a}{d};
		\beam{2}{b}{c};
		\beam{2}{b}{d};
		\beam{2}{b}{e};
		\beam{2}{b}{g};
		\beam{2}{c}{e};
		\beam{2}{c}{h};
		\beam{2}{d}{f};
		\beam{2}{d}{g};
		\beam{2}{e}{g};
		\beam{2}{e}{h};
		\beam{2}{f}{g};
		\beam{2}{g}{h};
		
		\support{4}{a}[270];
		\support{4}{f}[270];
		\support{2}{c}[0];
		
		\load{1}{f}[90][1][0.12];
		\notation{1}{f}{$1$}[above right=2mm];
		
		\dimensioning{1}{a}{d}{-1.5}[$0.5$]
		\dimensioning{1}{d}{b}{-1.5}[$0.5$]
            \dimensioning{1}{b}{e}{-1.5}[$0.5$]
            \dimensioning{1}{e}{c}{-1.5}[$0.5$]
		\dimensioning{2}{a}{d}{-0.8}[$0.5$]
		\dimensioning{2}{d}{g}{-0.8}[$0.5$]
		
		\notation{1}{a}{\circled{$a$}}[align=center];
		\notation{1}{b}{\circled{$b$}}[align=center];
		\notation{1}{c}{\circled{$c$}}[align=center];
		\notation{1}{d}{\circled{$d$}}[align=center];
		\notation{1}{e}{\circled{$e$}}[align=center];
		\notation{1}{f}{\circled{$f$}}[align=center];
		\notation{1}{g}{\circled{$g$}}[align=center];
		\notation{1}{h}{\circled{$h$}}[align=center];
	\end{tikzpicture}
}
\subfloat[\label{fig:mbb14cs}]{%
    \raisebox{1cm}{\begin{tikzpicture}[scale=0.75]
	\scaling{0.15}
	\point{a}{0}{0};
	\point{b}{5}{0};
	\point{c}{5}{10};
	\draw[black, fill=gray, fill opacity=0.2] (0.0,0.0) rectangle ++(1,2);
	\dimensioning{1}{a}{b}{-0.75}[$0.2$];
	\dimensioning{2}{b}{c}{1.75}[$h_i$];
    \end{tikzpicture}}
}
\subfloat[\label{fig:mbb14opt}]{%
      \raisebox{1cm}{\includegraphics[width=0.4\linewidth]{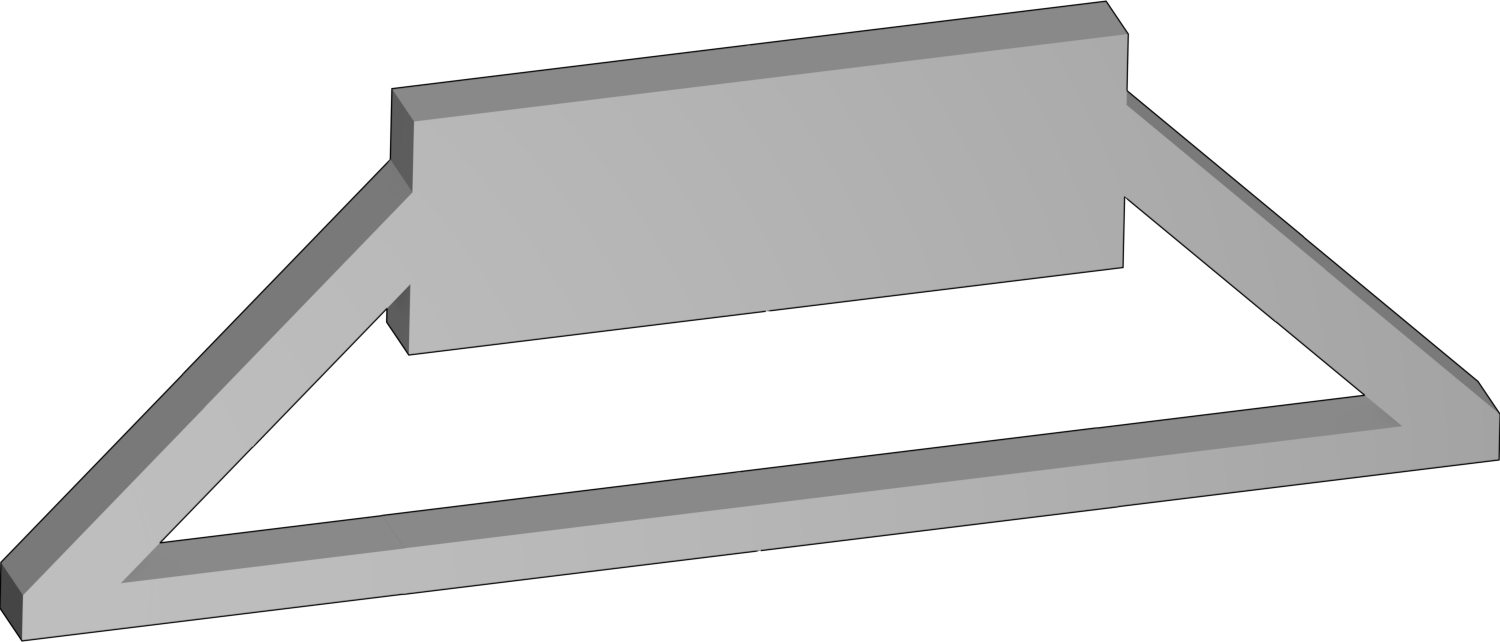}}
    }
    \caption{$14$-elements simply supported beam: (a) discretization and boundary conditions of the symmetric part, (b) cross-section parametrization, and (c) optimal symmetric design.}
\end{figure}
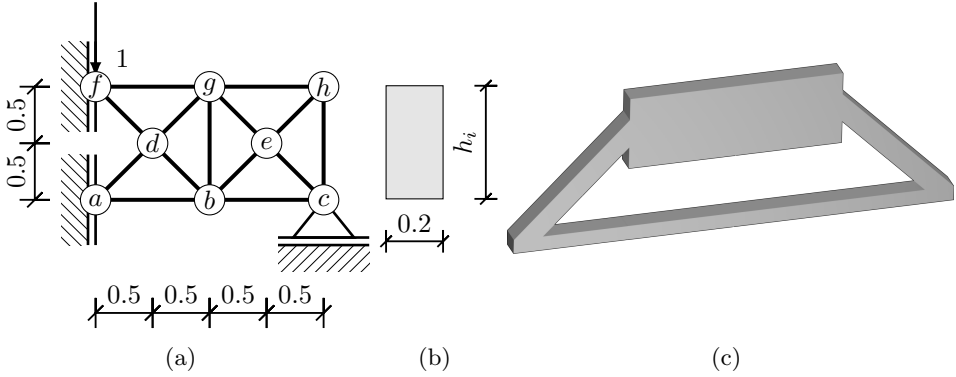

\subsection*{Compliance minimization}

First, we investigate the compliance minimization version of the problem, which requires an upper bound for the weight $\overline{w} = 0.25$. Using this setting, we were unable to certify optimality with the (lowest) degree-two relaxation of mSOS, and degree-three relaxation required more memory than what was available on the testing computer. Thus, we proceeded by using the sparsity techniques suggested in this paper, see Table~\ref{numerical-example-dense-reducedbasis-mbb14}. In particular, using degree-three relaxation with TSP and maximal chordal extension and sparsity order one allowed us to solve the problem globally in more than $30$ minutes with the relative optimality gap of $9\times10^{-6}$. However, a similar relative optimality gap was achieved using the NMT basis in only $40$~s and with the NMT basis with TSP and sparsity order $2$ in $45$~s. On the other hand, TSP with minimal chordal extension and sparsity order one was not sufficient for the convergence of the hierarchy. One would either need to increase the relaxation or the sparsity order. However, setting a higher sparsity order $k=2$ leads to a large number of cliques and larger number of variables than for TSP with maximal chordal extension (see Table \ref{tab:14-elements-compliance-TSP-min-ch-k=2}). This phenomenon was observed as well by using different algorithms for computing the chordal extension, notably minimum degree and maximum cardinality. This make using higher sparsity orders for TSP with minimal chordal extension non-suitable in problems where a higher relaxation order is required for convergence.

\begin{table}[h!]
 \caption{Computational results for compliance minimization of $14$-elements simply-supported beam problem and the relaxation order $r=3$.}
  \begin{tabular}{l l c c c  c }
    \hline 
  \textbf{Method}    &$ \bm{(n_c ,s)} $& \textbf{nvar}&\textbf{l.b.} & \textbf{time} & $\bm{\varepsilon_}{rel}$\\
        \hline
        mSOS  &$(16,136), (1,320), (1,816) $ & $54263$ & $-$   & $-$ & $-$ \\
         
 \multirow{2}{*}{TSP max.ch. $k=1$}  & $(2135,1),(15,31), (1,136)$  &\multirow{2}{*}{$28783$}  & \multirow{2}{*}{$210.27$}  &    \multirow{2}{*}{$2000$} &    \multirow{2}{*}{$9\mathrm{e}{-06}$}\\
       
         & $(1,256), (1,320)$ &    & & \\
          \multirow{2}{*}{TSP min.ch. $k=1$}   & $(31,2), (1,30), (15,40)$  & \multirow{2}{*}{  $4583$ }& \multirow{2}{*}{$113.63 $}  & \multirow{2}{*}{$  15  $ }& \multirow{2}{*}{$  1.46$ }\\
           &$(329,3),  (31,16) $ & & &  \\
           &$(329,3)$ & & &  \\

         NMT mSOS &  $(16,31), (1,46), (1,320) $ & $6369$ &  $210.27$   & $40$ & $2\mathrm{e}{-06} $\\
        
           NMT TSP $k=1$  & $ (35,2), (252,3) (15,40) $ & $1307$   & $91.66$ &  $2$ &  $2.05$    \\ 

            \multirow{2}{*}{NMT TSP $k=2$}   & $(1,3),(16,31), (14,32)$  & \multirow{2}{*}{$6264$}& \multirow{2}{*}{$210.27 $}  & \multirow{2}{*}{$  45  $}& \multirow{2}{*}{$6\mathrm{e}{-06}$}\\
           &$(1,320)$ & & &  \\
       \hline 
    \end{tabular}
 \label{numerical-example-dense-reducedbasis-mbb14}
    \end{table} 

  \begin{table}[h!]
 \caption{ TSP min.ch. using the minimum fill-in algorithm and with sparsity order $k=2$, for the weight minimization of $14$-elements a simply-supported beam and the relaxation order $r=3$.}
  \begin{tabular}{c c  }
    \hline 
  $ \bm{(n_c ,s)} $& \textbf{nvar}\\
        \hline         
        $ (91,4),(364,8), (16,18), (16,19), (16,20), (16,21),(16,22) ,(16,23), (16,24)$ &  \multirow{4}{*}{$42237$} \\
           $ (16,25),(16,26),(16,27), (16,28), (16,29),(16,30),(61,31),(14,45), (16,55) $&  \\ 
          $(14,58), (16,65), (16,74), (16,82), (16,89), (16,95), (16,100),(16,104)$&  \\ 
         $ (16,107),(16,109),(32,110), (1,122), (1,320)$ & \\
           \hline 
    \end{tabular}
 \label{tab:14-elements-compliance-TSP-min-ch-k=2}
    \end{table}

\subsubsection*{Weight minimization }
\label{sec:numerical-example-mbb14eb}

Having found the optimal compliance $c^*=210.27$ for the weight limit of $0.25$, we changed the problem to the weight minimization settings. This means that we fixed the upper bound of compliance to $\overline{c} = 210.27$ and expected the optimized weight to be less than or equal to $0.25$. In fact, the optimal weight is indeed $0.25$, which can be proven using mSOS-based hierarchies. Similarly to the previous case, the mSOS hierarchy is computationally intractable, and the same happened for the TSP variants. However, convergence occurred in the degree-five relaxation for the NMT basis version. In particular, NMT required more than $10$ minutes to obtain the relative optimality gap of $5\times10^{-5}$ and NMT with TSP and sparsity order two reached the relative optimality gap of $10^{-6}$ in $15$ minutes, see Table~\ref{numerical-example-dense-reducedbasis-mbb14eb-weight}.

\begin{table}[h!]
 \caption{Computational results for weight minimization of $14$-elements simply-supported beam problem and the relaxation order $r=5$.}
  \begin{tabular}{l l c c c  c }
    \hline 
  \textbf{Method}    &$ \bm{(n_c ,s)} $& \textbf{nvar}&\textbf{l.b.} & \textbf{time} & $\bm{\varepsilon_}{rel}$\\
        \hline
        \multirow{2}{*}{mSOS}  & $ (14,3060), (1,9520) $  & \multirow{2}{*}{$1961256 $}  &\multirow{2}{*}{$- $}   & \multirow{2}{*}{$- $}  & \multirow{2}{*}{$- $} \\
        &  $  (1,11628)$& & & & \\
        TSP max.ch. $k=1$  &$ -$ & $ - $ & $-$   & $-$ & $-$ \\
        TSP min.ch. $k=1$  &$ -$ & $ - $ & $-$   & $-$ & $-$ \\
       
         NMT mSOS &  $ (14,57), (1,71), (1,860)$ & $ 19887$ &  $ 0.25$   & $640 $ & $ 5\mathrm{e}{-05} $\\
  NMT TSP $k=1$  & $  (644,3), (14,80) $ & $2779$  & $0.09$ &  $9$  & $2.28$   \\ 
           NMT TSP $k=2$&  $ (196,44), (14,45), (1,860)$ & $ 18522$ &  $ 0.25$   & $ 1000 $ & $ 1\mathrm{e}{-06} $\\

       \hline 
    \end{tabular}
 \label{numerical-example-dense-reducedbasis-mbb14eb-weight}
    \end{table}

\subsection{$24$-elements modular frame structure}
\label{sec:numerical-example-frame24}

As the third problem, we adopt the $24$-elements modular frame structure introduced in \cite[Section 4.1]{tyburec2022global}. The structure contains $14$ nodes, with horizontal wind loads of magnitudes $1$, $1$, and $1/2$ applied to the nodes $\circled{c}$, $\circled{e}$ and $\circled{g}$, respectively, and vertical unit loads applied to the nodes $\circled{i}$, $\circled{j}$ and $\circled{k}$. Kinematic boundary conditions are enforced by clamped supports at the bottom nodes $\circled{a}$ and $\circled{b}$. The nodes are connected using $24$ elements made of the linear elastic material with the Young modulus $E=1$ and the density $\rho=1$; see Fig.~\ref{fig:frame24}.

\begin{figure}[!htbp]
\subfloat[\label{fig:frame24}]{%
    \begin{tikzpicture}
		\scaling{2.2}
		\point{a}{0.00}{0.0};
		\point{b}{1.50}{0.0}; 
		\point{c}{0.00}{1.0}; 
		\point{d}{1.50}{1.0}; 
		\point{e}{0.00}{2.0}; 
		\point{f}{1.50}{2.0}; 
		\point{g}{0.00}{3.0}; 
		\point{h}{1.50}{3.0}; 
		\point{i}{0.75}{1.0};
		\point{j}{0.75}{2.0}; 
		\point{k}{0.75}{3.0}; 
		\point{l}{0.75}{0.5}; 
		\point{m}{0.75}{1.5}; 
		\point{n}{0.75}{2.5}; 
		\point{o}{0.75}{0.0};
		
		\beam{2}{a}{c}; \notation{4}{a}{c}[$1$];
		\beam{2}{b}{d}; \notation{4}{b}{d}[$2$];
		\beam{2}{c}{e}; \notation{4}{c}{e}[$3$];
		\beam{2}{d}{f}; \notation{4}{d}{f}[$4$];
		\beam{2}{e}{g}; \notation{4}{e}{g}[$5$];
		\beam{2}{f}{h}; \notation{4}{f}{h}[$6$];
		\beam{2}{c}{i}; \notation{4}{c}{i}[$7$];
		\beam{2}{i}{d}; \notation{4}{i}{d}[$8$];
		\beam{2}{e}{j}; \notation{4}{e}{j}[$9$];
		\beam{2}{j}{f}; \notation{4}{j}{f}[$10$];
		\beam{2}{g}{k}; \notation{4}{g}{k}[$11$];
		\beam{2}{k}{h}; \notation{4}{k}{h}[$12$];
		\beam{2}{a}{l}; \notation{4}{a}{l}[$13$];
		\beam{2}{l}{d}; \notation{4}{l}{d}[$14$];
		\beam{2}{b}{l}; \notation{4}{b}{l}[$15$];
		\beam{2}{l}{c}; \notation{4}{l}{c}[$16$];
		\beam{2}{c}{m}; \notation{4}{c}{m}[$17$];
		\beam{2}{m}{f}; \notation{4}{m}{f}[$18$];
		\beam{2}{d}{m}; \notation{4}{d}{m}[$19$];
		\beam{2}{m}{e}; \notation{4}{m}{e}[$20$];
		\beam{2}{e}{n}; \notation{4}{e}{n}[$21$];
		\beam{2}{n}{h}; \notation{4}{n}{h}[$22$];
		\beam{2}{f}{n}; \notation{4}{f}{n}[$23$];
		\beam{2}{n}{g}; \notation{4}{n}{g}[$24$];

            \support{3}{a};
		\support{3}{b};

            \notation{1}{a}{$\circled{a}$}[align=center];
            \notation{1}{b}{$\circled{b}$}[align=center];
            \notation{1}{c}{$\circled{c}$}[align=center];
            \notation{1}{d}{$\circled{d}$}[align=center];
            \notation{1}{e}{$\circled{e}$}[align=center];
            \notation{1}{f}{$\circled{f}$}[align=center];
            \notation{1}{g}{$\circled{g}$}[align=center];
            \notation{1}{h}{$\circled{h}$}[align=center];
            \notation{1}{i}{$\circled{i}$}[align=center];
            \notation{1}{j}{$\circled{j}$}[align=center];
            \notation{1}{k}{$\circled{k}$}[align=center];
            \notation{1}{l}{$\circled{l}$}[align=center];
            \notation{1}{m}{$\circled{m}$}[align=center];
            \notation{1}{n}{$\circled{n}$}[align=center];
		
		\load{1}{i}[90][0.7][0.15]; \notation{1}{i}{$1$}[above=3mm,xshift=1.5mm];
		\load{1}{j}[90][0.7][0.15]; \notation{1}{j}{$1$}[above=3mm,xshift=1.5mm];
		\load{1}{k}[90][0.7][0.15]; \notation{1}{k}{$1$}[above=3mm,xshift=1.5mm];
		\load{1}{c}[180][0.7][0.15]; \notation{1}{c}{$1$}[below=0mm, xshift=-3mm];
		\load{1}{e}[180][0.7][0.15]; \notation{1}{e}{$1$}[below=0mm, xshift=-3mm];
		\load{1}{g}[180][0.35][0.15]; \notation{1}{g}{$0.5$}[below=0mm, xshift=-3mm];
		
		\dimensioning{2}{o}{l}{4.0}[$0.5$];
		\dimensioning{2}{l}{i}{4.0}[$0.5$];
		\dimensioning{2}{i}{m}{4.0}[$0.5$];
		\dimensioning{2}{m}{j}{4.0}[$0.5$];
		\dimensioning{2}{j}{n}{4.0}[$0.5$];
		\dimensioning{2}{n}{k}{4.0}[$0.5$];
		\dimensioning{1}{a}{o}{-0.75}[$0.75$];
		\dimensioning{1}{o}{b}{-0.75}[$0.75$];
	\end{tikzpicture}
}
\subfloat[\label{fig:frame24cs}]{%
    \raisebox{4.25cm}{\begin{minipage}{0.15\linewidth}
    \centering
	\squared{1}--\squared{6}:\\
	\begin{tikzpicture}
		\scaling{1.0}
		\point{a}{0.0}{0.0};\point{a0}{1.0}{0.0};
		\point{b}{0.0}{0.1};\point{b0}{1.0}{0.1};
		\point{c}{0.45}{0.1};\point{c0}{0.55}{0.1};
		\point{d}{0.45}{0.9};\point{d0}{0.55}{0.9};
		\point{e}{0.0}{0.9};\point{e0}{1.0}{0.9};
		\point{f}{0.0}{1.0};\point{f0}{1.0}{1.0};
		\draw[black, thick, fill=black!25] (a0) -- (b0) -- (c0) -- (d0) -- (e0) -- (f0) -- (f) -- (e) -- (d) -- (c) -- (b) -- (a) -- cycle;
		\dimensioning{2}{a}{f}{1.25}[$10t_e$];
		\dimensioning{1}{a}{a0}{1.25}[$10t_e$];
		\draw [-stealth](0.25,0.60) -- (0.5,0.75);
		\draw [-stealth](0.25,0.60) -- (0.35,0.95);
		\draw [-stealth](0.25,0.60) -- (0.35,0.05);
		\draw (-0.2,0.6) -- node[above=-0.5mm]{$t_e$}(0.025,0.6) -- (0.25,0.60);
	\end{tikzpicture}\\
        \vspace{3mm}
	\squared{7}--\squared{12}:\\
	\begin{tikzpicture}
		\scaling{1.0}
		\point{a}{0.0}{0.0};\point{a0}{1.0}{0.0};
		\point{b}{0.0}{0.1};\point{b0}{1.0}{0.1};
		\point{c}{0.45}{0.1};\point{c0}{0.55}{0.1};
		\point{d}{0.45}{1.9};\point{d0}{0.55}{1.9};
		\point{e}{0.0}{1.9};\point{e0}{1.0}{1.9};
		\point{f}{0.0}{2.0};\point{f0}{1.0}{2.0};
		\draw[black, thick, fill=black!25] (a0) -- (b0) -- (c0) -- (d0) -- (e0) -- (f0) -- (f) -- (e) -- (d) -- (c) -- (b) -- (a) -- cycle;
		\dimensioning{2}{a}{f}{1.25}[$20t_e$];
		\dimensioning{1}{a}{a0}{2.25}[$10t_e$];
		\draw [-stealth](0.25,0.60) -- (0.5,0.75);
		\draw [-stealth](0.25,0.60) -- (0.35,1.95);
		\draw [-stealth](0.25,0.60) -- (0.35,0.05);
		\draw (-0.2,0.6) -- node[above=-0.5mm]{$t_e$}(0.025,0.6) -- (0.25,0.60);
	\end{tikzpicture}\\
	\vspace{3mm}
	\squared{13}--\squared{24}:\\
	\begin{tikzpicture}
		\point{a}{-1.0}{0};\point{a0}{0.0}{0.0};
		\draw[thick, fill=black!25] (0.0,0) arc (0:360:0.5);
		\draw[thick, fill=white] (-0.1,0) arc (0:360:0.4);
		\dimensioning{1}{a}{a0}{0.65}[$10t_e$];
		\draw [-stealth](-0.3,-0.1) -- (-0.05,0.05);
		\draw (-0.7,-0.1) -- node[above=-0.5mm]{$t_e$}(-0.4,-0.1) -- (-0.3,-0.1);
	\end{tikzpicture}
        \end{minipage}}
}
\subfloat[\label{fig:frame24opt}]{%
      \includegraphics[width=0.37\linewidth]{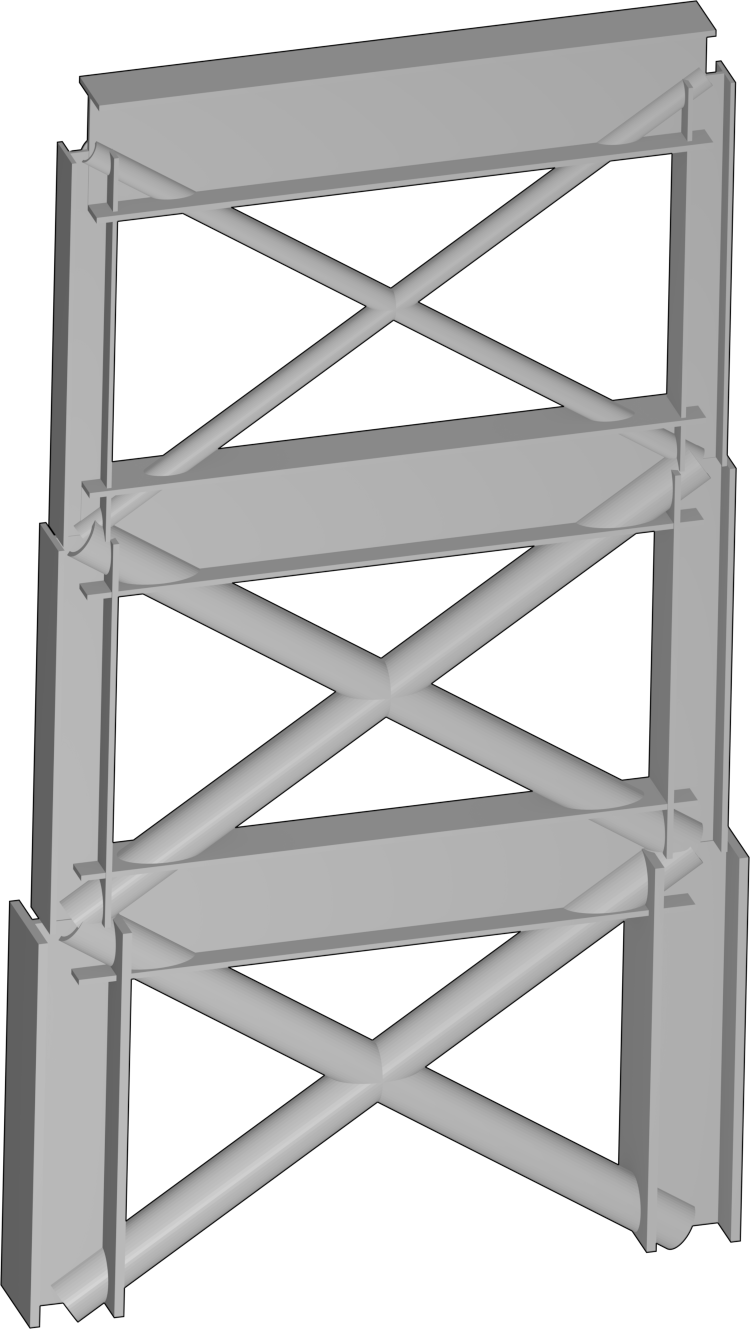}
    }
    \caption{$24$-elements modular frame structure: (a) discretization and boundary conditions, (b) cross-section parametrizations, and (c) optimal design.}
\end{figure}
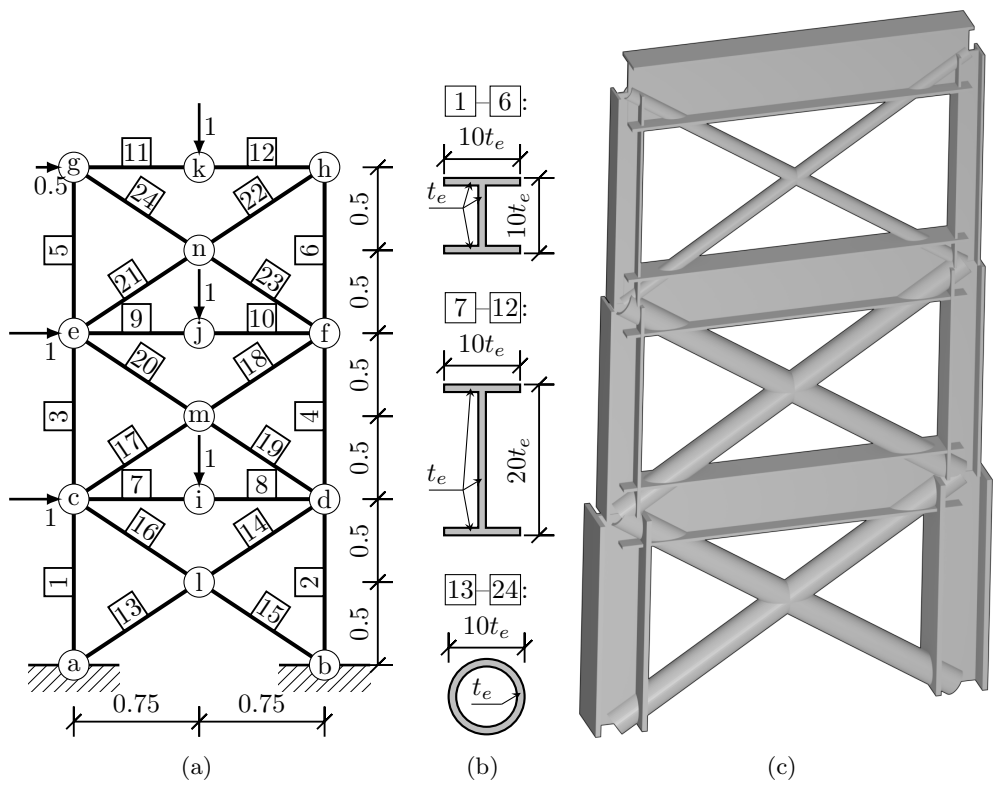

We restrict the structural topology to be symmetric. This can be justified by the observation that wind loads can act on either side of the structure. Moreover, we also enforce equal cross sections of the diagonal elements within each storey. Consequently, we have $a_1=a_2$, $a_3=a_4$ and $a_5=a_6$ for columns, $a_7=a_8$, $a_9=a_{10}$ and $a_{11}=a_{12}$ for horizontal beams and $a_{13}=a_{14}=a_{15}=a_{16}$, $a_{17}=a_{18}=a_{19}=a_{20}$ and $a_{21}=a_{22}=a_{23}=a_{24}$ for diagonals. Thus, we have $9$ independent cross-sectional variables. Moreover, we assign different cross-sectional shapes to the elements: columns have an I-shaped cross section, beams have an H-shaped section, and diagonals are assigned a thin-walled circular section; see Fig.~\ref{fig:frame24cs}, which implies that the stiffness matrix is a degree-two polynomial of the cross-section areas.

\subsubsection*{Weight optimization}

First, we consider the setting in \cite{tyburec2022global}, which is the minimization of weight. For this case, we set the compliance bound $\overline{c}=5,000$. Using the mSOS hierarchy, convergence occurs in the third relaxation based on the rank-flatness condition, and the associated relative optimality gap is evaluated as $3\times10^{-5}$. The optimal design appears in Fig.~\ref{fig:frame24opt}. However, the relaxation of degree three is not sufficient for the convergence of the other hierarchies, and also the speed-up by TSP is insignificant; see Table \ref{tab:24_weight_r3}. This is because, in the lowest relaxation, the size of the stiffness matrix ($37\times37$) is larger than the moment matrix ($10\times10$), and this also translates into higher-degree relaxations. 

\begin{table}[h!]
 \caption{Computational results for weight minimization of $24$-elements modular frame structure and the relaxation order $r=3$.}
  \begin{tabular}{l l c c c  c }
    \hline 
  \textbf{Method}    &$ \bm{(n_c ,s)} $& \textbf{nvar}&\textbf{l.b.} & \textbf{time} & $\bm{\varepsilon_}{rel}$\\
        \hline
        mSOS  &$ (9,55), (1,220), (1,2035) $ & $ 5005$ & $0.1180$   & $1600$ & $3\mathrm{e}{-05}$ \\

         \multirow{2}{*}{TSP max.ch. $k=1$} & $(444,1),(1,2035)$ &  \multirow{2}{*}{$4920$} &
          \multirow{2}{*}{$0.1179$} &  \multirow{2}{*}{$ 1500$} &  \multirow{2}{*}{$ 1\mathrm{e}{-04}$} \\
           & $(1,100), (9,19) $ & & & & \\

          \multirow{2}{*}{TSP min.ch. $k=1$} & $(90,2),(19,2)$ &  \multirow{2}{*}{$1158$} &
          \multirow{2}{*}{$0.1165$} &  \multirow{2}{*}{$ 1500$} &  \multirow{2}{*}{$ 1\mathrm{e}{-02}$} \\
           & $(36,111), (1,703) $ & & & & \\

         NMT mSOS &  $ (9,19), (1,28), (1,703)$ & $ 1194$ &  $0.1176 $   & $ 40$ & $ 4\mathrm{e}{-03}  $\\
 NMT TSP $k=1$  & $ (189,2),(9,111)   $ & $486$  & $0.0838$  & $3.4$ & $0.50$   \\ 
           NMT TSP $k=2$ & $ (9,19),(9,20), (1,703)$ & $1158 $ &  $0.1176 $   & $ 35$ & $ 4\mathrm{e}{-03}  $\\

       \hline 
    \end{tabular}
 \label{tab:24_weight_r3}
    \end{table} 

Consequently, solving higher relaxations using the canonical basis with or without TSP is computationally intractable using testing hardware. However, convergence can still be achieved using the NMT basis. In particular, we need the relaxation order $5$ to reach the relative optimality gap of $5\times10^{-6}$ and $6\times10^{-6}$ using the NMT mSOS hierarchy and the NMT TSP hierarchy with the sparsity order $3$, see Table \ref{tab:24_weight_r5}.

\begin{table}[h!]
 \caption{Computational results using NMT basis for weight minimization of $24$-elements modular frame structure and the relaxation order $r=5$.}
  \begin{tabular}{l l c c c  c }
    \hline 
  \textbf{Method}    &$ \bm{(n_c ,s)} $& \textbf{nvar}&\textbf{l.b.} & \textbf{time} & $\bm{\varepsilon_}{rel}$\\
        \hline

         NMT mSOS  &  $ (9,37), (1,46), (1,1369)$ & $ 6270 $ &  $0.1180 $   & $ 381 $ & $ 5\mathrm{e}{-06}  $\\
          NMT TSP $k=1$  & $  (171,2), (99,3), (9,185)   $ & $1098$  &  $0.0901$  &  $9$  &  $0.35$   \\ 
            NMT TSP $k=2$  & $  (81,21), (9,22), (9,777)  $ &  $3930$  &  $0.1179$  & $367 $ &  $2\mathrm{e}{-04}$   \\ 
           NMT TSP $k=3$ & $ (9,22),(10,37), (1,1369)$ & $  6234$ &  $0.1180 $   & $ 377 $ & $ 6\mathrm{e}{-06}  $\\

       \hline 
    \end{tabular}
  \label{tab:24_weight_r5}
    \end{table}

\subsubsection*{Compliance optimization}

Second, we use the optimal weight from the weight optimization setting as a constraint in the compliance minimization variant, $\overline{w}=0.118$. In this case, the hierarchy converges in the third relaxation. Numerical convergence can be proved using TSP with maximal chordal extension and sparsity order one in less than $2$ hours, with TSP with minimal chordal extension and sparsity order one in $4$ minutes, with the NMT basis in $5$ minutes, and with the NMT basis and TSP in $2$ minutes. On the other hand, convergence does not follow from the results of the mSOS hierarchy, since the rank-flatness condition does not hold and the relative optimality gap is only $10^{-3}$. We attribute this to a numerically inaccurate solution of the full hierarchy. We wish to emphasize here that even if the mSOS was accurate, the solution time is more than $100$-times longer than the NMT TSP technique with sparsity order $2$.

\begin{table}[h!]
\caption{Computational results for compliance minimization of $24$-elements modular frame structure and the relaxation order $r=3$.}
  \begin{tabular}{l l c c c  c }
    \hline 
  \textbf{Method}    &$ \bm{(n_c ,s)} $& \textbf{nvar}&\textbf{l.b.} & \textbf{time} & $\bm{\varepsilon_}{rel}$\\
        \hline
        mSOS  &$ (11,66), (1,286), (1,2442) $ & $  8007$ & $ 4995.00  $   & $ 15000$ & $ 1\mathrm{e}{-03}$ \\

         \multirow{2}{*}{TSP max.ch. $k=1$} & $(615,1),(10,21), (1,66)$ &  \multirow{2}{*}{$ 7587$} &
          \multirow{2}{*}{$5000.65 $} &  \multirow{2}{*}{$ 6000$} &  \multirow{2}{*}{$ 2\mathrm{e}{-06}$} \\
           & $(1,121), (1,2442) $ & & & & \\

          \multirow{3}{*}{TSP min.ch. $k=1$} & $(120,2), (36,3), (21,11)$ &  \multirow{3}{*}{$1577$} &
          \multirow{3}{*}{$5000.65 $} &  \multirow{3}{*}{$ 220$} &  \multirow{3}{*}{$ 5\mathrm{e}{-06}$} \\
           & $(1,20), (44,111) $ & & & & \\ 
        & $(1,148), (1,777) $ & & & & \\ 
       
         NMT mSOS &  $ (11,21), (1,31), (1,777)$ & $ 1605 $ &  $ 5000.65 $   & $ 300$ & $ 3\mathrm{e}{-06}  $\\

           NMT TSP $k=2$ & $ (20,21),(1,22), (1,777)$ & $ 1560 $ &  $5000.65$   & $ 130$ & $ 1\mathrm{e}{-06}  $\\

       \hline 
    \end{tabular}
  \label{tab8}
    \end{table}

\begin{remark}
 From an accuracy perspective, NMT mSOS is always expected to perform as well as or better than NMT TSP, as its feasible set is tighter than that of NMT TSP. However, in some cases, NMT TSP exhibits a slightly smaller relative error than NMT mSOS (as seen in Table \ref{tab8}), likely due to numerical issues or solver behavior (Mosek). Similarly, in terms of solving time, NMT TSP is generally expected to be faster than NMT mSOS. However, in certain cases (such as in Table \ref{numerical-example-dense-reducedbasis-mbb14eb-weight}), it can be slower. In these instances, we observe that Mosek requires a higher number of iterations to solve the relaxation.
\end{remark}
 
    \subsection{$24$-elements problems by Toragay \textit{et al.}~\cite{toragay2022exact}}

    As the fourth set of problems, we investigate slightly modified weight minimization problems introduced by Toragay \textit{et al.}~\cite{toragay2022exact}. In particular, we adopt a finite element discretization consisting of $9$ nodes interconnected by $24$ elements that are assigned circular cross sections, rendering the stiffness matrix a degree-two polynomial function of the cross sections. Individual elements are made of a linear elastic material with the Young modulus $E=109$~GPa and the dimensionless density $\rho=1$.

    Three types of boundary conditions are established. First, we have clamped supports at the nodes $\circled{a}$ and $\circled{g}$ and a vertical force $F$ at the node $\circled{f}$, see Fig.~\ref{fig:toragay1bc}. Second, we fix the nodes $\circled{a}$, $\circled{b}$ and $\circled{d}$ and apply a vertical load $F$ at the node $\circled{i}$; see Fig.~\ref{fig:toragay2bc}. Finally, we fix the nodes $\circled{a}$, $\circled{b}$ and $\circled{c}$ and apply a vertical load $F$ at the node $\circled{h}$, Fig.~\ref{fig:toragay3bc}. For each type of boundary condition, we consider three values of the load $F$: $25$~kN, $50$~kN, and $75$~kN. Thus, we have $9$ testing instances in total.

    \begin{figure}[!htbp]
\subfloat[\label{fig:toragay1bc}]{%
    \begin{tikzpicture}
		\scaling{0.0475};
		
		\point{a}{0.000000}{0.000000}
            \point{b}{0.000000}{25.000000}
            \point{c}{0.000000}{50.000000}
            \point{d}{25.000000}{0.000000}
            \point{e}{25.000000}{25.000000}
            \point{f}{25.000000}{50.000000}
            \point{g}{50.000000}{0.000000}
            \point{h}{50.000000}{25.000000}
            \point{i}{50.000000}{50.000000}
            \beam{2}{a}{d}
            \beam{2}{b}{e}
            \beam{2}{c}{f}
            \beam{2}{d}{g}
            \beam{2}{e}{h}
            \beam{2}{f}{i}
            \beam{2}{a}{b}
            \beam{2}{b}{c}
            \beam{2}{d}{e}
            \beam{2}{e}{f}
            \beam{2}{g}{h}
            \beam{2}{h}{i}
            \beam{2}{a}{e}
            \beam{2}{b}{f}
            \beam{2}{d}{h}
            \beam{2}{e}{i}
            \beam{2}{b}{d}
            \beam{2}{c}{e}
            \beam{2}{e}{g}
            \beam{2}{f}{h}
            \beam{2}{a}{f}
            \beam{2}{d}{i}
            \beam{2}{c}{d}
            \beam{2}{f}{g}

		\support{3}{a}[0];
		\support{3}{g}[0];
		
		\load{1}{f}[90][-0.75][1.0];
		\notation{1}{f}{$F$}[above right=2mm];
		
		\dimensioning{1}{a}{d}{-0.8}[$25$~mm]
		\dimensioning{1}{d}{g}{-0.8}[$25$~mm]
		\dimensioning{2}{a}{b}{-0.6}[$25$~mm]
		\dimensioning{2}{b}{c}{-0.6}[$25$~mm]
		
		\notation{1}{a}{$\circled{a}$}[align=center];
            \notation{1}{b}{$\circled{b}$}[align=center];
            \notation{1}{c}{$\circled{c}$}[align=center];
            \notation{1}{d}{$\circled{d}$}[align=center];
            \notation{1}{e}{$\circled{e}$}[align=center];
            \notation{1}{f}{$\circled{f}$}[align=center];
            \notation{1}{g}{$\circled{g}$}[align=center];
            \notation{1}{h}{$\circled{h}$}[align=center];
            \notation{1}{i}{$\circled{i}$}[align=center];
	\end{tikzpicture}
}%
\subfloat[\label{fig:toragay2bc}]{%
    \begin{tikzpicture}
		\scaling{0.0475};
		
		\point{a}{0.000000}{0.000000}
            \point{b}{0.000000}{25.000000}
            \point{c}{0.000000}{50.000000}
            \point{d}{25.000000}{0.000000}
            \point{e}{25.000000}{25.000000}
            \point{f}{25.000000}{50.000000}
            \point{g}{50.000000}{0.000000}
            \point{h}{50.000000}{25.000000}
            \point{i}{50.000000}{50.000000}
            \beam{2}{a}{d}
            \beam{2}{b}{e}
            \beam{2}{c}{f}
            \beam{2}{d}{g}
            \beam{2}{e}{h}
            \beam{2}{f}{i}
            \beam{2}{a}{b}
            \beam{2}{b}{c}
            \beam{2}{d}{e}
            \beam{2}{e}{f}
            \beam{2}{g}{h}
            \beam{2}{h}{i}
            \beam{2}{a}{e}
            \beam{2}{b}{f}
            \beam{2}{d}{h}
            \beam{2}{e}{i}
            \beam{2}{b}{d}
            \beam{2}{c}{e}
            \beam{2}{e}{g}
            \beam{2}{f}{h}
            \beam{2}{a}{f}
            \beam{2}{d}{i}
            \beam{2}{c}{d}
            \beam{2}{f}{g}

		\support{3}{a}[0];
		\support{3}{d}[0];
            \support{3}{b}[270];
		
		\load{1}{i}[90][-0.75][1.0];
		\notation{1}{i}{$F$}[above left=2mm];
		
		\dimensioning{1}{a}{d}{-0.8}[$25$~mm]
		\dimensioning{1}{d}{g}{-0.8}[$25$~mm]
		\dimensioning{2}{a}{b}{-0.6}[$25$~mm]
		\dimensioning{2}{b}{c}{-0.6}[$25$~mm]
		
		\notation{1}{a}{$\circled{a}$}[align=center];
            \notation{1}{b}{$\circled{b}$}[align=center];
            \notation{1}{c}{$\circled{c}$}[align=center];
            \notation{1}{d}{$\circled{d}$}[align=center];
            \notation{1}{e}{$\circled{e}$}[align=center];
            \notation{1}{f}{$\circled{f}$}[align=center];
            \notation{1}{g}{$\circled{g}$}[align=center];
            \notation{1}{h}{$\circled{h}$}[align=center];
            \notation{1}{i}{$\circled{i}$}[align=center];
	\end{tikzpicture}
}%
\subfloat[\label{fig:toragay3bc}]{%
    \begin{tikzpicture}
		\scaling{0.0475};
		
		\point{a}{0.000000}{0.000000}
            \point{b}{0.000000}{25.000000}
            \point{c}{0.000000}{50.000000}
            \point{d}{25.000000}{0.000000}
            \point{e}{25.000000}{25.000000}
            \point{f}{25.000000}{50.000000}
            \point{g}{50.000000}{0.000000}
            \point{h}{50.000000}{25.000000}
            \point{h2}{53.000000}{25.000000}
            \point{i}{50.000000}{50.000000}
            \beam{2}{a}{d}
            \beam{2}{b}{e}
            \beam{2}{c}{f}
            \beam{2}{d}{g}
            \beam{2}{e}{h}
            \beam{2}{f}{i}
            \beam{2}{a}{b}
            \beam{2}{b}{c}
            \beam{2}{d}{e}
            \beam{2}{e}{f}
            \beam{2}{g}{h}
            \beam{2}{h}{i}
            \beam{2}{a}{e}
            \beam{2}{b}{f}
            \beam{2}{d}{h}
            \beam{2}{e}{i}
            \beam{2}{b}{d}
            \beam{2}{c}{e}
            \beam{2}{e}{g}
            \beam{2}{f}{h}
            \beam{2}{a}{f}
            \beam{2}{d}{i}
            \beam{2}{c}{d}
            \beam{2}{f}{g}

		\support{3}{a}[270];
		\support{3}{b}[270];
            \support{3}{c}[270];
		
		\load{1}{h2}[90][0.75][0.15];
		\notation{1}{h}{$F$}[above right=2mm];
		
		\dimensioning{1}{a}{d}{-0.8}[$25$~mm]
		\dimensioning{1}{d}{g}{-0.8}[$25$~mm]
		\dimensioning{2}{a}{b}{-0.6}[$25$~mm]
		\dimensioning{2}{b}{c}{-0.6}[$25$~mm]
		
		\notation{1}{a}{$\circled{a}$}[align=center];
            \notation{1}{b}{$\circled{b}$}[align=center];
            \notation{1}{c}{$\circled{c}$}[align=center];
            \notation{1}{d}{$\circled{d}$}[align=center];
            \notation{1}{e}{$\circled{e}$}[align=center];
            \notation{1}{f}{$\circled{f}$}[align=center];
            \notation{1}{g}{$\circled{g}$}[align=center];
            \notation{1}{h}{$\circled{h}$}[align=center];
            \notation{1}{i}{$\circled{i}$}[align=center];
	\end{tikzpicture}
}\\
\subfloat[\label{fig:toragay1opt}]{%
      \includegraphics[width=0.3\linewidth]{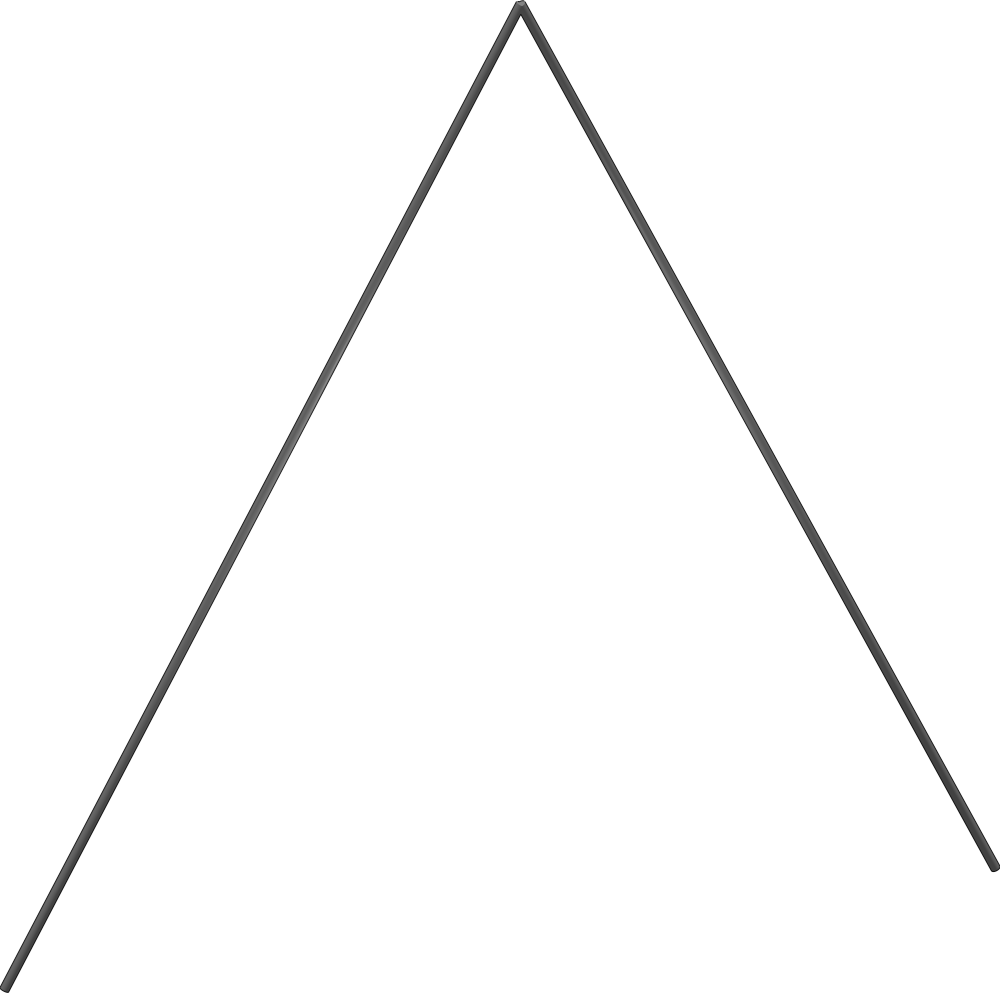}
    }%
\hfill\subfloat[\label{fig:toragay2opt}]{%
      \includegraphics[width=0.3\linewidth]{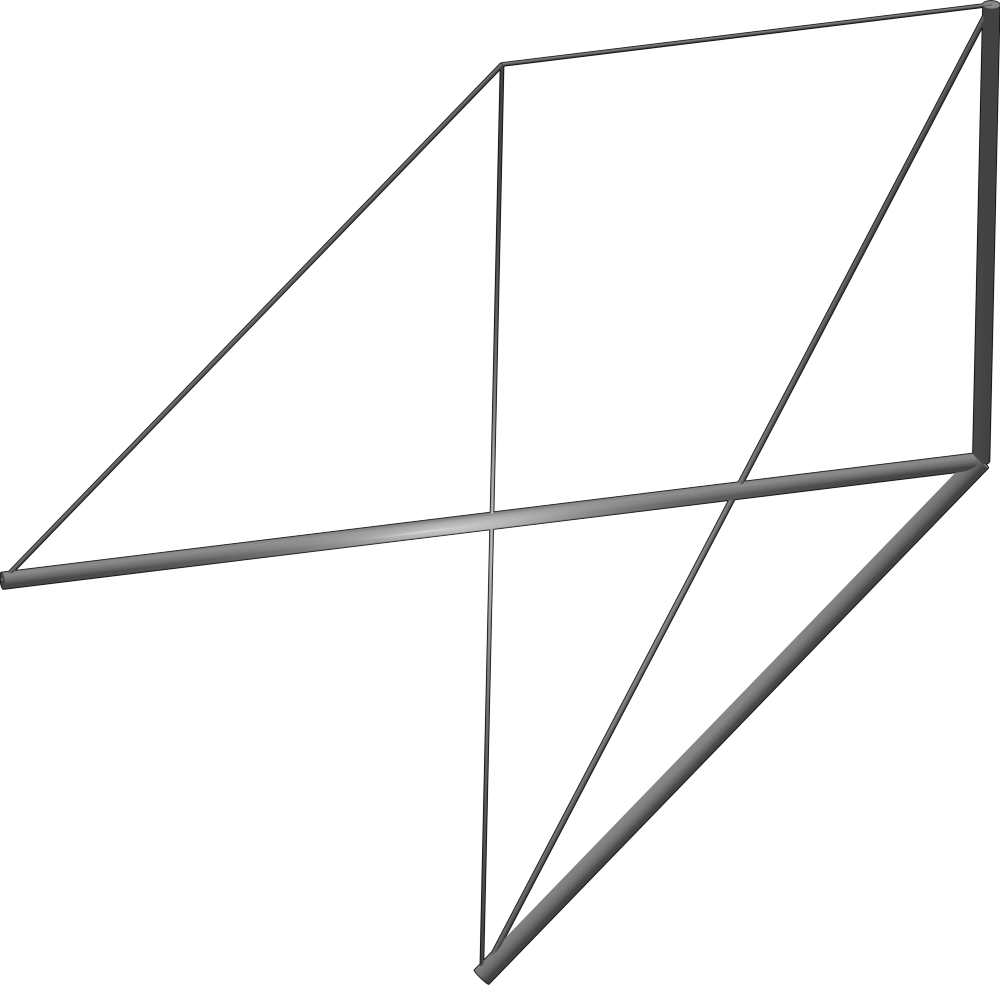}
    }%
\hfill\subfloat[\label{fig:toragay3opt}]{%
      \includegraphics[width=0.3\linewidth]{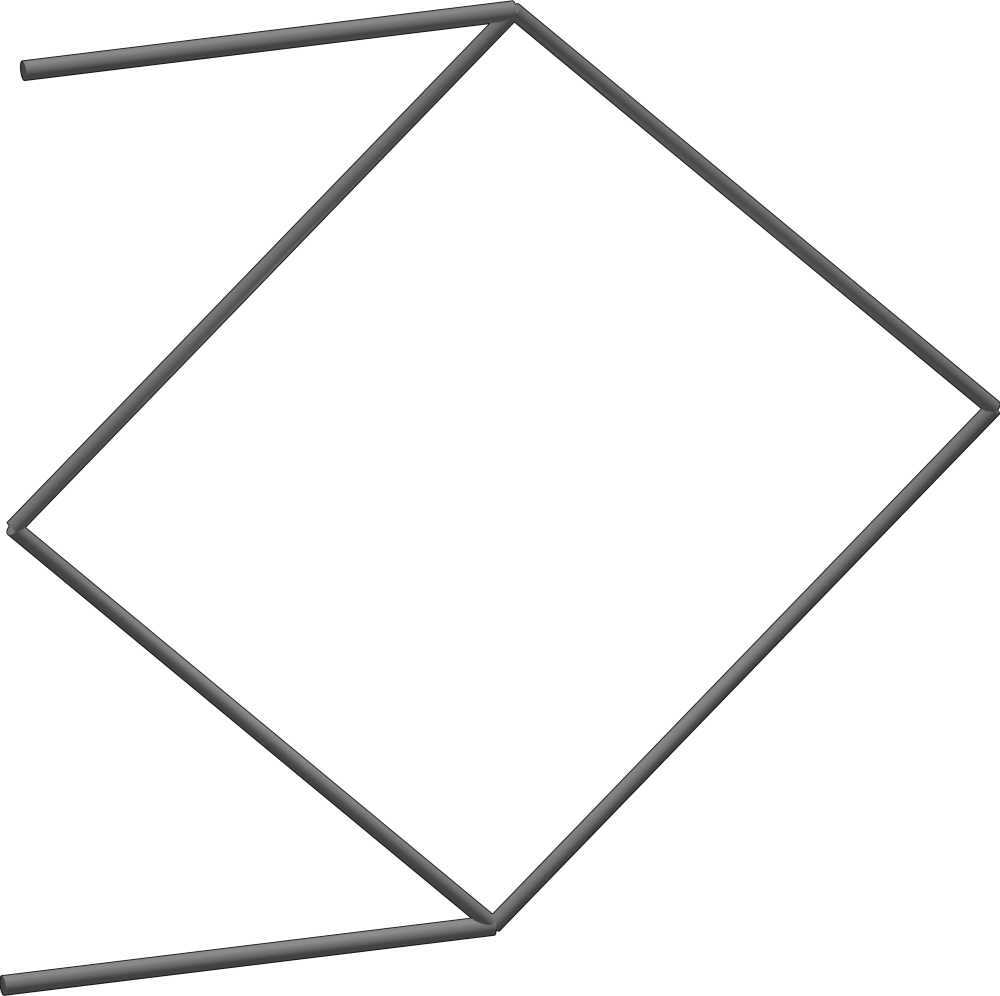}
    }
    \caption{$24$-elements problems by Toragay \textit{et al.}~\cite{toragay2022exact}: (a)--(c) discretization and boundary conditions, (d)--(f) the corresponding optimal designs.}
\end{figure}

    The setting used in \cite{toragay2022exact} considered weight minimization problems with maximum displacement constraints $u_{\max} = 0.095$~mm, upper bounds $a_{\mathrm{\max}} = \pi/4$~mm$^2$ on the cross-section areas, semicontinuous lower bounds on the cross-section areas--they are either zero or greater than $a_{\min} = \pi/25$~mm$^2$, and a constraint that prevents intersecting frame elements in a final design. 

    Since for all types of boundary conditions we have a single load only, recall Figs.~\ref{fig:toragay1bc}--c, and because the maximum displacement occurs at the point and direction of the applied load, the displacement constraints can actually be reduced into single compliance constraints. In particular, we have $\overline{c} = F u_{\max}$. Furthermore, we also adopt the upper bounds on the cross-sectional areas. On the other hand, we do not consider semicontinuous constraints as well as the prevention of virtually intersecting elements. Both of these constraints are still degree-two polynomials in the cross section variables and can thus be simply included in the formulation, but because the theory developed in \cite{tyburec2022global} does not apply to this setting, we omit them. Moreover, it turns out that in more than half of the test instances, these constraints are automatically satisfied.
    
    The numerical results of the computations appear summarized in Table~\ref{tab:toragay}, comparing the objective function and the solution time of NMT mSOS with the lowest solution time and the best objective function reported in \cite{toragay2022exact}. In addition, we also report if the design found by our technique satisfies the original (neglected) constraints. We note here, that we report the best results of  \cite{toragay2022exact}, where multiple formulations of the problem were investigated, and also used a considerably more powerful workstation.
    
    From the table it follows that for the boundary conditions in Fig.~\ref{fig:toragay1bc}, we solved the problem in $0.1$~s. For the load $25$~kN, the solution did not satisfy the semicontinuous constraints, and thus our objective function value is lower than the optimal one reported in \cite{toragay2022exact}. On the other hand, for the loads of $50$~kN and $75$~kN, we reached a global solution that satisfied all neglected constraints and this was achieved in shorter times than in \cite{toragay2022exact}. The optimal design for $F=50$~kN is shown in Fig.~\ref{fig:toragay1opt}.

    Using the boundary conditions in Fig.~\ref{fig:toragay2bc}, only the case of $25$~kN satisfied the neglected constraints and the objective function value matched the value reported in \cite{toragay2022exact}. The corresponding solution time using the NMT technique was by $40$~s longer. For the cases of $50$~kN and $75$~kN, Toragay \textit{et al.}~\cite{toragay2022exact} were unable to solve the problem globally and terminated the optimization in $5$ hours. We solved these instances in less than two minutes, yielding better objective functions, but the neglected constraints are not satisfied due to the intersecting elements; see Fig.~\ref{fig:toragay2opt} for our optimal design for $F=50$~kN.

    Finally, for the boundary conditions in Fig.~\ref{fig:toragay3bc}, we were able to improve the results presented in \cite{toragay2022exact} in two cases. In particular, for the cases of $25$~kN and $50$~kN (Fig.~\ref{fig:toragay3opt}) our designs satisfy neglected constraints and provide the value of the globally optimal objective function. Remarkably, the solution time was again less than two minutes. For the case of $75$~kN, we obtained a global solution in a similar time, but the neglected constraints were not satisfied.

     \begin{table}[h!]
           \caption{Computational results for weight minimization of $24$-elements problems by Toragay \textit{et al.} \cite{toragay2022exact} and the results obtained using NMT mSOS at relaxation order $r=1$ (setting of Fig.~\ref{fig:toragay1bc}) and $r=2$ (Figs.~\ref{fig:toragay2bc} and \ref{fig:toragay3bc}). Bold values denote globally optimality.}
 \centering   \begin{tabular}{ c  l c   c c }
    \hline
   \textbf{Load (kN)}  & \textbf{Model} & \textbf{Fig.~\ref{fig:toragay1bc}} & \textbf{Fig.~\ref{fig:toragay2bc}}  &  \textbf{Fig.~\ref{fig:toragay3bc}}   \\ \hline   
    \multirow{5}{*}{$25$}   & Toragay \textit{et al.} l.b. & $\bm{13.58}$  &  $\bm{37.71}$   &  $54.32$        \\
   
     & Toragay \textit{et al.} time (s) & $0.76$  &  $61.10$   & $18000$         \\
   
    & NMT mSOS l.b. &  $\bm{9.43}$  &  $\bm{37.71}$   & $ \bm{54.29}$     \\
   
     & NMT mSOS time (s) & $0.14$  &  $107$   &  $115$      \\

   & $\varepsilon_{rel}$ & $7.9\mathrm{e}{-6}$ & $2.2\mathrm{e}{-6}$ & $3.0\mathrm{e}{-7}$\\
 
   & C.satisfaction & no  &   yes  &  yes     \\
   
   \hline
\multirow{5}{*}{$50$}   &Toragay \textit{et al.} l.b. &  $\bm{18.86}$  & $75.51$    &  $110.36$      \\
  
     & Toragay \textit{et al.} time (s)  & $1.28$  & $18000$    & $18000$     \\
   
    & NMT mSOS l.b.&  $\bm{18.86}$&  $\bm{75.40}$   & $\bm{108.52}$      \\
  
     &NMT mSOS time (s)  &  $0.1$ &  $117$   &  $113$       \\

     & $\varepsilon_{rel}$ & $1.3\mathrm{e}{-5}$ & $3.1\mathrm{e}{-6}$& $4.3\mathrm{e}{-7}$\\
  
    &C.satisfaction & yes  &  no     &  yes       \\
   \hline
\multirow{5}{*}{$75$}   & Toragay \textit{et al.} l.b. & $\bm{28.30}$  &  $126.68$ & $164.68$     \\
  
     & Toragay \textit{et al.} time (s) &  $1.45$  & $18000$    & $18000$        \\
   
    & NMT mSOS l.b. & $\bm{28.29}$  &  $\bm{113.11}$   &  $\bm{162.80}$    \\
   
     & NMT mSOS time (s) & $0.1$  &  $ 120 $   &  $106$       \\

     & $\varepsilon_{rel}$ & $1.9\mathrm{e}{-5}$ & $1.2\mathrm{e}{-6}$ & $6.5\mathrm{e}{-7}$\\
 
   & C.satisfaction& yes   &  no     &  no    \\
   \hline
    \end{tabular}
        \label{tab:toragay}
    \end{table}

\subsection{$39$-element cantilever}

As the last illustration, we consider the problem containing $39$-elements and $16$ nodes; see Fig.~\ref{fig:larger4bc}. The nodes $\circled{a}$, $\circled{b}$, $\circled{c}$, and $\circled{d}$ are fixed, and a vertical load of magnitude $100$~kN is applied to the node $\circled{m}$. The elements are made of a linear elastic material of the Young modulus $E=109$~GPa and the dimensionless density $\rho=1$.

We assume a weight minimization problem with $\overline{c} = 4.75$ in particular. The problem is already too large to be solved using the mSOS hierarchy due to the $39$ design variables and the stiffness matrix of the size $36 \times 36$. Thus, we resort to using the NMT basis instead. Using NMT, solution of the second relaxation took more than an hour and provided the design shown in Figure \ref{fig:larger4opt} with a relative optimality gap of $5 \times 10^{-5}$, proving the approximate global optimality. When using NMT with TSP and the sparsity order one, a slightly worse design was found. However, this time the optimization took $8$~s only.

\begin{figure}[!htbp]
\subfloat[\label{fig:larger4bc}]{%
\begin{tikzpicture}
\scaling{0.0725}
\point{a}{0.000000}{0.000000}
\point{b}{0.000000}{16.666667}
\point{c}{0.000000}{33.333333}
\point{d}{0.000000}{50.000000}
\point{e}{16.666667}{0.000000}
\point{f}{16.666667}{16.666667}
\point{g}{16.666667}{33.333333}
\point{h}{16.666667}{50.000000}
\point{i}{33.333333}{0.000000}
\point{j}{33.333333}{16.666667}
\point{k}{33.333333}{33.333333}
\point{l}{33.333333}{50.000000}
\point{m}{50.000000}{0.000000}
\point{n}{50.000000}{16.666667}
\point{n2}{53.000000}{0.0}
\point{o}{50.000000}{33.333333}
\point{p}{50.000000}{50.000000}
\beam{2}{a}{e}
\beam{2}{b}{f}
\beam{2}{c}{g}
\beam{2}{d}{h}
\beam{2}{e}{i}
\beam{2}{f}{j}
\beam{2}{g}{k}
\beam{2}{h}{l}
\beam{2}{i}{m}
\beam{2}{j}{n}
\beam{2}{k}{o}
\beam{2}{l}{p}
\beam{2}{e}{f}
\beam{2}{f}{g}
\beam{2}{g}{h}
\beam{2}{i}{j}
\beam{2}{j}{k}
\beam{2}{k}{l}
\beam{2}{m}{n}
\beam{2}{n}{o}
\beam{2}{o}{p}
\beam{2}{a}{f}
\beam{2}{b}{g}
\beam{2}{c}{h}
\beam{2}{e}{j}
\beam{2}{f}{k}
\beam{2}{g}{l}
\beam{2}{i}{n}
\beam{2}{j}{o}
\beam{2}{k}{p}
\beam{2}{b}{e}
\beam{2}{c}{f}
\beam{2}{d}{g}
\beam{2}{f}{i}
\beam{2}{g}{j}
\beam{2}{h}{k}
\beam{2}{j}{m}
\beam{2}{k}{n}
\beam{2}{l}{o}
\support{3}{a}[270];
\support{3}{b}[270];
\support{3}{c}[270];
\support{3}{d}[270];
\load{1}{n2}[90][0.75][0.15];
\notation{1}{n2}{$100$~kN}[above right=2mm];
\dimensioning{1}{a}{e}{-0.8}[$\frac{50}{3}$~mm];
\dimensioning{1}{e}{i}{-0.8}[$\frac{50}{3}$~mm];
\dimensioning{1}{i}{m}{-0.8}[$\frac{50}{3}$~mm];
\dimensioning{2}{a}{b}{-0.6}[$\frac{50}{3}$~mm];
\dimensioning{2}{b}{c}{-0.6}[$\frac{50}{3}$~mm];
\dimensioning{2}{c}{d}{-0.6}[$\frac{50}{3}$~mm];
\notation{1}{a}{$\circled{a}$}[align=center];
\notation{1}{b}{$\circled{b}$}[align=center];
\notation{1}{c}{$\circled{c}$}[align=center];
\notation{1}{d}{$\circled{d}$}[align=center];
\notation{1}{e}{$\circled{e}$}[align=center];
\notation{1}{f}{$\circled{f}$}[align=center];
\notation{1}{g}{$\circled{g}$}[align=center];
\notation{1}{h}{$\circled{h}$}[align=center];
\notation{1}{i}{$\circled{i}$}[align=center];
\notation{1}{j}{$\circled{j}$}[align=center];
\notation{1}{k}{$\circled{k}$}[align=center];
\notation{1}{l}{$\circled{l}$}[align=center];
\notation{1}{m}{$\circled{m}$}[align=center];
\notation{1}{n}{$\circled{n}$}[align=center];
\notation{1}{o}{$\circled{o}$}[align=center];
\notation{1}{p}{$\circled{p}$}[align=center];
\end{tikzpicture}}
\hfill\subfloat[\label{fig:larger4opt}]{%
      \includegraphics[width=0.4\linewidth]{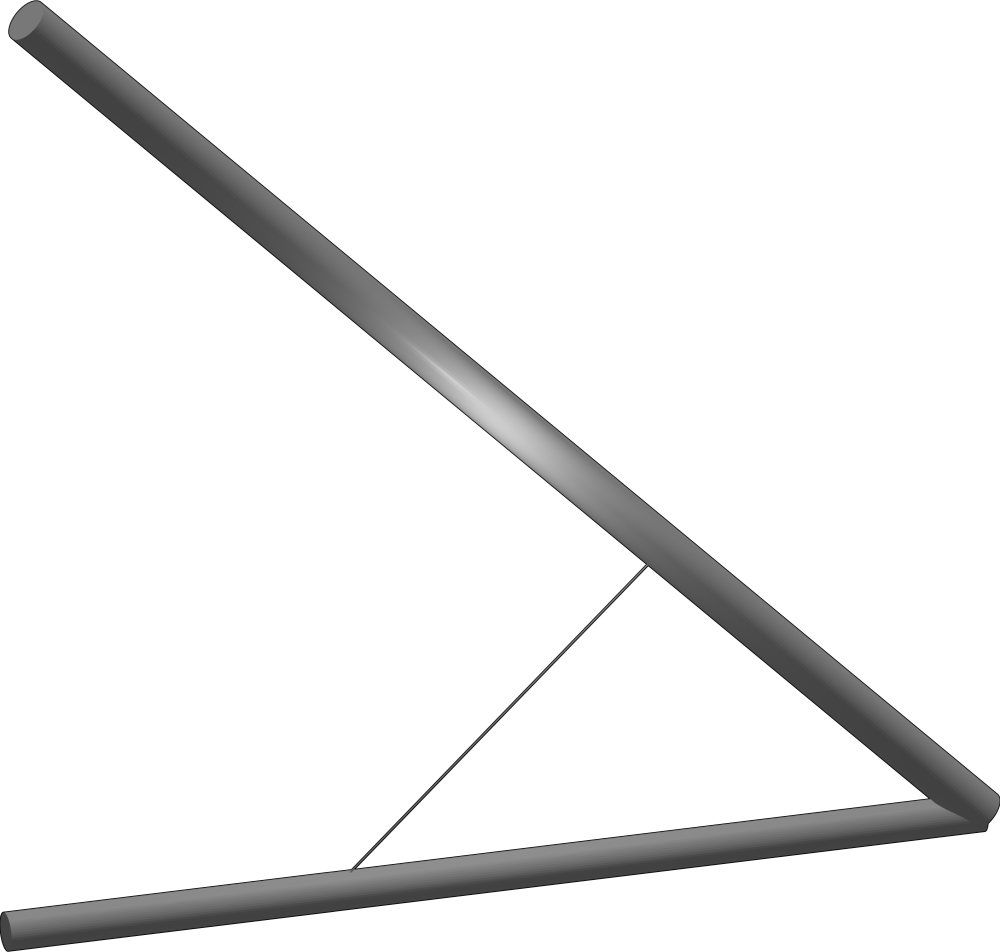}
    }
\caption{$39$-elements cantilever: (a) discretization and boundary conditions, and (b) optimal design.}
\end{figure}

\begin{table}[h!]
 \caption{Computational results using NMT basis for weight minimization of $39$-elements cantilever and the relaxation order $r=2$.}
  \begin{tabular}{l l c c c  c }
    \hline 
  \textbf{Method}    &$ \bm{(n_c ,s)} $& \textbf{nvar}&\textbf{l.b.} & \textbf{time} & $\bm{\varepsilon_}{rel}$\\
        \hline

         NMT mSOS  &  $ (39,40), (1,79), (1,1480)$ & $ 41158 $ &  $433.8 $   & $25815$ & $ 5\mathrm{e}{-05}  $\\ 
         NMT TSP $k=1$ & $ (1599,2),(39,74)$ & $3120$ &  $433.7 $   & $8$ & $ 7\mathrm{e}{-04}  $\\
   
       \hline 
    \end{tabular}

  \label{tab:39_weight_r2}
    \end{table}

\section{Conclusion and perspectives}
\label{sec:conclusion-perspective}

In this paper, we have applied the term-sparse moment SOS (TSP) technique introduced by Magron and Wang in \cite{magron2023sparse} to the problems of compliance and weight optimization of frame structures. We first generalized the TSP to handle polynomial matrix inequalities and then proposed a reduced monomial basis (NMT basis), which is adapted to the polynomial nature of the problems. This monomial basis was used to build the dense mSOS and for the TSP with minimal chordal extension. We showed through numerical experiments that for frame optimization problems the NMT basis reached a global minimum in computationally significantly shorter times than the dense mSOS, and even than the TSP with the standard monomial basis. Therefore, we solved, on the one hand, bigger instances of frame structures and, on the other hand, smaller problems that required high relaxation orders to converge. However, we do not offer theoretical guarantees for convergence to a global minimum of the mSOS using the NMT basis. This will be a direction to explore in future work. Further, the use of the term sparsity in frame structure problems has limitations when the stiffness matrix $\bm{K}$ becomes large. To solve this, it is necessary to exploit the sparsity that occurs in the entries of the stiffness matrix $\bm{K}$ and combine it with TSP. As shown in \cite{kovcvara2021decomposition}, the use of arrow matrix decomposition is especially useful for topology optimization problems. This method can be adapted in the case of polynomial matrix inequalities and will also be considered in future work. 

\appendix
\section{Proof of Proposition \ref{prop_maximal_extension}}
\label{proof-prop}
To provide a proof of Proposition \ref{prop_maximal_extension}, we need first to prove some preliminary results. We start by providing two definitions. 

\begin{definition}
    \label{definition_lengthmonomial}
Let $ \bm{x} \in \mathbb{R}^n$ and $\ell$ be an integer such that $1 \leq \ell \leq n$. The monomial $ \bm{x^\alpha}$ is of length $\ell$  if it can be written as $$ \bm{x^\alpha}=x_{\sigma_1}^{\alpha_{\sigma_1}}\dots x_{\sigma_\ell}^{\alpha_{\sigma_\ell}} \text{ where } \lbrace \sigma_1, \ldots, \sigma_\ell  \rbrace \subseteq \lbrace 1, \ldots, n \rbrace \text{ and } \alpha_{\sigma_i} \neq 0, \forall i \in \lbrace 1, \ldots, \ell\rbrace.$$
In particular if $\ell=1$, then $ \bm{x^\alpha}$ is an univariate monomial.

\end{definition}

\begin{definition}
    Let $\bm{\alpha} \in \mathbb{N}^n$. We call the sign type of $\bm{\alpha}$ the vector $\bm{v}=\bm{\alpha} \text{ mod }2= $ \\$(\alpha_1 \text{ mod } 2, \ldots, \alpha_n \text{ mod } 2  ) \in \lbrace 0,1 \rbrace^n$. In what follows, we denote the sign type of a monomial $\bm{x^\alpha}$ as the sign type of $\bm{\alpha}$. The sign type is of even degree if $\bm{v}=\bm{0}_{\mathbb{R}^n} $ and of odd sign type otherwise. 
\end{definition}
Let us note that for a given integer $r \geq 1$, every monomial $\bm{x^\alpha}\in \bm{b}_r(\bm{x}^2)$ is of the even sign type. Conversely, for a monomial $\bm{x^\alpha}$ of even sign type, if $\text{deg}(\bm{x^\alpha}) \leq 2r$, then $\bm{x^\alpha}\in \bm{b}_r(\bm{x}^2)$. 

Now, let $r$ be a fixed relaxation order and $\mathcal{S}^{(0)}=\mathcal{S}_c \cup \bm{b}_r(\bm{x}^2)$ where $\mathcal{S}_c $ is defined in \eqref{support_compliance_K3=0} or in \eqref{support_compliance_K3!=0}. Let $\mathcal{G}_{r,j}^{(k)}$ with $V(\mathcal{G}_{r,j}^{(k)})=\bm{b}_r(\bm{x})$ be the $k$-th graph obtained after performing TSP with maximal chordal extension. We recall that $\lbrace \bm{x^\alpha},\bm{x^\beta} \rbrace \in E(\mathcal{G}_{r,j}^{(k)})$ then $\bm{x}^{\bm{\alpha}+\bm{\beta}} \in \mathcal{S}^{(k)} $, and that at sparsity order $k=0$, we have  $ \mathcal{G}_{r,0}^{(0)}= \mathcal{G}^{\text{tsp}}$ and $ \mathcal{G}_{r,j}^{(0)}= \varnothing$ for $j \geq 1$. The following result allows us to determine the type of monomials that form a connected component in the TSP graph, according to their length and their sign type.

\begin{lemma}
   \label{lem_maximal_extension} 
Consider the problem \eqref{compliance_problem_sdp_scalled} and a fixed relaxation order $r \geq 2$. At the sparsity order $k=0$, there exists a connected component $\mathcal{C}_{r,0}^{(0)} \varsubsetneq \mathcal{G}_{r,0}^{(0)}=\mathcal{G}^{\text{tsp}}$ that contains \begin{itemize}
    \item  all the monomials of length $1$;  
    \item  monomials $\bm{x^\alpha}$ of length $\ell \geq 2$ and of even sign type;
    \item  monomials $\bm{x^\alpha}$ of length $\ell \geq 2$, of odd sign type, and for which there exists a univariate monomial $x_o^{\alpha_o}$ of degree $\alpha_o \geq  1$ such that $\bm{x^\alpha}x_o^{\alpha_o}$ is of even sign type.
\end{itemize} 
\end{lemma}
Lemma \ref{lem_maximal_extension} is illustrated in Figure \ref{example_lemma_4}.

\begin{figure}[!h]

\centering \begin{tikzpicture}[node distance={12mm},rotate=-20, thick, main/.style = {draw, circle, radius=0.1mm}, auto]
\node[main]  at (360/16* 1:2.4cm ) (1) {\adjustbox{max width=3mm}{$1$}}; 
\node[main] at (360/16* 2:3cm ) (2) {\adjustbox{max width=3mm}{ $x_1$}}; 
\node[main] at (360/16* 3:3.2cm ) (3) {\adjustbox{max width=3mm}{ $x_2$}}; 
\node[main] at (360/16* 4:3cm ) (4) {\adjustbox{max width=3mm}{$x_3$}}; 
\node[main] at (360/16* 5:2.4cm ) (5) {\adjustbox{max width=2.5mm}{$x_3^2$}}; 
\node[main] at (360/16* 6:2.6cm )(6) {\adjustbox{max width=2.5mm}{$x_2^2$}}; 
\node[main] at (360/16* 7:2.4cm )(7) {\adjustbox{max width=2.5mm}{$x_1^2$}}; 

\node[main] at (360/16* 14:3cm ) (8) {\adjustbox{max width=2.5mm}{$x_3^3$}}; 
\node[main] at (360/16* 15:3.2cm )(9) {\adjustbox{max width=2.5mm}{$x_2^3$}}; 
\node[main] at (360/16* 16:3cm )(10) {\adjustbox{max width=2.5mm}{$x_1^3$}}; 

\node[main] at (360/16* 8:2.4cm )(11) {\adjustbox{max width=4mm}{$x_1^2x_2$}}; 
\node[main] at (360/16* 9:2.4cm )(12) {\adjustbox{max width=4mm}{$x_1^2x_3$}}; 
\node[main]  at (360/16* 10:2.4cm ) (13) {\adjustbox{max width=4mm}{$x_2^2x_3$}}; 
\node[main] at (360/16* 11:2.4cm ) (14) {\adjustbox{max width=4mm}{$x_1x_2^2$}}; 
\node[main] at (360/16* 12:2.4cm )(15) {\adjustbox{max width=4mm}{$x_1x_3^2$}}; 
\node[main] at (360/16* 13:2.6cm )(16) {\adjustbox{max width=4mm}{$x_2x_3^2$}};

 \node[main]  [  right of=10] (17) {\adjustbox{max width=4mm}{$x_1x_3$}}; 
\node[main] [  above of=17] (18) {\adjustbox{max width=4mm}{$x_1x_2$}}; 
\node[main] [  above of=18]  (19) {\adjustbox{max width=4mm}{$x_2x_3$}}; 
 \node[main]  [  below of=17] (20) {\adjustbox{max width=4mm}{$x_1x_2x_3$}}; 

\draw (1) -- (2);
\draw (1) -- (3);
\draw (1) -- (4);
\draw (1) -- (5);
\draw (1) -- (6);
\draw (1) -- (7);
\draw (1) -- (8);
\draw (1) -- (9);
\draw (1) -- (10);

\draw (11) -- (3);
\draw (12) -- (4);
\draw (13) -- (4);
\draw (14) -- (2);
\draw (15) -- (2);
\draw (16) -- (3);
\draw (5) -- (6);
\draw (5) -- (7);
\draw (6) -- (7);
\draw (8) -- (9);
\draw (8) -- (10);
\draw (9) -- (10);

\draw (11) -- (9);
\draw (12) -- (8);
\draw (13) -- (8);
\draw (14) -- (10);
\draw (15) -- (10);
\draw (16) -- (9);
\end{tikzpicture}  

\caption{The TSP graph at $k=0$ of the illustrative problem \eqref{POP-example} with a relaxation degree $r=3$}
 \label{example_lemma_4}
\end{figure}
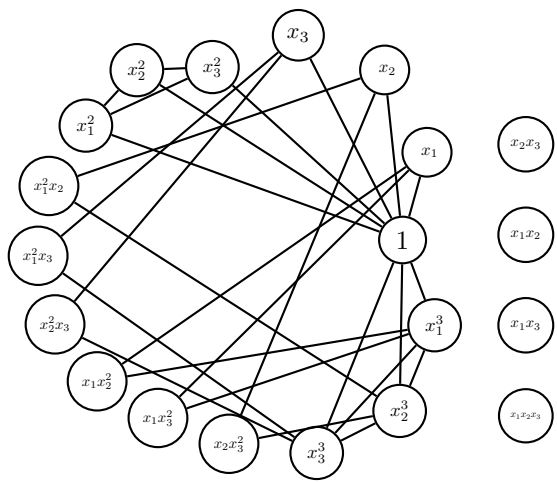

\begin{proof}
 For all $ i \in \lbrace 1, \ldots, n \rbrace$, we have $\lbrace 1,x_i\rbrace \in E(\mathcal{G}_{r,0}^{(0)})$ and $\lbrace 1,x_i^2\rbrace \in E(\mathcal{G}_{r,0}^{(0)})$. Now, in the case of a degree $ \alpha_i \geq 3 $, we have two possibilities:
    \begin{itemize}
        \item if $\alpha_i$ is even, then $\lbrace 1,x_i^{\alpha_i} \rbrace \in E(\mathcal{G}_{r,0}^{(0)})$;
        \item if $\alpha_i$ is odd, then $\lbrace x_i,x_i^{\alpha_i} \rbrace \in E(\mathcal{G}_{r,0}^{(0)})$.
    \end{itemize}
 Therefore, there exists a connected component $\mathcal{C}_{r,0}^{(0)} \subseteq \mathcal{G}_{r,0}^{(0)}$ such that $x_i^{\alpha_i} \in V(\mathcal{C}_{r,0}^{(0)}) $ for all $i \in \lbrace 1, \ldots, n \rbrace$ and $1 \leq \alpha_i \leq r$, i.e., $ \mathcal{C}_{r,0}^{(0)}$ contains all univariate monomials regardless of their degree. If $\bm{x^\alpha}$ is of length $\ell \geq 2$ and of even sign type, then it belongs to $\mathcal{C}_{r,0}^{(0)}$ since $\bm{x^\alpha} \in \bm{b}_r(\bm{x})$. Now, if $\bm{x^\alpha}$ is of length $\ell \geq 2$, of odd sign type and where $\bm{x^\alpha}x_o^{\alpha_o}$ is of the even sign type for some univariate monomial $x_o^{\alpha_o}$, then it belongs to $\mathcal{C}_{r,0}^{(0)}$ since $\bm{x^\alpha}x_o^{\alpha_o} \in \bm{b}_r(\bm{x}^2)$. The inclusion $\mathcal{C}_{r,0}^{(0)} \varsubsetneq \mathcal{G}_{r,0}^{(0)}$ is strict since the graph $\mathcal{G}_{r,0}^{(0)}$ is not connected. Indeed, consider for example the monomial $x_1x_2 \in \bm{b}_r(\bm{x})$, there exists no monomial $x^\beta \neq x_1x_2 \in V(\mathcal{G}_{r,0}^{(0)})=\bm{b}_r(\bm{x})$ such that $x_1x_2x^\beta \in \mathcal{S}^{(0)}$. 
\end{proof} 

According to Lemma \ref{lem_maximal_extension}, every monomial $\bm{x}^{\bm{\alpha}}= x_{\sigma_1}^{\alpha_{\sigma_1}}\ldots x_{\sigma_\ell}^{\alpha_{\sigma_\ell}}\in V(\mathcal{G}_{r,0}^{(0)}) \setminus  V(\mathcal{C}_{r,0}^{(0)})$ is of length $\ell \geq 2$, of odd sign type and has at least two odd exponents, i.e., $\exists \sigma_{o_1}, \sigma_{o_2} \in \lbrace \sigma_1, \ldots, \sigma_\ell \rbrace$  such that $\alpha_{\sigma_{o_1}}$ and $\alpha_{\sigma_{o_2}}$ are odd. For example, in Figure \ref{example_lemma_4}, it corresponds to the monomials $ x_1x_2, x_1x_3, x_2x_3$ and $x_1x_2x_3$. 

The next result is dedicated to monomials of this form and allows us to construct two other monomials that have specific properties. We further show how are these monomials linked to the connected component $\mathcal{C}_{r,0}^{(k)}$ after the extension of the support in the sparsity order $k=2$. This is crucial for deriving the proof of Proposition \ref{prop_maximal_extension}.

\begin{lemma}
\label{weard_lemma}
Let $\ell$ be an integer such that $ 2 \leq \ell \leq n$ and $\bm{x^\alpha}=x_{\sigma_1}^{\alpha_{\sigma_1}}\ldots x_{\sigma_\ell}^{\alpha_{\sigma_\ell}}$ be a monomial of length $\ell \geq 2$, of degree $d$ and of odd sign type. Suppose that $\exists \sigma_{o_1}, \sigma_{o_2} \in \lbrace \sigma_1,\ldots, \sigma_\ell \rbrace$ such that $ \alpha_{\sigma_{o_1}}$ and $ \alpha_{\sigma_{o_2}}$ are odd. Then there exists an integer $s$ satisfying
\begin{equation}
\label{weird_inequality}
    \max\Big\lbrace\ell-(\alpha_{\sigma_1}+\ldots +\alpha_{\sigma_s})-1,1\Big\rbrace \leq s \leq \min\Big\lbrace\alpha_{\sigma_{s+1}}+\ldots+\alpha_{\sigma_\ell},\ell-1\Big\rbrace,
\end{equation}
and two monomials $\bm{x^a}$ and $\bm{x^{a'}}$ such that
\begin{itemize}
    \item $\bm{x^a}$ is of length $s $, of degree $\text{deg}(\bm{x^a})\leq d$ and the monomial $\bm{x^a}x_{\sigma_{o_1}}$ is of even sign type; 
    \item $\bm{x^{a'}}$ is of length $\ell -s  $, of degree $\text{deg}(\bm{x^{a'}}) \leq d$ and the monomial $\bm{x^{a'}}x_{\sigma_{o_2}}$ is of even sign type.
\end{itemize}
\end{lemma}

\begin{proof}

 First, let us prove that we can always find an integer $s$ that satisfies inequality \eqref{weird_inequality}. Let $\underline{b}=\max\Big\lbrace\ell-(\alpha_{\sigma_{1}}+\ldots +\alpha_{\sigma_{s}})-1,1\Big\rbrace $ and $\overline{b}=\min\Big\lbrace\alpha_{\sigma_{s+1}}+\ldots+\alpha_{\sigma_{\ell}},\ell-1\Big\rbrace$. It suffices to prove that $\overline{b}-\underline{b} \geq 0$.
    \begin{itemize}
    
        \item If $\underline{b}=\ell-(\alpha_{\sigma_{1}}+\ldots+\alpha_{\sigma_{s}})-1,$
        \begin{itemize}
            \item if $ \overline{b}=\ell-1$, then $\overline{b}-\underline{b}=\alpha_{\sigma_{1}}+\ldots+\alpha_{\sigma_{s}} \geq 0$;
            \item if $ \overline{b}=\alpha_{\sigma_{s+1}}+\ldots+\alpha_{\sigma_{\ell}}$, then $\overline{b}-\underline{b}=d-\ell +1 \geq 0$, (it is straightforward that for any monomial of length $\ell$ and of degree $d$, we have $\ell \leq d$).

        \end{itemize}
        \item If $\underline{b}=1$, 
        \begin{itemize}
            \item if $\overline{b}=\ell-1$ then $\overline{b}-\underline{b}=\ell -2 \geq 0$, since we work with monomials of length $\ell \geq 2$;
            \item if $\overline{b}=\alpha_{\sigma_{s+1}}+\ldots+\alpha_{\sigma_{\ell}}$, then $\overline{b}-\underline{b}=\alpha_{\sigma_{s+1}}+\ldots+\alpha_{\sigma_{\ell}}-1 \geq 0$.
        \end{itemize}
    \end{itemize}

Furthermore, using the integer $s$, we construct two monomials $\bm{x^a}$ and $\bm{x^{a'}}$ from $\bm{x^\alpha} $ satisfying the claims of Lemma \ref{weard_lemma}. We start by reordering the exponent index set $\lbrace \alpha_1, \ldots, \alpha_\ell \rbrace$ so that $ \alpha_{\sigma_{o_{1}}}=\alpha_{\sigma_{s}}$ and $ \alpha_{\sigma_{o_{2}}}=\alpha_{\sigma_{s+1}}$, i.e., we consider the ordering \\ $\lbrace \sigma_1,\ldots,\sigma_{s-1},\sigma_{o_1},\sigma_{o_2},\sigma_{s+2},\ldots, \sigma_{\ell} \rbrace$\footnote{It is always possible since the letters of any given monomial are commutative. For example, the monomial $x_2x_3^2x_4^2x_5$ with $s=2$ is reordered as $ x_3^2x_2x_5x_4^2$ }. We define the two monomials $\bm{x^a}$ and $\bm{x^{a'}}$ from $\bm{x^\alpha} $ as 
 \begin{equation}
    \left\lbrace\begin{aligned}
    \bm{x^a}&=x_{\sigma_1}^{a_{\sigma_1}}\ldots x_{\sigma_{s-1}}^{a_{\sigma_{s-1}}}x_{\sigma_{o_1}}^{\alpha_{\sigma_{o_1}}}, \\
    \bm{x^{a'}}&=x_{\sigma_{o_2}}^{\alpha_{\sigma_{o_2}}}x_{\sigma_{s+2}}^{a_{\sigma_{s+2}}}\ldots x_{\sigma_{\ell}}^{a_{\sigma_{\ell}}},
    \end{aligned}\right.
    \label{weird-construction}
\end{equation} 
where  $  \forall i \in \lbrace \sigma_1,\ldots, \sigma_{s-1} \rbrace \cup \lbrace \sigma_{s+2}, \ldots, \sigma_{\ell} \rbrace,$
\begin{equation}
    a_{i}=\left\lbrace \begin{aligned}
        & \alpha_{i}+1, \text{ if } \alpha_{i} \text{ is odd}; \\
        &\alpha_{i}, \text{ if } \alpha_{i} \text{ is even}.
    \end{aligned}\right.
    \label{def-ai}
\end{equation}
With this construction, we have 
    \begin{align*}
            \text{deg}(\bm{x^a})&=\sum_{i=1}^{s-1}a_{\sigma_{i}}+\alpha_{\sigma_s}\\
            &\leq \sum_{i=1}^{s-1}\alpha_{\sigma_{i}}+s-1+\alpha_{\sigma_s} \\
            &\leq \sum_{i=1}^{s}\alpha_{\sigma_{i}}+ \sum_{i=s+1}^{\ell}\alpha_{\sigma_{i}}~\text{using the right part of the condition \eqref{weird_inequality}} \\
            &=d.
    \end{align*}
Analogously, we prove that  $\text{deg}(\bm{x^{a'}}) \leq d$.

 Finally, by construction, the monomials $\bm{x^a}x_{\sigma_{o_1}}$ and $\bm{x^{a'}}x_{\sigma_{o_2}}$ are of even sign type. 
\end{proof}
\begin{example}
\begin{enumerate}
    \item  Let $\bm{x} \in \mathbb{R}^6$ and consider the monomial $\bm{x^\alpha}=x_1x_2^2x_4x_5^3x_6$. We have $n=6$, $\ell=5$ and $d=8$. For $s=2$, we can choose $\bm{x^a}=x_2^2x_1$ and $\bm{x^{a'}}=x_5^4x_6^2x_4$. 
    \item Let $\bm{x} \in \mathbb{R}^3$ and consider the monomial $x_1x_2x_3$ in Figure \ref{example_lemma_4}. We have $\ell=3 $ and for $s=2$, we can choose $\bm{x^a}=x_1^2x_2$ and $\bm{x^{a'}}=x_3$.
\end{enumerate}
   \label{example-lemma2-monomials-length}
\end{example}

\begin{proof}\textit{of Proposition \ref{prop_maximal_extension}   }
It suffices to prove the result for $j=0$. Let $\mathcal{S}^{(k)}$ be the extended support defined in \eqref{extended_support} at the sparsity order $k$. We recall that the graph $\mathcal{F}^{(k)}_{r,j} \subseteq \mathcal{G}^{(k)}_{r,j}$ is obtained from $\mathcal{G}^{(k-1)}_{r,j}$ by adding edges after the support extension operation, and that if $\bm{x}^{\bm{\alpha}+\bm{\beta}} \in \mathcal{S}^{(k)}$ then $ \lbrace \bm{x}^{\bm{\alpha}},  \bm{x}^{\bm{\beta}} \rbrace \in \mathcal{F}_{r,j}^{(k+1)}$.

For $k=0$, according to Lemma \ref{lem_maximal_extension}, there exists a connected component $ \mathcal{C}_{r,0}^{(0)} \varsubsetneq \mathcal{G}_{r,0}^{(0)} $ containing all univariate monomials $x_i^{\alpha_i}$, monomials of length $\ell \geq 2$ of even sign type, and monomials $\bm{x^\alpha}$ of length $\ell \geq 2$ such that $\bm{x^\alpha}x_o$ is of odd degree for some univariate monomial $x_o$. 

At the sparsity order $k=1$ and after the maximal chordal extension operation, the connected component $ \mathcal{C}_{r,0}^{(0)}$ becomes a complete subgraph $ \mathcal{C}_{r,0}^{(1)} \varsubsetneq \mathcal{G}_{r,0}^{(1)}$ 
such that for all $i, k \in \lbrace 1, \ldots, n \rbrace$ and for any degrees $1 \leq \alpha_i, \alpha_k \leq r$, we have $\lbrace x_i^{\alpha_i}, x_k^{\alpha_k}\rbrace \in E(\mathcal{C}_{r,0}^{(1)})$ so that $ x_i^{\alpha_i}x_k^{\alpha_k} \in \mathcal{S}^{(1)}$, i.e., $\mathcal{S}^{(1)} $ contains all monomials of length $2$.

Now at sparsity order $k=2$ and right after the support extension operation, we consider the graph $\mathcal{F}_{r,0}^{(2)}$ and let $\mathcal{C}_{r,0}^{(2)} \subseteq \mathcal{F}_{r,0}^{(2)} $ be the largest connected component in $\mathcal{F}_{r,0}^{(2)}$. Further, we show that $\mathcal{F}_{r,0}^{(2)}$ is connected, i.e.  $\mathcal{F}_{r,0}^{(2)}=\mathcal{C}_{r,0}^{(2)}$. It suffice to show that for any monomial $\bm{x}^{\bm{\alpha}} \in V(\mathcal{F}_{r,0}^{(2)})$ of the form described in Lemma \ref{weard_lemma}, there exists a monomial $\bm{x^\beta} \in V(\mathcal{C}_{r,0}^{(2)})$ such that 
$\lbrace \bm{x^\alpha}, \bm{x^\beta}\rbrace \in E(\mathcal{F}_{r,0}^{(2)})$, so that $\bm{x}^{\bm{\alpha}} \in V(\mathcal{C}_{r,0}^{(2)})$. This claim is true for all monomials of length $\ell = 2$ since $\mathcal{S}^{(1)}$ contains all monomials of length $2$. By induction on monomials length, suppose that it is true for all the monomials of length at most $\ell$ and let $ \bm{x^\alpha} \in V(\mathcal{F}_{r,0}^{(2)}) $ be a monomial of length at most $\ell+1$. Using Lemma \ref{weard_lemma}, there exists an integer $s$ satisfying \eqref{weird_inequality} and two monomials $ \bm{x^a}$ and $ \bm{x^{a'}}$ defined by \eqref{weird-construction}, of lengths $s \leq \ell$ and $\ell-s \leq \ell$ respectively, and such that $ \bm{x^a}x_{s}$ and $ \bm{x^{a'}}x_{s+1}$ are of even sign type, for some univariate monomials $ x_{s}$ and $ x_{s+1}$. According to Lemma \ref{lem_maximal_extension}, we have $ \bm{x^a}, \bm{x^{a'}} \in V(\mathcal{C}_{r,0}^{(0)}) $ and, therefore, $\bm{x}^{\bm{a}+\bm{a}'} \in \mathcal{S}^{(1)}$. Furthermore, we can construct the monomial $\bm{x}^{\bm{\beta}}$ from $\bm{x}^{\bm{\alpha}}$, $\bm{x}^{\bm{a}}$ and $\bm{x}^{\bm{a'}}$ by  \begin{equation*}
    x_i^{\beta_i}=\left\lbrace \begin{aligned}
  &  x_i^{a_i-\alpha_i} \text{ for } i \in \lbrace \sigma_1,\ldots, \sigma_s \rbrace, \\
&    x_i^{a_i'-\alpha_i} \text{ for } i \in \lbrace \sigma_{s+1},\ldots, \sigma_\ell\rbrace.
\end{aligned} \right.
\end{equation*}  
Since that $a_{\sigma_{s}}=\alpha_{\sigma_{s}}$ and $a'_{\sigma_{s+1}}=\alpha_{\sigma_{s+1}}$ by the definition of $\bm{a}$ and $\bm{a'}$ in \eqref{def-ai}, the monomial $\bm{x^\beta}$ is of length at most $\ell-1$. By the induction hypothesis, we have $\bm{x^\beta} \in V(\mathcal{C}_{r,0}^{(2)})$. Therefore, we have $ \bm{x}^{\bm{\alpha}+\bm{\beta}}=\bm{x}^{\bm{a}+\bm{a}'} \in \mathcal{S}^{(1)}$ so that $\lbrace \bm{x^\alpha}, \bm{x^\beta}\rbrace \in E(\mathcal{F}_{r,0}^{(2)})$. Finally, after the maximal chordal extension, the graph $\mathcal{G}_{r,0}^{(2)}=\overline{\mathcal{F}_{r,0}^{(2)}}$ is complete since $\mathcal{F}_{r,0}^{(2)}$ is connected.  

\end{proof}

\noindent\textbf{Acknowledgements}
We would like to thank Prof. Ing. Jan Zeman for helpful suggestions on the initial versions of this manuscript.
\textbf{Funding} The authors acknowledge the support of the Czech Science foundation through project No. 22-15524S and co-funded by the European Union under the project ROBOPROX -- Robotics and Advanced Industrial Production (reg. no. CZ.02.01.01/00/22\_008/0004590).\\ \\
\textbf{Authors contributions}
Conceptualization: M.H.,  M.T.; Investigation: M.H.,  M.T.;  Methodology: M.H., M.T.;  Software: M.H., M.T.;  Writing (original draft preparation): M.H.; Writing (review and editing): M.T.,  M.K.;  Funding acquisition: M.K.; Supervision: M.K. All authors read and approved the final manuscript.  \\ \\
\textbf{Availability of data and materials} Source codes and input files to reproduce the computations are available at https://gitlab.com/tyburec/pof-dyna.
\section*{Declarations}
\textbf{Ethics approval and consent to participate} This work does not contain any studies with human participants or animals performed by any of the authors.\\ \\
\textbf{Consent for publication} Not applicable. \\ \\
\textbf{Competing interests} The authors declare that they have no competing interests.

\printbibliography
\end{document}